\documentclass[draft,reqno,10pt, centertags]{amsart}
\usepackage{amsmath,amsthm,amscd,amssymb}
\usepackage{upref}
\usepackage{latexsym}
\usepackage{lscape}
\usepackage{graphicx}
\usepackage{multirow}
\usepackage{mathrsfs} 

\makeatletter
\def\theequation{\@arabic\c@equation}

\newcommand{\gaD}{\gamma_{{}_D}}
\newcommand{\gr}{{\text{graph}}}

\newcommand{\gaN}{\gamma_{{}_N}}
\newcommand{\tN}{\tau_{{}_N}}
\newcommand{\Om}{\Omega}
\newcommand{\dOm}{{\partial\Omega}}

\newcommand{\Sp}{\operatorname{Spec}}
\newcommand{\e}{\hbox{\rm e}}

\newcommand{\bbM}{{\mathbb{M}}}
\newcommand{\bbN}{{\mathbb{N}}}
\newcommand{\bbR}{{\mathbb{R}}}
\newcommand{\R}{{\mathbb{R}}}

\newcommand{\bbZ}{{\mathbb{Z}}}

\newcommand{\C}{{\mathbb{C}}}

\newcommand{\bbK}{{\mathbb{K}}}

\newcommand{\cA}{{\mathcal A}}
\newcommand{\cB}{{\mathcal B}}

\newcommand{\cD}{{\mathcal D}}

\newcommand{\cF}{{\mathcal F}}
\newcommand{\cG}{{\mathcal G}}
\newcommand{\cH}{{\mathcal H}}
\newcommand{\cI}{{\mathcal I}}

\newcommand{\cK}{{\mathcal K}}
\newcommand{\cL}{{\mathcal L}}

\newcommand{\cQ}{{\mathcal Q}}

\newcommand{\cS}{{\mathcal S}}
\newcommand{\cT}{{\mathcal T}}

\newcommand{\cV}{{\mathcal V}}
\newcommand{\cW}{{\mathcal W}}
\newcommand{\cX}{{\mathcal X}}
\newcommand{\cY}{{\mathcal Y}}
\newcommand{\cZ}{{\mathcal Z}}

\newcommand{\bfi}{{\bf i}}

\newcommand{\no}{\nonumber}
\newcommand{\lb}{\label}

\newcommand{\wti}{\widetilde  }

\newcommand{\ran}{\text{\rm{ran}}}

\newcommand{\bi}{\bibitem}

\newcommand{\hatt}{\widehat}

\newcommand{\mi}{\operatorname{Mas}}
\newcommand{\mo}{\operatorname{Mor}}

\newcommand{\noh}{N^{1/2}(\partial\Omega)}

\newcommand{\rnohs}{\rangle_{(N^{1/2}(\partial\Omega))^*}}
\newcommand{\lnoh}{_{N^{1/2}(\partial\Omega)}\langle}

\newcommand{\gd}{\widehat{\gamma}_D}
\newcommand{\gn}{\widehat{\gamma}_N}

\newcommand{\Htwo}{H^2(\Omega)}

\newcommand{\Htworm}{H^2(\Omega;\R^{2m})}

\newcommand{\Ltworm}{L^2(\Omega;\R^{m})}

\newcommand{\Ltwor}{L^2(\Omega;\mathbb{R})}

\newcommand{\bndr}{H^{{1}/{2}}(\partial \Omega,\mathbb C^m)\times H^{-{1}/{2}}(\partial \Omega,\mathbb C^m)}
\newcommand{\bndro}{H^{{1}/{2}}(\partial \Omega)\times H^{-{1}/{2}}(\partial \Omega)}

\numberwithin{equation}{section}

\renewcommand{\det}{\operatorname{det}}

\newcommand{\dom}{\operatorname{dom}}
\newcommand{\codim}{\operatorname{codim}}

\newcommand{\tr}{\operatorname{Tr}}

\newcommand{\sign}{\operatorname{sign}}

\newcommand{\spec}{\operatorname{Spec}}
\newcommand{\spflow}{\operatorname{SpFlow}}

\renewcommand{\ker}{\operatorname{ker}}
\newcommand{\diag}{\operatorname{diag}}

\newtheorem{theorem}{Theorem}[section]
\newtheorem{hypothesis}[theorem]{Hypothesis}
\newtheorem{lemma}[theorem]{Lemma}
\newtheorem{corollary}[theorem]{Corollary}
\newtheorem{proposition}[theorem]{Proposition}

\theoremstyle{definition}
\newtheorem{definition}[theorem]{Definition}

\newtheorem{example}[theorem]{Example}

\newtheorem{remark}[theorem]{Remark}

\begin{document} 
\begin{abstract}
We consider second order elliptic differential operators on a bounded Lipschitz domain $\Omega$. Firstly, we establish a natural one-to-one correspondence between their self-adjoint extensions, with domains of definition containing in $H^1(\Omega)$, and Lagrangian planes in $\bndro$. Secondly, we derive a formula relating the spectral flow of the one-parameter families of such operators to the Maslov index, the topological invariant counting the signed number of conjugate points of paths of Lagrangian planes in $\bndro$. Furthermore, we compute the Morse index, the number of negative eigenvalues, in terms of the Maslov index for several classes of the second order operators: the $\vec{\theta}-$periodic Schr\"odinger operators on a period cell $Q\subset \bbR^n$,  the elliptic operators with Robin-type boundary conditions, and the abstract self-adjoint extensions of the Schr\"odinger operators on star-shaped domains. Our work is built on the techniques recently developed by B.~Boo{\ss}-Bavnbek, K.\ Furutani, and C. Zhu, and extends the scope of validity of their spectral flow formula by incorporating the self-adjoint extensions of the {\it second} order operators with domains in the {\it first} order Sobolev space $H^1(\Omega)$. In addition, we generalize the results concerning relations between the Maslov and Morse indices quite recently obtained by G.\ Cox, J.\ Deng, C. Jones, J.~\ Marzuola, A.\ Sukhtayev and the authors. Finally, we describe and study a link between the theory of abstract boundary triples and the Lagrangian description of self-adjoint extensions of abstract symmetric operators. 
\end{abstract}

\allowdisplaybreaks

\title[Maslov Index]{The Maslov Index and the Spectra of Second Order Elliptic Operators}

\author[Y. Latushkin]{Yuri Latushkin}
\address{Department of Mathematics,
The University of Missouri, Columbia, MO 65211, USA}
\email{latushkiny@missouri.edu}
\author[S. Sukhtaiev]{Selim Sukhtaiev	}
\address{Department of Mathematics,
The University of Missouri, Columbia, MO 65211, USA}
\email{sswfd@mail.missouri.edu}
\date{\today}
\thanks{Supported by the NSF grant  DMS-1067929, by the Research Board and Research Council of the University of Missouri, and by the Simons Foundation. We thank S.\ Cappell for bringing to our attention several important papers, F.\ Gesztesy for discussions of the topic of self-adjoint extensions of symmetric operators, and  M.\ Beck, G.\ Cox, C.\ K.\ R.\ T.\ Jones for discussions of applications of the Maslov index to nonlinear partial differential equations.}
\maketitle
{\scriptsize{\tableofcontents}}

\section{Introduction}\lb{intro}

This paper intertwines three major themes: (1) Relations between the spectral flow for a family of linear elliptic differential operators and the Maslov index of a path of Lagrangian planes formed by the abstract traces of solutions of respective homogeneous partial differential equations; (2) Relations between the Morse index and the Maslov index in the context of Lagrangian planes given by standard PDE traces of weak solutions of the homogeneous equations; (3) Relations between the self-adjoint extensions of abstract symmetric operators and the Lagrangian planes defined by means of boundary triples. 

The first topic is motivated by the celebrated Atiyah--Patodi--Singer index theorem \cite{APS, AS}, and goes back to the classical works \cite{CLM, Nic, RoSa95}. In particular, great progress has been recently made in calculations of the spectral flow via the Maslov index,  \cite{BbF95, BZ3, BZ1, BZ2, KL, F04, SW}. Here, the spectral flow is the net count of the eigenvalues of a family of self-adjoint differential operators that move through a given value of the spectral parameter, and  the Maslov index is a topological invariant that measures the signed number of intersections of paths in the space of Lagrangian planes  \cite{arnold67, Arn85, dG, McS}. The second topic has its roots in the classical Morse--Smale-type theorems, see \cite{A01, B56, CD77, CZ84, D76, M63}. It is of great interest in stability theory for multidimensional patterns for reaction-diffusion equations, see \cite{KP, DJ11}. In recent years the relation between the Morse index (the number of unstable eigenvalues) and the Maslov index has attracted much attention, see  \cite{CJLS, CJM1, CJM2, DJ11, HS, HLS, JLM, JLS, LSS, PW}. These results can be viewed as a far reaching generalization of the classical Sturm Theorems for ODE's and systems of ODE's, cf. \cite{B56, arnold67, Arn85}, Courant's nodal domain theorem \cite{CH}, and more recent results in \cite{Fr}. The third topic is originated in the Birman--Krein--Vishik theory of self-adjoint extensions of symmetric operators, see a modern exposition in \cite{AlSi, Gr1}, and also in the theory of the abstract boundary triples, see \cite{GG, Ko}.  A critical series of results in this work is that the self-adjoint extensions of a symmetric operator can be parametrized by Lagrangian planes in some abstract boundary spaces. We note that although the Lagrangian language is not used in \cite{BL, Br, BM, GG} one can easily see an equivalent Lagrangian reformulation of these results contained therein. Summarizing, one can say that the connection between self-adjoint extensions and Lagrangian planes resulted in various formulas relating the spectral flow and the Maslov index of paths of Lagrangian planes formed by {\it strong traces} of solutions to elliptic PDE's, see \cite{BZ3, BZ1, BZ2}. In contrast, the Lagrangian planes considered in \cite{DJ11, CJLS, CJM1, CJM2} are formed by the {\it weak} traces of weak solutions to second order elliptic PDE's. The main contribution of this paper is in tying together all three topics discussed above.
  

Our work is motivated by that of J.\ Deng and C.\ K.\ R.\ T.\ Jones \cite{DJ11} who proposed to compute the Morse index of the Schr\"odinger operator $L=-\Delta+V$ on a star-shaped domain $\Omega$ by scaling it, and counting negative eigenvalues of $L$ via the conjugate points defined by intersections of underlying paths of Lagrangian planes in $\bndro$. The paths are formed by the boundary data of the weak solutions to the equation $Lu-\lambda u=0, \lambda\in\bbR,\ u\in H^1(\Omega),$ and by the subspaces of $\bndro$ corresponding to the boundary conditions.   This approach leads to a natural generalization of the Smale Theorem for Schr\"odinger operators with Robin-type boundary conditions, cf.  \cite{CJLS}. Further advances of this idea appeared in subsequent works: Significantly more general domain variations are considered in \cite{CJM1, CJM2}, the Schr\"odinger operators with non-separated boundary conditions are considered in  \cite{JLM, JLS, LSS, HS}, the one dimensional Schr\"odinger operators defined on $\R$ are considered in \cite{HLS}. While all these papers deal with specific boundary value problems, most of them may be viewed through the prism of abstract theory of self-adjoint extensions of symmetric operators. In a different context, the work in this direction was initiated in the foundational paper \cite{BbF95}, where the following setup was used: Let $S\subset S^*$ be a symmetric operator acting in a Hilbert space $\cH$, and $\{V_t\}_{t=\alpha}^{\beta}$ be a continuous family of bounded self-adjoint operators acting in $\cH$. Let us suppose that $S_{\mathscr D},$ $ \dom(\cS_{\mathscr D})=\mathscr D,$ is a self-adjoint extension of $S$ with compact resolvent. Then $\Upsilon_t:=\ker(S^*+V_t)/\dom(S),\ t\in[\alpha,\beta],$ and $\mathscr D/\dom(S)$ are Lagrangian planes in the quotient space $\cH_S:=\dom(S^*)/\dom(S)$ with respect to the natural symplectic form 
\begin{equation}\lb{u10}
\omega([x],[y])=\langle S^*x,y\rangle_{\cH}-\langle x,S^*y\rangle_{\cH},\ [x],[y]\in\dom(S^*)/\dom(S).
\end{equation} 
It is shown in \cite{BbF95} that the spectral flow of $\{S_{\mathscr D}+V_t\}_{t=\alpha}^{\beta}$ is equal to the Maslov index of the path $\Upsilon_t,t\in[\alpha, \beta],$ with respect to the reference plane $\mathscr D/\dom(S)$.
We notice that the operator $S$ gives rise not only to the one-parameter family $\{S_{\mathscr D}+V_t\}_{t=\alpha}^{\beta}$  but also to the symplectic Hilbert space $\cH_S$ itself. Therefore, the scheme is not suited for a parameter dependent  family $\{S_t\}_{t=\alpha}^{\beta}$ in place of a single operator $S$. However, the subsequent manuscripts \cite{BZ3}, \cite{BZ1}, \cite{BZ2} suggest an elegant way out of this issue. Let us consider a family $\{S_t\}_{t=\alpha}^\beta$ of symmetric operators with a fixed domain, and fix an intermediate space $D_M\subset\cH$ such that 
\begin{equation}\lb{u12u}
	\dom(S_t)\equiv \dom(S_{\alpha})\subset D_M \subset \dom(S^*_t)\subset \cH,\ t\in[\alpha, \beta]. 
\end{equation} 
We will now consider only those self-adjoint extensions of $S_t$ whose domains are subsets of the {\it fixed} subspace $D_M$. Under these assumptions \cite{BZ2} proves the equality of the spectral flow for the family of self-adjoint extensions of $S_t$, and the Maslov index defined by means of the quotient space $D_M/\dom(S_{\alpha})$.

The techniques developed in \cite{BZ2} cover elliptic operators of order $d\in\bbN$ with $D_M$ being equal to the Sobolev space of degree $d$. In particular, letting $d=2$ we consider now a family $\{S_t\}_{t=\alpha}^{\beta}$ of second order uniformly elliptic operators on a smooth domain $\Omega\subset \bbR^n$. Then \cite{BZ2} yields the equality between the spectral flow of the self-adjoint extensions of $S_t$ with domains containing in $D_M=H^2(\Omega)$ and the Maslov index of the corresponding paths of Lagrangian planes.

The purpose of our work is threefold. First, we will reduce the regularity assumption and consider the self-adjoint extensions of second order elliptic operators $S_t$ with domains containing in $H^1(\Omega)$. To illuminate the importance of this improvement we recall that  many differential operators of interest in mathematical physics, spectral geometry, and partial differential equations are defined via first order sesquilinear forms with the help of Lax--Milgram Theorem. This procedure {\it a priori} leads to self-adjoint operators with domains contained in $H^1(\Omega)$. The higher $H^2(\Omega)-$regularity is a subtle issue and depends not only on coefficients of the differential operators but also on geometric characteristics of $\partial\Omega$. Thus the assumption that the domains of self-adjoint extensions of $S_t$  belongs to $H^1(\Omega)$ is quite natural. The main technical obstacle preventing from passing from $H^2-$ to $H^1-$ regularity is that the natural candidate for $D_M$ from \eqref{u12u} is given by the subspace
\begin{equation}\lb{u14}
\{u\in H^1(\Omega): S_tu\in L^2(\Omega)\},
\end{equation}     
and thus varies together with parameter $t$ (if $H^1(\Omega)$ here is replaced by $H^2(\Omega)$ then $S_tu\in L^2(\Omega)$ holds automatically). To overcome this difficulty we map the family of subspaces \eqref{u14} into a {\it fixed} Hilbert space $\bndro$ using the trace map consisting of the Dirichlet and Neumann trace operators.

This brings us to the second goal of this paper. We will replace the quotient space $H^1(\Omega)/H^2_0(\Omega)$ by the more conventional space $\bndro$ of distributions on the boundary. We stress that the two-component trace map consisting of the Dirichlet and Neumann trace operators is not onto when considered as a map from $H^1(\Omega)$ into $\bndro$, cf. Proposition \ref{k1}. Moreover, the quotient space $H^1(\Omega)/H^2_0(\Omega)$ is  not symplectomorphic to the boundary space $\bndro$. Nevertheless, one can bypass the quotient spaces and instead work directly in $\bndro$. And, finally, our third goal is to analyze the variation of spectra of differential operators with respect to geometric deformations of the domain $\Omega$.    


Employing the approach outlined above we derive the spectral count formulas in a very general setting. In particular, it at once covers the following known cases: The Schr\"odinger operators with non-local Robin-type boundary conditions on star-shaped domains $\Omega\subset\bbR^n, n\geq 2$, as in \cite{CJLS}, the Schr\"odinger operators with $\vec{\theta}-$periodic boundary conditions on the unit cell $Q$, as in \cite{LSS}, the second order elliptic operators with Dirichlet and Neumann boundary conditions defined by means of a one-parameter family of diffeomorophic domains $\{\Omega_t\}_{t=\alpha}^{\beta}$, as in \cite{CJM1, CJM2}. In addition, we establish a connection of the Lagrangian description of the self-adjoint extensions of symmetric operators as in \cite{AlSi, BbF95}, and the theory of abstract boundary triples as in \cite{BL, BM, GG, Ko}. Using this connection, we obtain formulas relating the Morse and Maslov indices in an abstract settings, assuming the existence of a family of perturbations and a family of boundary triples. We demonstrate how to apply these formulas for the matrix one- and multidimensional Schr\"odinger operators.

We will now describe the main results of the paper in more details. Let $\Omega\subset\bbR^n$, $n\geq 2$ be a bounded Lipschitz domain, $ m\in\bbN$, and  let the functions
\begin{align}
&A: t\mapsto A^t\in\C^{m\times m},\ B: t\mapsto B^t\in\C^m, q: t\mapsto q^t\in\bbR,\ t\in[\alpha, \beta],
\end{align}
satisfy the following assumptions:
\begin{align}
&A\in C({[\alpha,\beta]},L^{\infty}(\Omega, \C^{m\times m})),\ A^t{\text{ is a self-adjoint matrix for all}\ t\in[\alpha, \beta]}  ,\lb{u1}\\
&B\in C({[\alpha,\beta]},L^{\infty}(\Omega, \C^m)),\ q\in C({[\alpha,\beta]}, L^{\infty}(\Omega, \bbR)).
\end{align}
Let us consider a family $\{\cL^t\}_{t=\alpha}^{\beta}$ of formally self-adjoint differential expressions,
\begin{equation}\lb{u4}
\cL^t:=- \text{div}A^t\nabla + B^t\nabla -\nabla\cdot\overline{B^t} +q^t,\ t\in[\alpha, \beta].
\end{equation}
The associated family of symmetric operators acting in $L^2(\Omega)$ is given by
\begin{align}
&L^tu:=\cL^t u,\ u\in\dom(L^t):=C^{\infty}_0(\Omega),\ t\in[\alpha,\beta]. \lb{u5}
\end{align}   
Each operator $L^t$ is closable in $L^2(\Omega)$, its closure is denoted by $\cL^t_{min},\ t\in[\alpha,\beta]$. 

First, we establish a natural one-to-one correspondence between the self-adjoint extensions of $\cL^t_{min}$ and the Lagrangian planes in $\bndro$. 
The Lagrangian plane $\cG_{\mathscr D_t}$ associated to a self-adjoint extension $\cL^t_{\mathscr{D}_t}$ of $\cL_{min}^t$ with the domain $\mathscr D_t\subset H^1(\Omega)$ is given by the formula
\begin{equation}
	\cG_{\mathscr D_t}=\overline{\tr_{\cL^t}(\mathscr{D}_t)}^{\bndro},\ t\in[\alpha,\beta],
\end{equation}    
where $\tr_{\cL^t}$ is a trace map defined below in \eqref{b5} by means of the differential expression $\cL^t$ from \eqref{u4}. For example, the plane corresponding to the Dirichlet Laplacian is given by $\{0\}\times H^{-1/2}(\partial\Omega)$, to the Neumann Laplacian is given by $ H^{1/2}(\partial\Omega)\times\{0\}$, and to the Robin Laplacian is given by $$\{(u, -\Theta u): u\in H^{1/2}(\partial\Omega)\}, \Theta\in \cB_{\infty}\big(H^{1/2}(\partial\Omega),H^{-1/2}(\partial\Omega)\big).$$ 

Second, we define the set of traces of the weak solutions of the corresponding homogeneous PDE with no boundary conditions by the formula
\begin{equation}\lb{u6}
{\cK}_{\lambda, t}:=\tr_{\cL^t}\{u\in H^1(\Omega): \cL^tu-\lambda u=0\},
\end{equation}
show that this plane is Lagrangian, and recast the eigenvalue problem
\begin{equation}\lb{u8}
	\cL^tu-\lambda u=0,\ u\in\mathscr{D}_t,
\end{equation}
in terms of the intersection of the Lagrangian planes $\cK_{\lambda,t}$ and $\cG_{\mathscr D_t}$.
Namely, we prove that
\begin{equation}\lb{u9}
\dim\ker(\cL_{\mathscr{D}_t}-\lambda)=\dim\big(\cK_{\lambda,t}\cap \cG_{\mathscr{D}_t}\big),\ \lambda\in\bbR,\ t\in[\alpha,\beta].
\end{equation}
Formula \eqref{u9} provides a link between the eigenvalues of $\cL^t_{\mathscr{D}_t}$ and the conjugate points of the paths formed by  the Lagrangian planes $\cK_{\lambda,t},\cG_t$ in $\bndro$.
With this at hand we show the principal result of our work, cf. Theorem \ref{w33},
\begin{equation}\lb{u7}
\mo(\cL^{\alpha}_{\mathscr D_{\alpha}})-\mo(\cL^{\beta}_{\mathscr D_{\beta}})=\mi\big((\cK_{0,t},\cG_{\mathscr {D} _t})|_{t\in[\alpha,\beta]}\big),
\end{equation}
where the Morse index $\mo(\cL_{\mathscr D_{\alpha}})$ is defined as the number of negative eigenvalues of the operator $\cL_{\mathscr D_{\alpha}}$, and $\mi\big((\cK_{0,t},\cG_{\mathscr {D} _t})|_{t\in[\alpha,\beta]}\big)$ denotes the Maslov index of the paths $\{\cK_{t}\}_{t=\alpha}^{\beta}, \{\cG_{\mathscr{ D}_t}\}_{t=\alpha}^{\beta}$ defined Section 2.1. 

The left-hand side of \eqref{u7} can be viewed as the spectral flow through zero of the eigenvalues of the operator family $\{\cL_{\mathscr{D}_t}\}_{t=\alpha}^{\beta}$. A slight generalization of  \eqref{u7} recovers the above mentioned relations between the Maslov index and the spectral flow from \cite{BbF95}, \cite{BZ2} (in case of second order operators), and between Maslov index and the Morse index from \cite{CJLS, CJM1, DJ11, JLM, JLS, LSS}.

The paper is organized as follows. Section \ref{sec2} provides the one-to-one correspondence between the self-adjoint extensions  of $\cL_{min}$, with the domains contained in $H^1(\Omega)$, and Lagrangian planes in $\bndro$. In Section \ref{section3} we derive the formula relating the Maslov and Morse indices for second order elliptic operators subject to self-adjoint boundary conditions on smooth domains. The applications of the general result are illustrated for three topics: the spectral flow formula, the spectral count in the context of geometric deformations, and the Smale-type theorem for Robin boundary conditions. In Section \ref{section4} we deal with the Maslov--Morse type formulas for the Schr\"odinger operators with matrix-valued potentials subject to self-adjoint boundary conditions on Lipschitz domains. In particular, we consider $\vec{\theta}-$periodic and non-local Robin-type boundary conditions. Finally, in Section \ref{section5} we discuss the abstract boundary triples \cite{BL, BM, Br, GG} in the context of the quotient spaces introduced in \cite{BbF95}. 

To conclude we summarize the notation used in this paper. The scalar product in a complex Hilbert space $\cH$ is denoted by $\langle\cdot,\cdot \rangle_{\cH}$. $\spec(S)$, $ \spec_{ess}(S)$, $\spec_d(S)$ denote the spectrum, the essential spectrum, the discrete spectrum of a closed operator $S$ correspondingly. The number of negative eigenvalues of $S$ is denoted by $\mo(S)$. If $\cG\subset \cH$ then $\overline{\cG}^{\cH}$  denotes the closure of $\cG$ with respect to the norm of $\cH$. The range of an operator $\cS$ acting from  a Banach space $\cX$ into a Banach space $\cY$  is denoted by $\ran(\cS)\subset\cY$, the kernel of $\cS$ is denoted by $\ker(\cS)\subset\cX$. The space of bounded operators acting from $\cX$ to $\cY$ is denoted by $\cB(\cX,\cY)$, the space of compact operators is denoted by $\cB_{\infty}(\cX,\cY)$. If $\Omega\subset\bbR^n$, then $\cD(\Omega)$ denotes the space of test functions $C_0^{\infty}(\Omega)$ equipped with the standard inductive limit topology, $\cD'(\Omega)$ denotes the dual space, the paring between $\cD(\Omega)$ and $\cD'(\Omega)$ is denoted by $_{\cD(\Omega)}\langle\cdot,\cdot \rangle_{\cD'(\Omega)}$. Duality pairing between $H^{1/2}(\partial\Omega)$ and $H^{-1/2}(\partial\Omega)$ is denoted by $\langle\cdot,\cdot\rangle_{-1/2}$. 
\section{Self-adjoint extensions and Lagrangian planes}\lb{sec2}
In this section we focus on a one-to-one correspondence between self-adjoint extensions of second order elliptic operators on bounded domains in $\R^n$ and Lagrangian subspaces in $\bndr$. 

\subsection{Assumptions} In this subsection we state our main assumptions and recall some known facts.
\begin{hypothesis}\lb{d2}
Let $n\in\bbN, n\geq 2$ and assume that $\Omega\subset \R^n$ is a bounded Lipschitz domain. 
\end{hypothesis}
To set the stage, we introduce a formally self-adjoint differential expression
\begin{equation}\lb{d1}
\cL :=-\sum_{j,k=1}^{n} \partial_j A_{jk}\partial_k + \sum_{j=1}^{n}A_j\partial_j -\partial_j \overline{A_j}^{\top}+A,
\end{equation}
where bar means complex conjugation, ${\top}$ means matrix transposition, and the coefficients satisfy the following standard assumptions, see, e.g., \cite[Chapter 4]{Mc}.
\begin{hypothesis}\lb{l1} Let $\Omega\subset\R^n$ be a bounded open set, and assume that
\begin{align}
&A_{jk}=\big\{a^{jk}_{pq}\big\}_{p,q=1}^m\in L^{\infty}(\overline{\Omega}, \C^{m\times m}),\, A_{jk}=\overline{A_{jk}}^{\top},\,1\leq j,k\leq n,  \no\\
&A_{j}=\big\{a^{j}_{pq}\big\}_{p,q=1}^m\in L^{\infty}(\overline{\Omega}, \C^{m\times m}),  1\leq j\leq n,\no\\
&A_{jk},\ A_{j} \text{\ are Lipschitz functions on}\ \overline{\Omega},\ 1\leq j,k\leq n,\no\\
&A=\big\{a_{pq}\big\}_{p,q=1}^m\in L^{\infty}(\overline{\Omega}, \C^{m\times m}),\,  A=\overline{A}^{\top}.\no
\end{align}
\end{hypothesis}

The differential expression $\cL$ acts on a vector-valued function $u\in C^{\infty}(\Omega, \C^m)$ as follows
\begin{align}
&\big((\cL u)(x)\big)_p=-\sum_{j,k=1}^{n} \sum_{q=1}^{m}\partial_j\{ a^{jk}_{pq}(x)\partial_k u_q(x)\} +\sum_{j=1}^{n} \sum_{q=1}^{m}{a^j_{pq}(x)}\partial_ju_q(x)  \\
&\quad - \sum_{q=1}^{m}\partial_j \{\overline {a}^j_{qp}(x)u_q(x)\}+\sum_{q=1}^{m}a_{pq}(x)u_q(x),\, \text{a.e.}\ x\in\Omega,\, 1\leq p\leq m,
\end{align}
where $(v)_p$ denotes the $p-$th coordinate of a vector $v\in\C^m$.
The sesquilinear form associated with $\cL$ is given by
\begin{align}
\begin{split}
\mathfrak{l}[u,v]&=\sum_{j,k=1}^{n}\langle A_{jk}\partial_ku,\partial_j v\rangle_{L^2(\Omega,\C^m)}+\sum_{j=1}^{n}\langle A_j\partial_ju,v\rangle_{L^2(\Omega,\C^m)}\lb{cc1}\\
&\quad +\sum_{j=1}^{n}\langle u,A_j\partial_jv\rangle_{L^2(\Omega,\C^m)}+\langle Au,v\rangle_{L^2(\Omega,\C^m)},\ u,v\in H^1(\Omega,\C^m).
\end{split}
\end{align}

We seek to establish a one-to-one correspondence between self-adjoint extensions of $\cL:C_0^{\infty}(\Omega,\mathbb C^m)\subset L^2(\Omega,\mathbb C^m)\rightarrow L^2(\Omega,\mathbb C^m)$ and Lagrangian planes in $\bndr$ employing the second Green identity. To this end, we recall
that by the standard trace theorem (cf., e.g., \cite[Theorem 3.38]{Mc}) the linear mapping
\begin{equation}\lb{e8}
\gamma_{D}^0: C(\overline{\Omega},\C^m)\rightarrow C(\partial\Omega,\C^m),\ \ \gamma_{D}^0u=u_{|_{\partial\Omega}},
\end{equation}
can be extended by continuity and considered as a linear bounded operator,
\begin{equation}\lb{e9}
\gaD \in \cB(H^s(\Omega,\C^m), H^{(s-\frac12)}(\partial\Omega,\C^m)),\, 1/2<s<3/2;
\end{equation} 
in addition (cf. \cite[Lemma A.4]{GM08}),
\begin{equation}\lb{aa6}
\gaD \in \cB(H^{(3/2)+\varepsilon}(\Omega,\C^m), H^{1}(\partial\Omega,\C^m)),\, \varepsilon>0.
\end{equation}
The conormal derivative corresponding to the differential expression $\cL$ is given by
\begin{align}
&{\gaN^{\cL,2}}u:=\sum_{j,k=1}^{n} A^{jk}\nu_j\gaD(\partial_k u)+\sum_{j=1}^{n}\overline{A_j}^{\top}\nu_j \gaD u,\ u\in H^2(\Omega,\C^m),\lb{e2}
\end{align}
with $\nu=(\nu_1,\cdots,\, \nu_n)$ denoting the outward unit normal on $\partial\Omega$. Setting $\varepsilon=1/2$ in \eqref{aa6} and using \eqref{e2}, we introduce the trace map
\begin{equation}\lb{e6}
\tr_{\cL,2}:
\begin{cases}
H^2(\Omega,\C^m)\rightarrow H^{1}(\partial\Omega,\C^m)\times L^2(\partial\Omega,\C^m), \\
\hspace*{0.759cm} u\mapsto (\gaD u, {\gaN^{\cL,2}}).
\end{cases}
\end{equation}

We introduce the function space 
\begin{align}
& D^{s}_{\cL}(\Omega):=\{u\in H^s(\Omega,\mathbb C^m): \cL u\in L^2(\Omega,\mathbb C^m)\},\ s\geq 0,\lb{cc4}
\end{align}
equipped with the graph norm of $\cL$,
\begin{equation}\lb{e1}
\|u\|_{\cL,s}:=\left(\|u\|_{H^s({\Omega},\C^m)}^2+\|\cL u\|_{L^2({\Omega,\C^m})}^2\right)^{1/2}, 
\end{equation}
where $\cL u$ should be understood in the sense of distributions.
Next, we recall the extension of $\tr_{\cL,2}$ defined on $D^{1}_{\cL}(\Omega)$ and the first and second Green identities.
	
\begin{proposition}\cite[Lemma 4.3]{Mc}\lb{e10} Assume Hypothesis \ref{d2}. Then the operator
\begin{equation}\lb{cc6}
{\gaN^{\cL,2}}:H^2(\Omega,\C^m)\rightarrow L^2(\partial\Omega,\C^m),\ u\in H^2(\Omega,\C^m),
\end{equation}	
can be extended to a bounded, linear operator $\gamma_N^{\cL}\in\cB\big(\cD^1_{\cL}(\Omega), H^{-1/2}(\partial\Omega,\C^m)\big)$. Moreover, the first Green identity holds, that is,
\begin{align}
&\mathfrak{l}[u,v]=\langle\cL u, v\rangle_{L^2(\Omega,\C^m)}+\langle {\gaN^{\cL}}u,\gaD v\rangle_{-1/2},\lb{e14}
\end{align}
for all $u\in\cD^1_{\cL}(\Omega), v\in H^1{(\Omega,\C^m)}$.
\end{proposition}
\begin{proposition}\cite[Theorem 4.4 ({\it iii})]{Mc}\lb{e12} Assume Hypothesis \ref{d2}. Then the second Green identity holds, that is, 
\begin{align}
\langle\cL u, v\rangle_{L^2(\Omega,\C^m)}- \langle u,\cL v\rangle_{L^2(\Omega,\C^m)}=\overline{\langle{\gaN^{\cL}} v ,\gaD u \rangle}_{{-1/2}}-\langle{\gaN^{\cL}} u ,\gaD v \rangle_{{-1/2}},\lb{e13}
\end{align}
for all $u,v\in\cD^1_{\cL}(\Omega)$.
\end{proposition}
The trace operator 
\begin{equation}\lb{b5}
\tr_{\cL}\in\cB\big(\cD^1_{\cL}(\Omega), \bndr\big),\ \tr_{\cL}:u\mapsto (\gaD u, {\gaN^{\cL}} u),
\end{equation}
is compatible with \eqref{e6}. We notice that, it follows from the unique continuation principle, \cite[Theorem 3.2.2]{Is}, that 
\begin{equation}\lb{UCP}
\ker\{\tr_{\cL}\}\cap\{u\in H^1(\Omega): \cL u=0\  \text{in}\ (H^1_0(\Omega))^*  \}=\{0\}.
\end{equation}

Next we turn to a symmetric operator acting in $L^2(\Omega,\C^m)$ and associated with differential expression \eqref{d1}.
\begin{proposition} \lb{w3}Assume Hypotheses \ref{d2} and \ref{l1}. Then the linear operator defined by 
\begin{equation}\lb{dd1}
Lf:=\cL f,\ f\in\dom(L):=C_0^{\infty}(\Omega),
\end{equation}	
and considered in $L^2(\Omega,\C^m)$ is closable. Its closure $\cL_{min}$ is densely defined symmetric operator in $L^2(\Omega,\C^m)$. Moreover, the linear operator acting in $L^2(\Omega,\C^m)$ and given by 
\begin{equation}\lb{dd2}
	\cL_{max}u:=\cL u,\ u\in\dom(\cL_{max}):=\{u\in L^2(\Omega,\C^m): \cL u\in L^2(\Omega,\C^m)\},
\end{equation}
$($where $\cL u$ is defined is the sense of distributions$)$ is adjoint to $\cL_{min}$, i.e.,
\begin{equation}\lb{e18}
(\cL_{min})^*=\cL_{max}.
\end{equation}
\end{proposition}
\begin{proof} Using the second Green identity \eqref{e13} with arbitrary $u,v\in C^{\infty}_0(\Omega)$ and noticing that $\tr_{\cL}u=\tr_{\cL}v=0$ we arrive at 
\begin{equation}\lb{dd5}
\langle\cL u, v\rangle_{L^2(\Omega,\C^m)}=\langle u,\cL v\rangle_{L^2(\Omega,\C^m)},\ \text{for all}\ u,v\in \dom(L).
\end{equation}
Hence, $L\subset L^*$ that is $L$ is symmetric in $L^2(\Omega,\C^m)$, consequently it is closable. 

Next, we turn to \eqref{e18}. Let us show  $(\cL_{min})^*\subset\cL_{max}$. Pick any $f\in \dom \left((\cL_{min})^*\right)$, then $g=(\cL_{min})^*f\in L^2(\Omega,\C^m)$, and for arbitrary $\psi\in C_0^{\infty}(\Omega)\subset\dom\left(\cL_{min}\right)$ one has
\begin{equation}\lb{dd6}
{}_{\cD(\Omega)} \langle\psi, \cL f\rangle_{\cD'(\Omega)}=\langle\cL \psi, f\rangle_{L^2(\Omega,\C^m)}=\langle \psi, g\rangle_{L^2(\Omega,\C^m)}.
\end{equation}
Therefore, $g=\cL f$ in distributional sense and $\cL f\in L^2(\Omega,\C^m)$ as required. In order to show the opposite inclusion we notice that 
\begin{equation}\lb{dd7}
\langle \cL \varphi, g\rangle_{L^2(\Omega,\C^m)}=\langle \varphi, \cL g\rangle_{L^2(\Omega,\C^m)},\ \text{for all}\ \varphi\in C_0^{\infty}(\Omega), g\in\dom(\cL_{max}).
\end{equation}
Whenever $f\in\dom(\cL_{min})$, there exists a sequence $\{\varphi_\ell, \ell\geq 1\}\subset C_0^{\infty}(\Omega)=\dom(L)$, such that
\begin{equation}
\lim\limits_{\ell\rightarrow\infty}\varphi_{\ell}=f\ \ \text{and}\ \lim\limits_{\ell\rightarrow\infty}\cL\varphi_{\ell}=\cL f \ \ (\text{in}\ L^2(\Omega,\C^m)).
\end{equation}
Plugging $\varphi_{\ell}$ in \eqref{dd7} and passing to limit as $\ell\rightarrow\infty$, one obtains
\begin{equation}\lb{dd8}
\langle \cL f, g\rangle_{L^2(\Omega,\C^m)}=\langle f, \cL g\rangle_{L^2(\Omega,\C^m)},\ \text{for all}\ f\in \dom(\cL_{min}),\, g\in\dom(\cL_{max}).
\end{equation}
Thus, $\cL_{max}\subset(\cL_{min})^*$, and the proof is completed. 
\end{proof}
\begin{hypothesis}\lb{bb1} Assume Hypotheses \ref{d2} and \ref{l1}. Suppose that the deficiency indices of $\cL_{min}$ are equal, that is,
\begin{equation*}
\dim\ker(\cL_{max}-\bfi)=\dim\ker(\cL_{max}+\bfi).
\end{equation*} In addition:
			
(i)	assume that $\ran(\tr_{\cL})$ is dense in $\bndr$,

(ii) assume that $D^1_{\cL}(\Omega)$ is dense in $\cD^0_{\cL}(\Omega)$.
\end{hypothesis}
\begin{remark}\lb{yy1}
In the sequel we consider two special cases. In the  first case, the coefficients $A_{jk}, A_j, A$ of a uniformly elliptic operator $\cL$ from \eqref{d1} are scalar-valued and defined on domains with smooth boundary, cf. Hypothesis \ref{q1} below. In this scenario both parts of Hypothesis \ref{bb1} are satisfied. Indeed, by \cite[Proposition 2.1]{Gr}, \cite[Section 4.3]{BM},
\begin{equation}\no
\ran(\tr_{\cL,2})=H^{3/2}(\partial\Omega)\times H^{1/2}(\partial\Omega),
\end{equation}
and the right-hand side is dense in $\bndro$. Furthermore, by \cite[Theorem 3.2]{Gr}, $H^2(\Omega)$ is dense in $\cD^s_{\cL}(\Omega), s<2$, hence $D^1_{\cL}(\Omega)$ is dense in $\cD^0_{\cL}(\Omega)$. In the second case, the coefficients are given by $A_{jk}=\delta_{jk}I_m$,  where $\delta_{jk}$ denotes the Kronecker delta, $A_j=0_n, 1\leq j,k\leq n$, and $A=V$, that is, we deal with the Schr\"odinger operator $\cL=-\Delta+V$ with a matrix potential, cf. Section \ref{section4}. The domain $\Omega$ in this case is assumed to be Lipschitz. Then, using auxiliary spaces of distributions on $\partial\Omega$, cf. \cite{GM08}, we show in Proposition \ref{ff8} that both {\it (i)} and {\it (ii)} from the Hypothesis \ref{bb1} are verified. While a detailed analysis of Hypothesis \ref{bb1} is of independent interest (cf., e.g., \cite{BG}, \cite{DM}, \cite{Ge}, \cite{GM10}, \cite{GM08}) and barely touched upon in the present paper, we stress that in all our applications the assumptions of Hypothesis \ref{bb1} are satisfied.
\end{remark}


\subsection{The Lagrangian planes and the self-adjoint extensions of differential operators}
Let us introduce a complex symplectic bilinear form
\begin{align}
\begin{split}
&\omega: \bndr\rightarrow \C\times\C,\lb{dd9}\\
&\omega((f_1,g_1), (f_2,g_2))=\overline{\langle g_2 ,f_1 \rangle}_{{-1/2}}-\langle g_1,f_2 \rangle_{{-1/2}},\\
&(f_1,g_1), (f_2,g_2)\in\bndr.
\end{split}
\end{align} 
Then the second Green identity \eqref{e13} reads as follows
\begin{equation}\lb{bb88}
\langle\cL u_0, v\rangle_{L^2(\Omega,\C^m)}- \langle u_0,\cL v\rangle_{L^2(\Omega,\C^m)}=\omega\left((\gaD u_0,{\gaN^{\cL}} u_0), (\gaD v,{\gaN^{\cL}} v)\right),
\end{equation}
for all $u,v\in \cD_{\cL}^1(\Omega)$. 
We recall that the annihilator of $\cF\subset \bndr$ is defined by 
\begin{align}
&\cF^{\circ}:=\{(f,g)\in\bndr|\no\\
&\hspace{4cm}\omega((f,g), (\phi, \psi))=0,\text{\,for all\,} (\phi, \psi)\in \cF\}.\no
\end{align}
The subspace $\cF$ is said to be isotropic if $\cF\subset \cF^{\circ}$, co-isotropic if $\cF^{\circ}\subset \cF$, $\cF$ is called Lagrangian if it is simultaneously isotropic and co-isotropic. Furthermore, $\cF$ is Lagrangian if and only if it is maximally isotropic, cf., e.g., \cite{F04}.

The principal goal of this section is to identify self-adjoint extensions of $\cL_{min}$, whose domains are subsets of $H^1(\Omega,\C^m)$, with the Lagrangian subspaces in $\bndr$.  We recall notation \eqref{b5}.

\begin{theorem}\lb{l4} Assume Hypothesis \ref{bb1}. Then the self-adjoint extensions of $\cL_{min}$ whose domains are contained in $H^1(\Omega)$ are in one-to-one correspondence with Lagrangian planes in $\bndr$, that is, the following two assertions hold.

1. Let $\mathscr D\subset\cD^1_{\cL}(\Omega)$, and let $\cL_{\mathscr D}$ be the linear operator acting in $L^2(\Omega, \C^m)$ and given by the formula
\begin{equation}
\cL_{\mathscr D}f:=\cL_{max} f,\ f\in \dom(\cL_{\mathscr D}):=\mathscr D.
\end{equation}
If $\cL_{\mathscr D}$ is self-adjoint, then the set
\begin{equation}
	\cG_{\mathscr D}:=\overline{\tr_{\cL}(\mathscr D)}^{\bndr},
\end{equation}
is a Lagrangian plane in $\bndr$ with respect to symplectic form $\omega$ defined in \eqref{dd9}. 

2. A Lagrangian plane $\cG\subset \bndr$ defines a self-adjoint extension of $\cL_{min}$. Namely, the linear operator $\cL_{\tr_{\cL}^{-1}(\cG)}$ acting in $L^2(\Omega)$ and given by the formula
\begin{equation}
\cL_{\tr_{\cL}^{-1}(\cG)}f:=\cL_{max} f,\ f\in \dom\left(\cL_{\tr_{\cL}^{-1}(\cG)}\right):={\tr_{\cL}^{-1}(\cG)},
\end{equation}
is essentially self-adjoint, here $\tr_{\cL}^{-1}(\cG)$ denotes the preimage of $\cG$, that is,
\begin{equation*}
\tr_{\cL}^{-1}(\cG):=\{u\in\cD_{\cL}(\Omega)\ :\ \tr_{\cL}u\in\cG\}.
\end{equation*}
\end{theorem}

\begin{proof} 

First we concentrate on proving part {\it 1}. In order to show that $\cG_{\mathscr D}$ is isotropic we employ the second Green identity \eqref{e13}: For arbitrary pairs $(\gaD u, {\gaN^{\cL}} u)\in \cG_{\mathscr D}$, $  (\gaD v, {\gaN^{\cL}} v)\in \cG_{\mathscr D}$ one has
\begin{align}
\begin{split}
\omega\left((\gaD u, {\gaN^{\cL}} u), (\gaD v, {\gaN^{\cL}} v)\right)&=\overline{\langle{\gaN^{\cL}} v ,\gaD u \rangle}_{{-1/2}}-\langle{\gaN^{\cL}} u ,\gaD v \rangle_{{-1/2}}\\
&=\langle\cL u, v\rangle_{L^2(\Omega)}- \langle u,\cL v\rangle_{L^2(\Omega)}=0,
\end{split}
\end{align}
where the latter equality follows since $\cL_{\mathscr D}$ is symmetric. Next, we show maximality of the isotropic subspace $\cG_{\mathscr D}$, that is, that 
\begin{equation}\lb{b2}
\cG_{\mathscr D}^{\circ}\subset\cG_{\mathscr D}.
\end{equation}
To this end, we shall establish an intermediate inclusion 
\begin{equation}\lb{b3}
	\cG_{\mathscr D}^{\circ}\cap\tr_{\cL}(\cD^1_{\cL}(\Omega))\subset \cG_{\mathscr D}.
\end{equation}
Indeed, if $(f,g)\in \cG_{\mathscr D}^{\circ}\cap\tr_{\cL}\left(\cD^1_{\cL}(\Omega)\right)$ then  
\begin{equation}\lb{b6}
 (f,g)=\left(\gaD u_0,{\gaN^{\cL}} u_0\right),\text{\, for some\ } u_0\in \cD^1_{\cL}(\Omega),
\end{equation}
and
\begin{equation}\lb{b7}
\omega\left((\gaD u_0,{\gaN^{\cL}} u_0), (\gaD v,{\gaN^{\cL}} v)\right)=0,\text{\, for all\ }v\in\mathscr D.
\end{equation}
On the other hand, using the second Green identity \eqref{e13} with $u=u_0$ and $v\in\mathscr D$, one has
\begin{equation}\lb{b8}
\langle\cL u_0, v\rangle_{L^2(\Omega,\C^m)}- \langle u_0,\cL v\rangle_{L^2(\Omega,\C^m)}=\omega\left((\gaD u_0,{\gaN^{\cL}} u_0), (\gaD v,{\gaN^{\cL}} v)\right).
\end{equation}
Hence,
\begin{equation}\lb{b9}
\langle\cL u_0, v\rangle_{L^2(\Omega)}=\langle u_0,\cL v\rangle_{L^2(\Omega)},\text{\, for all\ }v\in\mathscr D, 
\end{equation}
and therefore, 
\begin{equation}\lb{b10}
u_0\in \dom\left((\cL_{\mathscr D})^*\right)=\dom (\cL_{\mathscr D}),
\end{equation}
since $\cL_{\mathscr D}$ is self-adjoint by the assumption. Finally, using \eqref{b6} and inclusion \eqref{b10}, one infers $(f,g)\in \cG_{\mathscr D}$ and completes the proof of assertion \eqref{b3}.
Next we prove inclusion \eqref{b2}. Employing \eqref{b3} and the standard properties of annihilator, we obtain  
\begin{equation}\lb{b11}
		\cG_{\mathscr D}^{\circ}\subset \overline{\cG_{\mathscr D}+\left\{\tr_{\cL}(\cD^1_{\cL}(\Omega))\right\}^{\circ}}^{\bndr}.
\end{equation}
But
\begin{equation}\lb{b12}
\left\{\tr_{\cL}(\cD^1_{\cL}(\Omega))\right\}^{\circ}=\{(0,0)\},
\end{equation}	
since by the assumptions in Hypothesis \ref{bb1} the set $\tr_{\cL}(\cD_{\cL}^1(\Omega)))$ is dense in $\bndr$. Hence, the proof of part {\it 1} is completed.

Proof of part {\it 2.} 
Following \cite[Section 3.1]{BbF95} we introduce the space of abstract boundary values $\cH_{\cL}:= \dom(\cL_{max})/\dom(\cL_{min})$ equipped with the norm
\begin{equation}
\|[x]\|_{\cH_{\cL}}:=\inf\{\|x+f\|_{\cL,0}: f\in\dom(\cL_{min}) \},
\end{equation}
where $[x]$ is the equivalence class of $x\in\dom(\cL_{max})$ and $\|\cdot\|_{\cL,0}$ is the graph norm from \eqref{e1}. We define the symplectic form on $\cH_{\cL}$ by the formula 
\begin{align}
\begin{split}
\wti \omega([x]&,[y]):=\langle \cL_{max} x, y \rangle_{L^2(\Omega,\C^m)}-\langle x,\cL_{max} y 
\rangle_{L^2(\Omega,\C^m)},\lb{ee2}\\
&\text{for all}\ [x],[y]\in\dom(\cL_{max})/\dom(\cL_{min}).
\end{split}
\end{align}   

Now we are ready to proceed with the proof of part {\it 2}.  Let $\cD_{\cG}:=\tr^{-1}(\cG)$. By \cite[Lemma 3.3 (b)]{BbF95}, it suffices to show that the closure of the subspace
\begin{equation}
[\cD_{\cG}]:=\{[x]: x\in\cD_{\cG}\},
\end{equation}
is Lagrangian in $\cH_{\cL}$ with respect to $\wti \omega$. Denoting the annihilator of $[\cD_{\cG}]$ by $[\cD_{\cG}]^{\circ}$, we notice that 
\begin{equation}\lb{c2}
[\cD_{\cG}]\subset [\cD_{\cG}]^{\circ},
\end{equation}
hence, the subspace is isotropic. In order to show the maximality of the closure of $[\cD_{\cG}]$, we will obtain an auxiliary inclusion
\begin{equation}\lb{c3}
[\cD_{\cG}]^{\circ}\cap [\cD^1_{\cL}(\Omega)]\subset [\cD_{\cG}].
\end{equation}
Indeed, if  $[u_0]\in[\cD_{\cG}]^{\circ}\cap [\cD_{\cL}^1(\Omega)]$ then 
\begin{equation}\lb{c4}
\wti\omega\left([u_0],[v]\right)=0, \text{\, for all\,} [v]\in [\cD_{\cG}],
\end{equation}
that is,
\begin{equation}\lb{c5}
\langle \cL_{max} u_0, v \rangle_{L^2(\Omega,\C^m)}-\langle u_0,\cL_{max} v \rangle_{L^2(\Omega,\C^m)}=0, \text{\, for all\,} [v]\in [\cD_{\cG}].
\end{equation}
Since $[u_0]\in [\cD^1_{\cL}(\Omega)]$, the trace map $\tr_{\cL}$ is well defined on $u_0$ and $(\gaD u_0,\gaN^{\cL} u_0)\in \bndr$. Hence, by the second Green identity, one has 
\begin{align}
\begin{split}
&\langle \cL_{max} u_0, v \rangle_{L^2(\Omega)}-\langle u_0,\cL_{max} v \rangle_{L^2(\Omega,\C^m)}\lb{c5.1}\\
&\quad=\overline{\langle{\gaN^{\cL}} v ,\gaD u_0 \rangle}_{{-1/2}}-\langle{\gaN^{\cL}} u_0 ,\gaD v \rangle_{{-1/2}}=0, 
\end{split}
\end{align}
for all $[v]\in [\cD_{\cG}]$. Therefore 
\begin{equation}\lb{c6}
\left(\gaD u_0,\gaN^{\cL} u_0\right)\in \big(\cG\cap\tr_{\cL}(\cD_{\cL}^1(\Omega))\big)^{\circ}.
\end{equation}
We claim that $\big(\cG\cap\tr_{\cL}(\cD_{\cL}^1(\Omega))\big)^{\circ}=\cG$. Indeed,
\begin{equation}\lb{c7}
\big(\cG\cap\tr_{\cL}(\cD_{\cL}^1(\Omega))\big)^{\circ}=\overline{\cG+\big(\tr_{\cL}(\cD_{\cL}^1(\Omega))\big)^{\circ}}=\overline{\cG+{(0,0)}}=\cG,
\end{equation}
where we used \eqref{b12}. Inclusion \eqref{c6} together with \eqref{c7} yield $(\gaD u_0,\gaN^{\cL} u_0)\in \cG$, which in turn implies $\tr_{\cL}^{-1}(u_0)\in \cD_{\cG}$. Consequently \eqref{c3} holds. Next, applying the annihilator operator $^{\circ}$ to \eqref{c3}, one obtains
\begin{equation}\lb{c8}
\left[\cD_{\cG}\right]^{\circ}\subset \overline{\overline{[\cD_{\cG}]}+ \left[\cD^1_{\cL}(\Omega)\right]^{\circ}}.
\end{equation}
Since $\cD^1_{\cL}(\Omega)$ is dense in $D^1_{\cL}(\Omega)$, one has 
\begin{equation}\lb{c9}
\left[\cD^1_{\cL}(\Omega)\right]^{\circ}=\{[0]\}.
\end{equation}
Combining \eqref{c9} and \eqref{c8} one obtains \eqref{c2}. 
\end{proof}

We illustrate Theorem \ref{l4} by describing the Lagrangian planes associated with different self-adjoint extensions of $\cL_{min}$. First, we consider the setup from \cite[Chapter 7]{F95}, \cite{CJM1}.  Let $\cX$ be a closed subspace in $H^1(\Omega,\C^m)$ and assume that $H_0^1(\Omega,\C^m)\subset\cX\subset H^1(\Omega,\C^m)$. In addition, suppose that the form 
\begin{equation}\lb{yy20}
\mathfrak{l}:L^2(\Omega,\C^m)\times L^2(\Omega,\C^m)\rightarrow \C, \dom(\mathfrak{l}):=\cX,
\end{equation}
is closed and bounded from below in $L^2(\Omega,\C^m)$. Then, by \cite[Theorem 2.8]{EE89}, there exists a unique self-adjoint operator $\cL_{\cX}$ acting in $L^2(\Omega,\C^m)$ such that
\begin{equation}\lb{yy22}
\mathfrak{l}[u,v]=\langle\cL_{\cX} u, v\rangle_{L^2(\Omega,\C^m)}\text{\ for all\ }u\in\dom(\cL_{\cX}), v\in\cX. 
\end{equation}
The domain of $\cL_{\cX}$ is given by the formula
\begin{align}
\begin{split}\lb{2.52a}
\dom(\cL_{\cX}):=\{ u\in\cX\ : &\exists w\in L^2(\Omega)\text{\ such that\ } \\
&\langle w,v\rangle_{L^2(\Omega)}=\mathfrak{l} [w,v]\ \text{for all}\  v\in\cX \}.
\end{split}
\end{align}
\begin{proposition}\lb{prop29} Let
\begin{align}
\begin{split}
\cG_{\cX}:=\{(&f,g)\in\bndro: \lb{r12r}\\
&f\in\gaD({\cX}), \langle g, \gaD w \rangle_{-1/2}=0 \text{\ for all\ }w\in {\cX}\}. 
\end{split}
\end{align}
Then $\cG_{\cX}$ is a Lagrangian plane. Moreover, $\tr^{-1}_{\cL}(\cG_{\cX})$ is a core of $\cL_{\cX}$.
\end{proposition}
\begin{proof} The plane $\cG_{\cX}$ is Lagrangian by \cite[Lemma 3.6]{CJM1}. It remains to show that $\tr_{\cL}(\dom(\cL_{\cX}))\subset\cG_{\cX}$. By the first Green identity \eqref{e14}, for each $u\in\cD^1_{\cL}(\Omega)$, $v\in\cX$ we have
\begin{equation}\lb{yy21}
\mathfrak{l}[u,v]=\langle\cL u, v\rangle_{L^2(\Omega,\C^m)}+\langle {\gaN^{\cL}}u,\gaD v\rangle_{-1/2}.
\end{equation}
Combining \eqref{yy22} and \eqref{yy21} we obtain
\begin{equation}\lb{yy23}
\langle {\gaN^{\cL}}u,\gaD v\rangle_{-1/2}=0\text{\ for all\ }u\in\dom(\cL_{\cX}), v\in H^1(\Omega,\C^m).
\end{equation}
Using $\dom(\cL_{\cX})\subset \cX$ and \eqref{yy23} we conclude that $(\gaD u, {\gaN^{\cL}}u)\in\cG_{\cX}$ if $u\in \dom(\cL_{\cX})$, as required.
\end{proof}
 Proposition \ref{prop29} shows that the operator $\cL_{\cX}$, defined in  \eqref{2.52a}, \eqref{yy22}, is associated with Lagrangian plane $\cG_{\cX}$ as indicated in Theorem \ref{l4}. In particular, if $\cX:=H^1_0(\Omega,\C^m)$ then $\cG_{\cX}=\{0\}\times H^{-1/2}(\partial\Omega,\C^m)$ and  $\cL_{\cX}$ corresponds to the Dirichlet boundary conditions. If $\cX:=H^1(\Omega,\C^m)$ then $\cG_{\cX}=H^{1/2}(\partial\Omega,\C^m)\times\{0\}$ and  $\cL_{\cX}$ corresponds to the Neumann boundary conditions.

 We remark next that $\tr_{\cL}$ is not onto in general. This amounts to the fact that $\tr_{\cL}$ does not map domains of self-adjoint extensions into Lagrangian planes in $\bndr$, but only into their dense subsets. 
\begin{proposition}\lb{k1}
Let $\Omega\subset \R^n, n\geq 2$ be an open, bounded domain with smooth boundary. Then the map $\tr_{\Delta}$ corresponding to the Laplacian,
\begin{align}
\tr_{\Delta}:=(\gaD,\gaN):\cD^1_{\Delta}(\Omega)\rightarrow H^{1/2}(\partial\Omega)\times H^{-1/2}(\partial\Omega),
\end{align}
is not surjective.  
\end{proposition}
\begin{proof}
We prove the assertion by contradiction. Assume that $\tr_{\Delta}$ is surjective. Under this assumption one can show that $\cF:=\tr_{\Delta}(H^2(\Omega)\cap H^1_0(\Omega))$ is Lagrangian plane in $\bndr$. Indeed, $\cF\subset \cF^{\circ}$ since 
\begin{align}
\begin{split}
\omega(\tr_{\Delta} u, \tr_{\Delta} v)&=\overline{\langle\gaN v ,\gaD u \rangle}_{{-1/2}}-\langle\gaN u ,\gaD v \rangle_{{-1/2}}=0,\lb{m4}\\
&\text{for all}\ u,v\in H^2(\Omega)\cap H^1_0(\Omega).
\end{split}
\end{align}
In order to prove $\cF^{\circ}\subset \cF$, let us fix an arbitrary $(f,g)\in\lambda^{\circ}$. Since $\tr_{\Delta}$ is assumed to be surjective, there exists $v_0\in \cD_{\Delta}^1(\Omega)$ such that $\tr_{\Delta}(v_0)=(f,g)$. Furthermore,
\begin{equation}\lb{m7}
\omega(	\tr_{\Delta} u,(f,g))=\omega\left(\tr_{\Delta} u, \tr_{\Delta} v_0\right)=0,\ \text{for all}\ u\in H^2(\Omega)\cap H^1_0(\Omega). 
\end{equation}
In addition, the second Green identity yields
\begin{align}
&\langle-\Delta u, v_0\rangle_{L^2(\Omega)}- \langle u,-\Delta v_0\rangle_{L^2(\Omega)}=\overline{\langle\gaN v_0 ,\gaD u \rangle}_{{-1/2}}-\langle\gaN u ,\gaD v_0 \rangle_{{-1/2}}.\lb{m8}
\end{align}
Combining \eqref{m7} and \eqref{m8}, one obtains
\begin{equation}\lb{m9}
\langle-\Delta u, v_0\rangle_{L^2(\Omega)}- \langle u,-\Delta v_0\rangle_{L^2(\Omega)}=0, \text{for all}\ u\in H^2(\Omega)\cap H^1_0(\Omega).
\end{equation}
Let us recall that $H^2(\Omega)\cap H^1_0(\Omega)$ is the domain of Dirichlet Laplacian $-\Delta_{D}$, which is a self-adjoint operator in $L^2(\Omega)$. Therefore, \eqref{m9} leads to $v_0\in H^2(\Omega)\cap H^1_0(\Omega)$, which in turn implies $\tr_{\Delta} v_0=(f,g)\in \cF$ and $\cF^{\circ}\subset \cF$. Finally, we arrive at $\cF=\cF^{\circ}.$ On the other hand, 
\begin{equation}\lb{m99}
\cF=\tr_{\Delta}\left(H^2(\Omega)\cap H^1_0(\Omega)\right)=\{0\}\times H^{1/2}(\partial\Omega).
\end{equation}
The set $\{0\}\times H^{1/2}(\partial\Omega)$ is not closed in $ H^{1/2}(\partial\Omega)\times  H^{-1/2}(\partial\Omega)$, thus it is not Lagrangian. This contradiction completes the proof.
\end{proof}

\subsection{The Maslov index} We will now recall from \cite{BZ3}, \cite{BZ1}, \cite{BZ2}, \cite{BZ4} a definition of the Maslov index of a path of Lagrangian planes in a complex Hilbert space $\cX$ relative to a reference plane. This will require some preliminaries. Let $\omega$ be a symplectic form on $\cX$, i.e.,  we assume that $\omega: \cX\times \cX\rightarrow \C$ is a sesquilinear, bounded, skew-symmetric, non-degenerate form. Then there exists a bounded operator $J:\cX\rightarrow \cX$, such that
\begin{equation}\lb{ab1}
\omega(u,v)=\langle Ju,v\rangle_{\cX},\ \ u,v\in\cX,
\end{equation}
and
\begin{equation}\lb{ab2}
J^2=-I_{\cX}, J^{*}=-J.
\end{equation}
Moreover, $\cX$ admits an orthogonal decomposition into direct sum of the eigenspaces of the operator $J$, that is, 
\begin{equation}\lb{ab3}
	\cX=\ker(J-\bfi I)\oplus\ker(J+\bfi I).
\end{equation}
Therefore, the form $-\bfi\omega$ is positive definite on $\ker(J-\bfi I)$, the form $-\bfi\omega$ is negative definite on $\ker(J+\bfi I)$, and $\omega(u,v)=0$ whenever  $u\in \ker(J-\bfi I), v\in\ker(J+\bfi I)$.
 We denote the annihilator of a subset $\cF\subset\cX$ by
\begin{equation}
	\cF^{\circ}:=\{u\in \cX: \omega(u,v)=0 \text{\ for all}\ v\in \cF\}.
\end{equation}
The subspace $\cF$ is called {\it Lagrangian} if $\cF=\cF^{\circ}$. The set of Lagrangian subspaces of $\cX$ is denoted by
\begin{equation}\lb{ba20}
\Lambda(\cX):=\{\cF\subset\cX: \cF\ \text{is Lagrangian in $\cX$} \}.
\end{equation}
 Following \cite[Lemma 3]{BZ2}, we notice that every Lagrangian plane $\cF$ can be uniquely represented as a graph of a bounded operator $U\in\cB(\ker(J+\bfi I_{\cX}), \ker(J-\bfi I_{\cX}))$, i.e., one has
\begin{equation}\lb{ab4}
	\cF=\gr(U):=\{y+Uy: y\in \ker(J+\bfi I_{\cX})\}.
\end{equation}
That is, $Uy\in\ker(J-\bfi I_{\cX})$ is the unique vector satisfying $y+Uy\in\cF$ for $y\in\ker(J+\bfi I_{\cX})$.
Moreover,
\begin{equation}\lb{ab5}
	\omega(x,y)=-\omega(Ux,Uy), \ x,y\in \ker(J+\bfi I_{\cX}).
\end{equation}
The operator $U$ is a unitary map acting between the Hilbert spaces $\ker(J+\bfi I_{\cX})$ and $\ker(J-\bfi I_{\cX})$. Indeed, for arbitrary $x,y\in\ker(J+\bfi I_{\cX})$ one has
\begin{align}
\begin{split}
&\langle x,y\rangle_{\cX}=\bfi \langle Jx,y\rangle_{\cX}=\bfi\omega(x,y)\lb{ab6}\\
&\quad=-\bfi \omega(Ux,Uy)=-\bfi\langle JUx,Uy\rangle_{\cX}=\langle Ux,Uy\rangle_{\cX}.
\end{split}
\end{align}

A pair of Lagrangian planes $\cF,\mathcal Z$ is called Fredholm pair if
\begin{align}\lb{ab7}
\dim(\cF\cap\cZ)<\infty,\  \cF+\cZ \text{\ is closed in $\cX$, and\ } \codim(\cF+\cZ)<\infty.
\end{align}
Let $\cF=\gr(U)$ and $\cZ=\gr(V)$ be Lagrangian planes in $\cX$, then by \cite[Lemma 2]{BZ2}, the pair $(\cF,\cZ)$ is Fredholm if and only if $UV^{-1}-I_{\cX}$ is Fredholm operator in $\ker(J-\bfi I_{\cX})$. Furthermore,
\begin{equation}\lb{ab8}
	\dim (\cF\cap \cZ)= \dim \ker (UV^{-1}-I_{\cX}).
\end{equation}

Let us fix a Lagrangian plane 
\begin{equation}
\cZ\subset\cX, \cZ=\gr(V),
\end{equation}
where $ V\in \cB(\ker(J+\bfi I_{\cX}), \ker(J-\bfi I_{\cX}))$ is a unitary operator. The Fredholm-Lagrangian-Grassmannian is the space
\begin{equation}\lb{ab9}
	F\Lambda(\cZ):=\{\cF\subset\cX: \cF \ \text{is Lagrangian, and the pair\ } (\cF,\cZ)\ \text{is Fredholm}\},
\end{equation}
equipped with metric 
\begin{equation}
	d(\cF_1, \cF_2):=\|P_{\cF_1}-P_{\cF_2}\|_{\cB(\cH)},\ \cF_1,\cF_2\in F\Lambda(\cZ),
\end{equation}
where $P_{\cF}$ denotes the orthogonal projection onto ${\cF}$.
Let $\cI=[a,b]\subset\bbR$ be a set of parameters. 
Let us fix a continuous path in $F\Lambda(\cZ)$
\begin{equation}
\Upsilon:\cI\rightarrow F\Lambda(\cZ),\ \ \Upsilon(s)=\cF_s,\ \Upsilon\in C(\cI,  F\Lambda(\cZ)),
\end{equation}
and  introduce the corresponding family of unitary operators $U_s$ such that 
\begin{align}
&\hspace{2cm}\cF_s=\gr(U_s),\ s\in\cI,\no\\
&\upsilon : \cI\rightarrow \cB(\ker(J+\bfi I_{\cX}), \ker(J-\bfi I_{\cX})),\  \upsilon(s)=U_s.\no
\end{align} 
The following is proved in \cite{BZ1}
\begin{align}
& \upsilon\in C(\cI,\cB(\ker(J+\bfi I_{\cX}), \ker(J-\bfi I_{\cX}))),\lb{ab12}\\
& U_sV^{-1}  \text{\ is unitary in\ } \ker(J-\bfi I_{\cX}),\ s\in\cI,\lb{ab13}\\
& U_sV^{-1}-I_{\cX} \text{\ is Fredholm in\ } \ker(J-\bfi I_{\cX}),\ s\in \cI,\lb{ab14}\\
&\dim (\cF_s\cap \cZ)= \dim \ker (U_sV^{-1}-I_{\cX})\lb{ab15},\ s\in \cI.
\end{align}
Utilizing \eqref{ab12}-\eqref{ab15} we will now define the Maslov index as a spectral flow through, the point $1\in\C$, of the family $\upsilon(s), s\in \cI$.  An illuminating discussion of the notion of the spectral flow of a family of closed operators through an admissible curve $\ell\subset\C$ can be found in \cite[Appendix]{BZ2}.  To proceed with the definition, note that due to \eqref{ab14} there exists a partition $a=s_0<s_1<\cdots<s_N=b$ of $[a,b]$ and positive numbers $\varepsilon_j\in(0,\pi)$, such that  $\e^{\pm\bfi \varepsilon_j}\not \in \Sp (U_sV^{-1})$ if $s\in[s_{j-1},s_j]$, for each $1\leq j\leq N$, see \cite[Lemma 3.1]{F04}. For any $\varepsilon>0$ and $s\in[a,b]$ we let
\begin{equation}\lb{ba18}
k(s,\varepsilon):=\sum\nolimits_{0\leq \varkappa\leq \varepsilon}\dim\ker(U_sV^{-1}-\e^{\bfi\varkappa}),
\end{equation}
and define the Maslov index
\begin{equation}\lb{dfnMInd}
\text{Mas}(\Upsilon,\cZ):=\sum\limits_{j=1}^{N}\left(k(s_j,\varepsilon_j)-k(s_{j-1},\varepsilon_j)\right).
\end{equation}
The number Mas$(\Upsilon,\cX)$ is well defined, i.e., it is independent on the choice of the partition $s_j$ and $\varepsilon_j$ (cf., \cite[Proposition 3.3]{F04}).

Next we turn to the computation of the Maslov index via the crossing forms.  Assume that $\Upsilon\in C^1(\cI, F\Lambda(\cX))$ and let $t_*\in\cI$.  There exists a neighbourhood $\cI_0$ of $t_*$ and a family $R_t\in C^1(\cI_0, \cB(\Upsilon(t_*), \Upsilon(t_*)^{\perp}))$, such that $\Upsilon(t)=\{u+R_tu\big| u\in \Upsilon(t_*)\}$, for $t\in \cI_0$ see, e.g., \cite{F04} or \cite[Lemma 3.8]{CJLS}.  
We will use the following terminology from \cite[Definition 3.20]{F04}.
\begin{definition}\label{def21} Let $\cZ$ be a Lagrangian subspace and $\Upsilon\in C^1(\cI, F\Lambda(\cZ))$.
	
	{\it (i)} We call $s_*\in\cI$ a conjugate point or crossing if $\Upsilon(s_*)\cap \cZ\not=\{0\}$.
	
	{\it (ii)} The finite dimentional form $$\cQ_{s_*,\cZ}(u,v):=\frac{d}{ds}\omega(u,R_sv)\big|_{s=s_*}=\omega(u, \dot{R}_{s=s_*}v), \text{\ for\ }u,v \in \Upsilon(s_*)\cap \cZ,$$  is called the crossing form at the crossing $s_*$.
	
	{\it (iii)} The crossing $s_*$ is called regular if the form $\cQ_{s_*,\cZ}$ is non-degenerate, positive if $\cQ_{s_*,\cZ}$ is positive definite, and negative if $\cQ_{s_*,\cZ}$ is negative definite.

\end{definition}

The following result (cf., {\cite[Proposition 3.2.7]{BZ1}} and Remark \ref{ba23}) provides an efficient tool for computing the Malsov index at regular crossings. We denote by $n_+$ and $n_-$ the number of positive and negative squares fo a form, the signature is defined by the formula $\sign=n_+-n_-$. 
\begin{theorem} \lb{masform}
	Let $\Upsilon\in C^1(\cI, F\Lambda(\cZ))$, and assume that all crossings are regular. Then the crossings are isolated, and one has 
	\begin{equation}\lb{ab17}
	\mi(\Upsilon,\cZ)=-n_-(\cQ_{a,\cZ})+\sum\limits_{a<s<b}\sign(\cQ_{s,\cZ})+n_+(\cQ_{b,\cZ}).
	\end{equation}
\end{theorem}
We will now review the definition of the Maslov index for {\it two} paths with values in Lagrangian--Grassmannian $\Lambda(\cX)$, see \cite[Section 3.5]{F04}. Let us fix $\Upsilon_1,\Upsilon_2\in C(\cI, \Lambda(\cX))$ and assume that $(\Upsilon_1(s),\Upsilon_2(s))$ is a Fredholm pair for all $s\in \cI$. Let $\diag:=\{(p,p):p\in\cX\}$ denote the diagonal plane in $\cX\oplus\cX$.  On  $\cX\oplus\cX$ we define the symplectic form $\hat{\omega}:=\omega\oplus(-\omega)$ with the complex structure $\wti{J}:=J\oplus(-J)$, denoting the resulting space of Lagrangian planes  by $\Lambda_{\hat{\omega}}(\cX\oplus\cX)$. We consider the path $\wti{\Upsilon}:=\Upsilon_1\oplus\Upsilon_2\in C (\cI,\Lambda_{\hat{\omega}}(\cX\oplus\cX))$ and define the Maslov index of the two paths $\Upsilon_1,\Upsilon_2$ as $\mi(\Upsilon_1,\Upsilon_2):=\mi(\wti{\Upsilon},\diag).$
If $\Upsilon_2(t)=\cZ$ for all $t\in\cI$, then $ \mi(\Upsilon_1\oplus \Upsilon_2,\diag)=\mi(\Upsilon_1,\cZ).$
\begin{remark}\lb{ba23}
We adopted definition \eqref{dfnMInd} of the Maslov index as the spectral flow of $U_sV^{-1}$ through the point $1$. Since $\varkappa$ in \eqref{ba18} is allowed to be equal to zero, the Maslov index defined in \eqref{dfnMInd} counts the number of the eigenvalues of $U_sV^{-1}$ that leave the {\it closed} segment $\{e^{\bfi\varkappa}: \varkappa\in[0,\varepsilon] \}$ through $1$ as parameter $s$ varies from $a$ to $b$. In comparison, the Maslov index defined in \cite[Definition 2.1.1]{BZ1}  counts the number of eigenvalues that leave the {\it open} segment $\{e^{\bfi\varkappa}: \varkappa\in(0,\varepsilon) \}$. This difference in definitions is reflected in the formula relating the Maslov index and the signature of the crossing form. In our case, the Maslov index at the left (respectively, right) regular endpoint  crossing is equal to minus(respectively, plus) the number of negative (respectively, positive) directions of the crossing form. The Maslov index from \cite[Proposition 3.2.7]{BZ1} is equal to the number of positive(respectively, minus the number of negative) directions. We find definition \eqref{dfnMInd} more convenient as it permits to obtain a relation between the Maslov index of a certain path, and the Morse index of a family of self-adjoint operators without adding the dimension of subspace corresponding to the zero eigenvalue into the Morse index. 
\end{remark}
\begin{remark}
In this remark we discuss the relation between two different versions of representation of a Lagrangian plane as the graph of an operator. Assume that a Lagrangian plane $\cV\subset \cX$ is written in two different ways,
\begin{align}
&\cV=\gr(U), U\in \cB(\ker(J+\bfi I_{\cX}), \ker(J-\bfi I_{\cX})),\\
&\cV=\{x+JAx: x\in\cF\},\ \cF\in\Lambda(\cX),\ A\in\cB(\cF), 
\end{align}
Then, the operator $A$ is self-adjoint, since for all $x,y\in\cF$
\begin{align}
\begin{split}
0&=\omega(x+JAx,y+JAy)=\omega(x,JAy)+\omega(JAx,y)\lb{ba1}\\
&=(Jx,JAy)_{\cX}+(J^2Ax,y)_{\cX}=(x,Ay)_{\cX}-(Ax,y)_{\cX}.
\end{split}
\end{align}
Moreover, the operator $U_{\cF}:=(I-\bfi J)^{-1}U(I+\bfi J)$ considered in the subspace $\cF$ (notice that $\ker(I\pm\bfi J)=\{x\pm\bfi Jx: x\in \cF\}$) is equal to the Cayley transform of $A$. Indeed, using 
\begin{equation}\lb{baa20}
U(x+\bfi Jx)= U_{\cF}x-\bfi JU_{\cF}x,\ x\in\cF,
\end{equation}
together with the definition of $U$, one obtains
\begin{align}
&x+\bfi Jx+U_{\cF}x-\bfi JU_{\cF}x=y+JAy,\ y\in\cF\lb{ba21},\\
&x+U_{\cF}x+J(\bfi x-\bfi U_{\cF})=y+JAy\lb{ba22}.
\end{align}
Hence, $x+U_{\cF}x=y,\ x-U_{\cF}x=-\bfi Ay$ as asserted.
\end{remark}
\begin{remark}
The starting point for the definition of the Maslov index given in \cite{BbF95}, \cite{F04}  is a {\it real}  Hilbert space $\cH_{\bbR}$ equipped with a symplectic form. The Maslov index  in \cite{BbF95}, \cite{F04} is defined as the spectral flow (through $-1$) of a family of unitary operators (acting in an auxiliary complex space $\cH_{\C}$) obtained via the Souriau map. While the assumption that $\cH_{\bbR}$ is a {\it real} Hilbert space is not restrictive in many applications (cf., e.g., \cite{CJLS}, \cite{CJM1}, \cite{CJM2},  \cite{JLM}, \cite{JLS}, \cite{LSS}), it does prevent one from considering complex-valued boundary conditions (such as $\theta-$periodic, see below) without reduction to equivalent real-valued boundary conditions. Given the abstract nature of the  eigenvalue problem for self-adjoint extensions of $\cL$ (as in \eqref{d1}), a reduction to the real Hilbert spaces (i.e., to the real boundary conditions) cannot be carried out explicitly. Instead, we choose to adopt the definition of the Maslov index in complex symplectic Hilbert spaces. As it was pointed out in \cite[Corollary 2]{BZ2}, there is a natural identification between the Maslov index in the real Hilbert space $\cH_{\bbR}$ and the Maslov index in the complex Hilbert space $\cH_{\bbR}\otimes\C$ (the complexification of $\cH_{\R}$) defined as in \eqref{dfnMInd}. 
\end{remark}
\section{The Maslov index for second order elliptic operators on smooth domains}\lb{section3}
The main result of this section concerns with an index formula for second order elliptic operators with scalar coefficients defined on a smooth domain $\Omega\subset\R^n$, see Theorem \ref{w33}.
\subsection{Weak solutions and their traces} In this subsection we reformulate the eigenvalue problems for elliptic operators in terms of Lagrangian subspaces formed by the traces of {\it weak} solutions of corresponding equations.
\begin{hypothesis}\lb{q1}
Let $\Omega\subset\R^n,\ n\geq 2,$ be a bounded open set with smooth boundary. Let $\cI:=[\alpha,\beta],\ -\infty<\alpha<\beta<+\infty,$ be the interval of parameters. Assume that $ a^t, a^t_{j}, a^t_{jk}$ are contained in  $C^{\infty}(\overline{\Omega})$ for all $t\in\cI$. Suppose that  
\begin{align}
& a_{jk}:t\mapsto a^t_{jk},\ a_{jk}\in C^1(\cI,L^{\infty}(\Omega)),\ a^t_{jk}(x)=\overline{a^t_{kj}(x)},\ 1\leq j \leq n, x\in\overline{\Omega},\lb{q2}\\
& a^t_{jk}(x)\xi_k\overline{\xi_j}\geq c \sum\limits_{j=1}^n|\xi_j|^2\ \text{for all}\ x\in\overline{\Omega}, \xi=(\xi_j)_{j=1}^n\in\C^n, t\in\cI, \text{and some}\ c>0,\ \lb{q11}\\
& a_{j}:t\mapsto a^t_{j},\ a_{j}\in C^1(\cI,L^{\infty}(\Omega)),\ 1\leq j \leq n,\lb{q12}\\
& a:t\mapsto a^t,\ a\in C^1(\cI,L^{\infty}(\Omega)),\ a^t(x)\in\bbR,\ x\in\Omega,\ t\in\cI.\lb{q13}
\end{align}
\end{hypothesis}
Given the families of the functions $\{a^t\}_{t=\alpha}^{\beta}$, $\{a_j^t\}_{t=\alpha}^{\beta}$, $\{a_{jk}^t\}_{t=\alpha}^{\beta}$ we now consider the family $\{\cL^t\}_{t=\alpha}^{\beta}$ of the differential expressions
\begin{equation}\lb{r7}
\cL^t :=-\sum_{j,k=1}^{n} \partial_j a^t_{jk}\partial_k + \sum_{j=1}^{n}a^t_j\partial_j -\partial_j \overline{a^t_j}+a^t,\ t\in\cI,
\end{equation}
which are formally self-adjoint.
For $t\in\cI$  the minimal operator corresponding to the differential expression $\cL^t$ in $L^2(\Omega)$ is defined by the formula
\begin{equation}\lb{w1}
\cL^t_{min} f= \cL^tf,\  f\in\dom(\cL^t_{min}):=H^2_0(\Omega).
\end{equation}
The operator $\cL^t_{min}$ is a densely defined, bounded from bellow, symmetric operator. Its adjoint $\cL_{max}^t:=(\cL_{min}^t)^*$ is acting in $L^2(\Omega)$ and given by the formula  
\begin{equation}\lb{w2}
\cL^t_{max}u:=\cL^t u,\ u\in\dom(\cL^t_{max}):=\{u\in L^2(\Omega,\C^m): \cL^t u\in L^2(\Omega,\C^m)\}.
\end{equation}
Given a family of self-adjoint extensions  $\{\cL^t_{\mathscr {D}_t}\}_{t=\alpha}^{\beta}$ of $\cL_{min}^t$ with $\ \dom(\cL^t_{\mathscr {D}_t})={\mathscr {D}_t}$, one has the chain of extensions
\begin{equation}\lb{w3w}
\cL^t_{min}\subset \cL^t_{\mathscr D_t}\subset \cL^t_{max}.
\end{equation}

As discussed in Remark \ref{yy1}, all assumptions of Hypothesis \ref{bb1} are satisfied with $\cL$ in \eqref{d1} replaced by $\cL^t$ from \eqref{r7}. Hence $\cL^t_{min}$  fits the framework of Theorem \ref{l4} and its self-adjoint extensions are uniquely associated with Lagrangian planes in $\bndro$ via Theorem \ref{l4}.  

Our objective is to relate the Morse indices of the operators $\cL^{\beta}_{\mathscr D_{\beta}}$ and $\cL^{\alpha}_{\mathscr D_{\alpha}}$ to the Maslov index of a certain path  of Lagrangian planes in $\bndro$ defined by the given one parameter family of self-adjoint operators $\{\cL^t_{\mathscr {D}_t}\}_{t=\alpha}^{\beta}$. This will be achieved by utilizing homotopy invariance of the Maslov index. To this end we introduce a parametrization of the square loop in  Figure 1,
\begin{figure}
	\begin{picture}(100,100)(-20,0)
	\put(78,0){0}
	\put(80,8){\vector(0,1){95}}
	\put(5,10){\vector(1,0){95}}
	\put(71,40){\text{\tiny $\Gamma_2$}}
	\put(12,40){\text{\tiny $\Gamma_4$}}
	\put(45,75){\text{\tiny $\Gamma_3$}}
	\put(43,14){\text{\tiny $\Gamma_1$}}
	\put(100,12){$\lambda$}
	\put(85,100){$s$}
	\put(80,20){\line(0,1){60}}
	\put(10,20){\line(0,1){60}}
	\put(80,8){\line(0,1){4}}
	\put(2,2){$\lambda_\infty$}
	\put(10,20){\line(1,0){70}}
	\put(10,80){\line(1,0){70}}
	\put(65,20){\circle*{4}}
	\put(80,50){\circle*{4}}
	\put(80,70){\circle*{4}}
	\put(20,80){\circle*{4}}
	\put(40,80){\circle*{4}}
	\put(60,80){\circle*{4}}
	\put(20,87){{\tiny \text{eigenvalues}}}
	\put(14,24){{\tiny \,\,\,\text{eigenvalues}}}  
	\put(85,23){\rotatebox{90}{{\tiny conjugate points}}}
	\put(82,18){{\tiny $\alpha$}}
	\put(82,80){{\tiny $\beta$}}
	\put(0,25){\rotatebox{90}{{\tiny no intersections}}}
	\end{picture}
	\caption{\ }
\end{figure} 
\begin{align}
&\Sigma:=\cup_{j=1}^4\Sigma_j\to\Gamma=\cup_{j=1}^4\Gamma_j,\ s\mapsto (\lambda(s), t(s)),\lb{w17w}
\end{align}
where $\Gamma_j,\ j=1,\cdots, 4$ are the positively oriented sides of the boundary of the square $[\lambda_{\infty},0]\times [\alpha,\beta]$, the parameter set $\Sigma=\cup_{j=1}^4 \Sigma_j$ and $\lambda(\cdot)$, $t(\cdot)$ are defied as follows:
\begin{align}
&\lambda(s)=s,\, t(s)=\alpha ,\, s\in\Sigma_1:=[\lambda_{\infty},0],\lb{18}\\
&\lambda(s)=0,\, t(s)=s+\alpha ,\, s\in\Sigma_2:=[0,\beta-\alpha ],\lb{19}\\
&\lambda(s)= -s+\beta-\alpha ,\, t(s)= \beta,\, s\in\Sigma_3:=[\beta-\alpha,\beta-\alpha -\lambda_{\infty}],\lb{20}\\
&\lambda(s)=\lambda_{\infty},\, t(s)=-s+2\beta-\alpha -\lambda_{\infty},\lb{21}\\
&\hskip3cm s\in\Sigma_4:=[\beta-\alpha -\lambda_{\infty}, 2(\beta-\alpha)-\lambda_{\infty}].\no
\end{align} 

We now turn to the eigenvalue problem
\begin{equation}\lb{w5}
\cL^{t(s)}_{\mathscr D_{t(s)}} u=\lambda(s) u,\ u\not=0,\ s\in\Sigma.
\end{equation}
Recalling notation \eqref{b5}, for the family $\{\cL^{t(s)}_{\mathscr{D}_t}\}_{s\in\Sigma}$ of the self-adjoint operators from \eqref{r7}--\eqref{w3w} and the parametrization $t(\cdot)$, $\lambda(\cdot)$ from \eqref{w17w}--\eqref{21} we now define the following subspaces:
\begin{align}
&\cG_{t(s)}:=\overline{\tr_{\cL^{t(s)}}(\mathscr D_{t(s)})},\ \cK_{\lambda(s),t(s)}:=\tr_{\cL^{t(s)}}({\bf K}_{\lambda(s),t(s)}), \lb{w15}
\end{align}
\begin{align}
\begin{split}
&{\bf K}_{\lambda(s),t(s)}:=\Big\{u\in H^1(\Omega):\sum_{j,k=1}^{n}\langle a^{t(s)}_{jk}\partial_ku,\partial_j \varphi\rangle_{L^2(\Omega)}+\sum_{j=1}^{n}\langle a^{t(s)}_j\partial_ju,\varphi\rangle_{L^2(\Omega)}\\
&\quad+\sum_{j=1}^{n}\langle u,a^{t(s)}_j\partial_j\varphi\rangle_{L^2(\Omega)}+\langle a^{t(s)}u-\lambda(s)u,\varphi\rangle_{L^2(\Omega)}=0,\ \varphi\in H^1_0(\Omega)\Big\},\ s\in\Sigma. \no
\end{split}
\end{align}
The subspace ${\bf K}_{\lambda(s),t(s)}$ is the set of weak solutions to the equation $\cL^{t(s)}u=\lambda_s u$, the subspace $\cK_{\lambda(s),t(s)}$ is the set of their traces, and $\cG_{t(s)}$ is the subspace in $\bndro$ that corresponds to $\mathscr{D}_{t(s)}$ as indicated in Theorem \ref{l4}.

Our next Theorem \ref{w7} shows, in particular, that the existence of nontrivial solutions to  \eqref{w5} is equivalent to
\begin{equation}\lb{w6}
\cG_{t(s)} \cap \cK_{\lambda(s),t(s)}\not=\{0\},\ s\in\Sigma.
\end{equation}
Theorem \ref{w7} is an improvement of \cite[Proposition 3.5]{CJM1}, see also \cite[Propositoin 4.10]{CJLS}. Proposition 3.5 in \cite{BbF95} provides an elegant proof of a related assertion in the context of strong solutions and abstract boundary traces. This result cannot be directly applied in the setting of the weak traces and weak solutions, however, we adopted the proof of \cite[Proposition 3.5]{BbF95} in order to show part {\it ii)} in the following theorem. The novel part {\it ii)} of this theorem states that just the Fredholm property of the operator $\cL^{t(s)}_{\lambda(s)}-\lambda_{(s)} I_{L^2(\Omega)}$ alone implies that the pair of subspaces $\cK_{\lambda(s),t(s)}$ ({\it weak} traces of {\it weak} solutions) and $\cG_{t(s)}$ is Fredholm . We note that assertion {\it iii)} in the next theorem was proved in \cite[Proposition 3.5]{CJM1} (see also  \cite[Proposition 4.10]{CJLS}). 
\begin{theorem}\lb{w7} Assume Hypothesis \ref{q1}. Let $\mathscr D_t\subset \cD_{\cL^t}^1(\Omega),$ $t\in\cI:=[\alpha,\beta]$, and assume that the linear operator $\cL^t_{\mathscr D_t}$ acting in $L^2(\Omega)$ and given by
\begin{equation}\lb{w8}
\cL^t_{\mathscr D_t} u:=\cL^tu,\ u\in \dom(\cL^t_{\mathscr D_t}):={\mathscr D_t},
\end{equation}
is self-adjoint for each $t\in\cI.$

Then the following assertions hold:

i$)$ if $s\in \Sigma$, then $\cK_{\lambda(s),t(s)}$ and $\cG_{t(s)}$ 
 are Lagrangian planes with respect to symplectic form  \eqref{dd9},
	
ii$)$ if $\spec_{ess}\left(\cL^{t(s)}_{\mathscr{D}_{t(s)}}\right)\cap(-\infty,0]=\emptyset$ then $\left(\cK_{\lambda(s),t(s)},\cG_{t(s)}\right)$ is a Fredholm pair of Lagrangian planes in $\bndro$, moreover,
\begin{equation}\lb{w11}
\dim\left(\cK_{\lambda(s),t(s)}\cap\cG_{t(s)}\right)=\dim\ker\left(\cL^{t(s)}_{\mathscr D_{t(s)}}-\lambda(s)\right), s\in\Sigma,
\end{equation}
	
iii$)$ the path $s\mapsto \cK_{\lambda(s),t(s)}$ on $\Sigma=\cup_{j=1}^4\Sigma_j$ is continuous and is contained in the space
\begin{equation}\no
C^1\big(\Sigma_k, \Lambda(\bndro)\big), 1\leq k\leq 4.
\end{equation}
\end{theorem}
\begin{proof} Let $s\in \Sigma$, then by Theorem \ref{l4}, the subset $\cG_{t(s)}\subset\bndro$ is Lagrangian.
The fact that $\cK_{\lambda(s),t(s)},$  $s\in \Sigma,$ is Lagrangian and part {\it iii}) were proved in \cite[Proposition 3.5]{CJM1}.
	
It remains to prove part {\it ii)}. Let $s\in\Sigma$ be fixed. In order to prove \eqref{w11}, we will firstly show an auxiliary result: The map 
\begin{equation}\lb{w17}
\tr_{\cL^{t(s)}}:\ker \left(\cL^{t(s)}_{\mathscr D_{t(s)}}-\lambda(s)\right) \rightarrow \cK_{\lambda(s),t(s)}\cap\cG_{t(s)},
\end{equation}
is one-to-one and onto. Indeed, it is injective since
\begin{equation*}
\text{If}\  \tr_{\cL^{t(s)}}u=0 \text{\ and\ } u\in \ker \left(\cL^{t(s)}_{\mathscr D_{t(s)}}-\lambda(s)\right), \text{\ then\ }  u=0,
\end{equation*}
due to unique continuation principle (cf. \cite[Theorem 3.2.2]{Is}). Next, we prove that \eqref{w17} is surjective. To this end, let us fix an arbitrary $(\phi,\psi)\in\cK_{\lambda(s),t(s)}\cap\cG_{t(s)}$. Since $(\phi,\psi)\in\cK_{\lambda(s),t(s)}$ there exists $u\in \cD^1_{\cL^{t(s)}}(\Omega)$ such that $\tr_{\cL^{t(s)}} u=(\phi,\psi)$. It suffices  to show that $u\in\mathscr D_{t(s)}$. Recall that 
\begin{equation*}
\tr_{\cL^{t(s)}} u\in\cG_{t(s)}=\overline{\tr_{\cL^{t(s)}}\left(\mathscr D_{t(s)}\right)},
\end{equation*}
thus, there exists a sequence $u_n\in\mathscr D_{t(s)},\ n\geq 1$, such that 
\begin{equation*}
\tr_{\cL^{t(s)}}u_n= \left(\gaD u_n, \gaN^{\cL^{t(s)}} u_n\right)\rightarrow \tr_{\cL^{t(s)}} u,\ n\rightarrow \infty,
\end{equation*}
in $\bndro$. For arbitrary $v\in \mathscr D_{t(s)}$ and all $n\geq 1$, one has
\begin{equation}
\begin{split}
&\omega\left((\gaD u_n, \gaN^{\cL^{t(s)}} u_n),(\gaD v, \gaN^{\cL^{t(s)}} v)\right)\\
&\quad=\overline{\langle\gaN^{\cL^{t(s)}} v ,\gaD u_n \rangle}_{{-1/2}}-\langle\gaN^{\cL^{t(s)}} u_n ,\gaD v \rangle_{{-1/2}}\lb{w20w}\\
&\quad=\langle\cL^{t(s)}u_n, v\rangle_{L^2(\Omega)}- \langle u_n,\cL^{t(s)} v\rangle_{L^2(\Omega)}=0,
\end{split}
\end{equation}
since $\cL^{t(s)}_{\mathscr{D}_{t(s)}}$ is self-adjoint. Passing to the limit in \eqref{w20w}, one obtains
\begin{equation}\lb{w20}
\omega\left((\gaD u,\gaN^{\cL^{t(s)}} u),(\gaD v, \gaN^{\cL^{t(s)}} v)\right)=0, \text{\ for all}\ v\in \mathscr D_{t(s)}.
\end{equation}
By the second Green identity \eqref{e13}
\begin{equation}\lb{w21}
\omega\left((\gaD u, \gaN^{\cL^{t(s)}} u),(\gaD v, \gaN^{\cL^{t(s)}} v)\right)=\langle\cL^{t(s)} u, v\rangle_{L^2(\Omega)}- \langle u,\cL^{t(s)} v\rangle_{L^2(\Omega)}.
\end{equation} 
From \eqref{w20} and \eqref{w21} one infers 
\begin{equation}\lb{w22}
\langle\cL^{t(s)} u, v\rangle_{L^2(\Omega)}- \langle u,\cL^{t(s)} v\rangle_{L^2(\Omega)}=0,
\end{equation} 
for all $v\in \mathscr D_{t(s)}$. Combining \eqref{w22} and the fact that $\cL^{t(s)}_{\mathscr{D}_{t(s)}}$ is self-adjoint we conclude that $u\in \mathscr D_{t(s)}$ and thus that map \eqref{w17} is onto.

In order to show that the pair $\left(\cK_{\lambda(s),t(s)},\cG_{t(s)}\right)$ is Fredholm we need to check the following assertions,
\begin{align}
&\dim\left(\cK_{\lambda(s),t(s)}\cap\cG_{t(s)}\right)<\infty\text{\ and\ }\codim\left(\cK_{\lambda(s),t(s)}+\cG_{t(s)}\right)<\infty,\lb{yy15}\\
&\cK_{\lambda(s),t(s)}+\cG_{t(s)} \text{\ is closed in\ }\bndro.\lb{yy16}
\end{align}
The first inequality in \eqref{yy15} follows from the fact that $\cL^{t(s)}_{\mathscr D_{t(s)}}-\lambda(s)$ is a Fredholm operator and that map \eqref{w17} is bijective. To show the second one, we observe that
\begin{align}
\begin{split}
&\codim\left(\cK_{\lambda(s),t(s)}+\cG_{t(s)}\right)=\dim\left(\cK_{\lambda(s),t(s)}+\cG_{t(s)}\right)^{\circ}\lb{w25w}\\
&=\dim\left((\cK_{\lambda(s),t(s)})^{\circ}\cap(\cG_{t(s)})^{\circ}\right)=\dim\left(\cK_{\lambda(s),t(s)}\cap\cG_{t(s)}\right)<\infty, 
\end{split}
\end{align}
because both $\cK_{\lambda(s),t(s)}$ and $\cG_{t(s)}$ are Lagrangian subspaces.
Next we  show \eqref{yy16}. Let us notice that 
\begin{align}
\begin{split}
{\bf K}_{\lambda(s),t(s)}+\mathscr D_{t(s)}=\big\{u\in\cD^1_{\cL^{t(s)}}(\Omega):\cL^{t(s)}u-\lambda(s)x=\cL^{t(s)}v-\lambda(s)v,&\lb{w24}\\
\ \text{in}\ (H^{1}_0(\Omega))^* \text{for some}\ v\in \mathscr D_{t(s)}&\big\},
\end{split}
\end{align}
(a similar equality first appeared in \cite[Proposition 3.5]{BbF95} in the context of strong kernel of $\cL^{t(s)}-\lambda(s)$).  
Utilizing \eqref{w24} and the fact that the operator $\cL^{t(s)}-\lambda(s)$ is Fredholm we will show that ${\bf K}_{\lambda(s),t(s)}+\mathscr {D}_{t(s)} $ is closed in $\cD_{\cL^{t(s)}}^1(\Omega)$. Indeed, if
\begin{equation*}
u_n\in\left({\bf K}_{\lambda(s),t(s)}+\mathscr D_{t(s)}\right) , n\geq 1, \text{\ and\ } u_n\rightarrow u \text{\ in\ } \cD_{\cL^{t(s)}}^1(\Omega),
\end{equation*}
then
\begin{align}
&\cL^{t(s)}u_n-\lambda(s) u_n=\cL^{t(s)}v_n-\lambda(s)v_n,\ \text{for some}\ v_n\in\mathscr D_{t(s)},\ n\geq 1.\lb{w25}
\end{align}
Since $\cL^{t(s)} \in\cB\left(\cD^1_{\cL^{t(s)}}(\Omega), L^2(\Omega)\right)$ then
\begin{align}
&\cL^{t(s)}u_n-\lambda(s)u_n\rightarrow \cL^{t(s)}u-\lambda(s)u,\ n\rightarrow \infty,\ \text{in $L^2(\Omega)$},\lb{w26}
\end{align}
moreover, since the operator  $\cL^{t(s)}_{\mathscr D_{t(s)}}-\lambda(s)$ is Fredholm, one has
\begin{equation}\lb{w27}
\cL^{t(s)}v_n-\lambda(s)v_n\rightarrow \cL^{t(s)}v-\lambda(s)v,\ n\rightarrow \infty,\ \text{in $L^2(\Omega)$}, 
\end{equation}
for some $v\in \mathscr  D_{t(s)}$. Combining \eqref{w25}, \eqref{w26}, \eqref{w27}, one obtains
\begin{equation*}
\cL^{t(s)}u-\lambda(s)u=\cL^{t(s)}v-\lambda(s)v,
\end{equation*}
hence, $u\in\left( {\bf K}_{\lambda(s),t(s)}+\mathscr D_{t(s)}\right)$.
Next, the linear operator 
\begin{equation*}
\tr_{\cL^{t(s)}}: \big({\bf K}_{\lambda(s),t(s)}+\mathscr D_{t(s)} \big)\rightarrow \bndro,
\end{equation*}
acting from the Banach space $\big({\bf K}_{\lambda(s),t(s)}+\mathscr D_{t(s)} \big)$ equipped with $\cD_{\cL^{t(s)}}^1(\Omega)-$norm to the Hilbert space $\bndro$, is bounded. Furthermore, its range
\begin{equation*}
\tr_{\cL^{t(s)}} \big({\bf K}_{\lambda(s),t(s)}+\mathscr D_{t(s)} \big)=\left(\cK_{\lambda(s),t(s)}+\tr_{\cL^{t(s)}}(\mathscr D_{\cL^{t(s)}})\right)
\end{equation*}
has finite codimension. Therefore, by  \cite[Corollary 2.3]{GGK1}, the subset 
\begin{equation*}
\cK_{\lambda(s),t(s)}+\tr_{\cL^{t(s)}}(\mathscr D_{\cL^{t(s)}})\subset \bndr
\end{equation*}
is closed. Hence, $\left(\cK_{\lambda(s),t(s)}+\cG_{t(s)} \right)$ is also closed.
\end{proof}
\subsection{The Maslov and Morse indices}
We are ready to state the principal result of this section.  In the following theorem we consider a one-parameter family of self-adjoint extensions of uniformly elliptic operators. One of our main assumptions is that each operator from this family is semibounded from below. This assumption is satisfied for all standard self-adjoint extensions such as the Dirichlet, Neumann, Robin, and periodic Laplace operators. However, it is not evident that all self-adjoint extensions of an elliptic operator  are  necessarily semibounded from below (cf. \eqref{smb}). Next, we notice that the relations between the Maslov and Morse indices have been extensively studied by many authors cf., e.g., \cite[Theorem 5.1]{BbF95}, \cite[Theorem 2.4, Theorem 2.5]{DJ11}, \cite[Theorem 1.5]{BZ3}, \cite[Theorem 4.5.4]{BZ1}, \cite[Theorem 1.3, Theorem 1.4]{CJLS}, \cite[Theorem 1]{CJM1}, \cite[Theorem 1.5]{HS}. The work in this direction was originated in \cite{BbF95}, where the authors considered the Lagrangian planes formed by the abstract traces of strong solutions (i.e., by the abstract traces of the kernels of adjoint operators) assuming that the domain of the adjoint operator is fixed. Later this assumption was relaxed in a series of works \cite{BZ3, BZ1, BZ2, BZ4} by considering only those extensions whose domains are contained in a fixed subspace. We, on the other hand, consider the Lagrangian planes formed by the weak traces of weak solutions which allows us to reduce regularity assumptions for the domains of self-adjoint extensions. 
\begin{theorem}\lb{w33}
Assume Hypothesis \ref{q1} and recall the differential expressions \eqref{r7}.   Let $\mathscr D_t\subset \cD_{\cL^t}^1(\Omega)$, $ t\in\cI$, and assume that the linear operator $\cL^t_{\mathscr D_t}$ acting in $L^2(\Omega)$ and given by
\begin{equation*}
\cL^t_{\mathscr D_t} u:=\cL^tu,\ u\in \dom\left(\cL^t_{\mathscr D_t}\right):={\mathscr D_t},
\end{equation*}
is self-adjoint  with the property
\begin{equation*}
\spec_{ess}\left(\cL^t_{\mathscr{D}_t}\right)\cap(-\infty, 0]=\emptyset, \text{\ for all\ } t\in\cI.
\end{equation*}
Assume further that there exists $\lambda_{\infty}<0$, such that
\begin{equation}\lb{smb}
\ker\left(\cL^t_{\mathscr D_t}-\lambda\right)=\{0\}, \text{\ for all\ } \lambda\leq \lambda_{\infty}, t\in\cI.
\end{equation}
Suppose, finally,  that the path
\begin{equation*}
t\mapsto \cG_t:=\overline{\tr_{\cL^t}\big(\mathscr D_t\big)},\ t\in \cI,
\end{equation*}
is contained in $C\left(\cI,\Lambda\left(\bndro\right)\right)$.

Then
\begin{equation}\lb{w32}
\mo\left(\cL^{\alpha}_{\mathscr D_{\alpha}}\right)-\mo\big(\cL^{\beta}_{\mathscr D_{\beta}}\big)=\mi\left((\cK_{0,t},\cG_t)|_{t\in\cI}\right),
\end{equation}
the Lagrangian plane $\cK_{0,t}$ is defined by \eqref{w15}. 
\end{theorem}

\begin{proof}
We will compute the Maslov index of the path $s\mapsto (\cK_{\lambda(s),t(s)},\cG_{t(s)})$ on each interval $\Sigma_1, \Sigma_2, \Sigma_3, \Sigma_4$  parameterizing the respective sides of the boundary of the square $[\lambda_{\infty},0]\times [\alpha,\beta]$, see Figure 1,  and use a catenation argument to determine the Maslov index on $\Sigma$. To this end we split the proof into four parts. 	

{\bf Step 1.} In this step we show that
\begin{align}
&\mi\left(\cK_{{\alpha},\lambda(s)}|_{s\in\Sigma_1},\cG_{\alpha}\right)=-\mo\left(\cL_{\alpha}\right).\lb{w34}
\end{align}
The proof goes along the lines of the argument in \cite{BbF95}, where a variant of \eqref{w34} is established in the context of strong kernels, abstract trace maps, and fixed domains of the maximal operators. In order to obtain \eqref{w34} in our setting, we intend to prove that each crossing on $\Sigma_1$ is negative (hence, non-degenerate), and use \eqref{ab17} to verify that geometric multiplicities of negative eigenvalues of $\cL^{\alpha}_{\mathscr D_{\alpha}}$ add up to minus the Maslov index. 
Let $s_*\in (\lambda_{\infty}, 0)$ be a conjugate point, i.e. $\cK_{\lambda(s_*),\alpha}\cap\cG_{\alpha}\not=\{0\}$. There exists a small neighbourhood $\Sigma_{s_*}\subset(\lambda_{\infty}, 0)$ of $s_*$ and a family of operators $R_{s+s_*}$ so that
\begin{equation}\lb{w35}
(s+s_*)\mapsto R_{(s+s_*)}\ \text{in}\  C^1\big(\Sigma_{s_*}, \cB(\cK_{\lambda(s_*),\alpha}, (\cK_{\lambda(s_*),\alpha})^{\perp})\big),\ R_{s_*}=0,
\end{equation} 
and 
\begin{equation}\lb{w36}
\cK_{\lambda(s),\alpha}=\{(\phi,\psi)+R_{s+s_*}(\phi,\psi)\big| (\phi,\psi)\in \cK_{\lambda(s_*),\alpha} \}\ \text{for all }\ (s+s_*)\in \Sigma_{s_*},
\end{equation}
see  \cite[Lemma 3.8]{CJLS}. 
Let us fix $(\phi_0,\psi_0)\in \cK_{\lambda(s_*),\alpha}$ and consider the family 
\begin{equation*}
(\phi_s,\psi_s):=(\phi_0,\psi_0)+R_{(s+s_*)}(\phi_0,\psi_0)\text{\ with small\ }|s|.
\end{equation*}
Since $(\phi_s,\psi_s)\in\cK_{\lambda(s),\alpha}$, by the unique continuation principle (cf. \eqref{UCP}), there exists a unique $u_s\in\cD_{\cL^{\alpha}}^1(\Omega)$ such that
\begin{equation*}
\tr_{\cL^{\alpha}}u_s=(\phi_s,\psi_s)\text{\ for small\ }|s|.
\end{equation*}
Next, using the second Green identity \eqref{e13}, we calculate:
\begin{align*}
&\omega\left((\phi_0,\psi_0), R_{(s+s_*)}(\phi_0,\psi_0)\right)=\overline{\langle\psi_s ,\phi_0 \rangle}_{{-1/2}}-\langle\psi_0 , \phi_s \rangle_{{-1/2}}\no\\
&\quad=\langle\cL^{\alpha} u_0, u_s-u_0\rangle_{L^2(\Omega)}- \langle u_0,\cL^{\alpha}( u_s-u_0)\rangle_{L^2(\Omega)}\no\\
&\quad=\langle(\cL^{\alpha} -\lambda(s_*))u_0, u_s-u_0\rangle_{L^2(\Omega)}- \langle u_0,(\cL^{\alpha}-\lambda(s_*))( u_s-u_0)\rangle_{L^2(\Omega)}\no\\
&\quad=- \langle u_0,(\cL^{\alpha}-\lambda(s_*)) u_s\rangle_{L^2(\Omega)}\no\\
&\quad=- \langle u_0,(\cL^{\alpha}-\lambda(s+s_*)) u_s\rangle_{L^2(\Omega)}+ \langle u_0,(\lambda(s_*)-\lambda(s+s_*)) u_s\rangle_{L^2(\Omega)}\no\\
&\quad=-\langle u_0,s u_s\rangle_{L^2(\Omega)}.\
\end{align*}
The mapping $s\mapsto u_s\in H^1(\Omega)$ is continuous at $0$, since, using the standard elliptic estimate in Lemma \ref{q3} given below, 
\begin{align}
\begin{split}
\|u_s-u_0\|_{H^1(\Omega)}&\leq C\big\|\tr_{\cL^{\alpha}}(u_s-u_0)\big\|_{\bndro}\lb{w37}\\
&=C \|(\phi_s-\phi_0,\psi_s-\psi_0)\|_{\bndro},
\end{split}
\end{align}
where $C>0$ does not depend on $s$. We proceed by evaluating  the crossing form from Definition \ref{def21} {\it (ii)} 
\begin{align}
&\cQ_{s_*,\cG_{\alpha}}\left((\phi_0,\psi_0),(\phi_0,\psi_0)\right):=\frac{d}{ds}\omega\left((\phi_0,\psi_0), R_{(s+s_*)}(\phi_0,\psi_0)\right)\big|_{s=0}\no\\
&=\lim\limits_{s\rightarrow 0} \frac{\omega\left((\phi_0,\psi_0), R_{(s+s_*)}(\phi_0,\psi_0)\right)}{s}=\lim\limits_{s\rightarrow 0} \frac{-\langle u_0,s u_s\rangle_{L^2(\Omega)}}{s}=-\|u_0\|^2_{L^2(\Omega)}.\no
\end{align}
Therefore, the crossing form is negative definite at all conjugate points on $[\lambda_{\infty}, 0]$ and, using \eqref{ab17}, one obtains
\begin{align}
&\mi\left(\cK_{\lambda(s), \alpha}|_{s\in\Sigma_1},\cG_{\alpha}\right)=-n_-\left(\cQ_{\lambda_{\infty},\cG_{\alpha}}\right)+\sum\limits_{\substack{\lambda_{\infty}<s<0:\\
	\cK_{\lambda(s),\alpha}\cap \cG_{\alpha}\not=\{0\}	}}\sign\ \cQ_{s,\cG_{\alpha}}\no\\
&\quad+n_+(\cQ_{0,\cG_{\alpha}})=-\sum\limits_{\lambda_{\infty}\leq s< 0}\dim \ker\left(\cL^{\alpha}_{\mathscr D_{\alpha}}-\lambda(s)\right)=-\mo\left(\cL^{\alpha}_{\mathscr D_{\alpha}}\right),\lb{w38}
\end{align}
where we employed $n_+\left(\cQ_{0,\cG_{\alpha}}\right)=0$, and the fact that there are no crossings to the left of $\lambda_{\infty}$.

{\bf Step 2.} A similar computation can be carried out in case $s\in\Sigma_3$, leading to the analog of \eqref{w34}, 
\begin{align}
&\mi\left(\cK_{{\alpha},\lambda(s)}|_{s\in\Sigma_3},\cG_{\alpha}\right)=\mo\left(\cL_{\alpha}\right).\lb{w39}
\end{align}

{\bf Step 3.} Since, by assumptions,  $\ker(\cL^t_{\mathscr D_t}-\lambda)=\{0\}$ for all $\lambda\leq \lambda_{\infty}, t\in\cI$, there are no crossings on $\Sigma_4$, therefore, the Maslov index  vanishes on this interval 
\begin{equation} 
\mi\big((\cK_{t(s),\lambda_{\infty}},\cG_{t(s)})|_{s\in\Sigma_4}\big)=0.\lb{w40}
\end{equation}

{\bf Step 4.} In this step we will combine \eqref{w34}, \eqref{w39}, \eqref{w40}, and the homotopy invariance of the Maslov index to obtain \eqref{w32}. Since the curve $\Gamma$, cf., \eqref{w17w}, can be contracted to a point, one has
\begin{equation}\lb{w41}
	\mi\left((\cK_{t(s),\lambda{(s)}},\cG_{t(s)})|_{s\in\Sigma}\right)=0.
\end{equation}
On the other hand, due to the catenation property of the Maslov index,
\begin{align}
\begin{split}
&\mi\left((\cK_{t(s),\lambda{(s)}},\cG_{t(s)})|_{s\in\Sigma}\right)=\mi\left((\cK_{t(s),\lambda{(s)}},\cG_{t(s)})|_{s\in\Sigma_1}\right)\lb{w42}\\
&\quad+\mi\left((\cK_{t(s),\lambda{(s)}},\cG_{t(s)})|_{s\in\Sigma_2}\right)+\mi\left((\cK_{t(s),\lambda{(s)}},\cG_{t(s)})|_{s\in\Sigma_3}\right)\\
&\quad+\mi\left((\cK_{t(s),\lambda{(s)}},\cG_{t(s)})|_{s\in\Sigma_4}\right).
\end{split}
\end{align} 
Combining \eqref{w34}, \eqref{w39}, \eqref{w40}, \eqref{w41}, \eqref{w42}, one obtains \eqref{w32}.
\end{proof}

\begin{lemma}\lb{q3} Assume Hypothesis \ref{q1}. Then there exists a positive constant $C>0$ independent of $s$ such that if $u\in H^1(\Omega)$ is a weak solutions to $\cL^su=0, s\in\cI$ then
\begin{equation}\lb{q5}
	\|u\|_{H^1(\Omega)}\leq C \big\|\tr_{\cL^s}u\big\|_{\bndro}, \text{\ for all\ }s\in\cI.
\end{equation}
\end{lemma}
\begin{proof} Recall the function space \eqref{cc4}. For arbitrary $u\in \cD^1_{\cL^s}(\Omega),$ $s\in\cI$, one has 
\begin{equation}\lb{q6}
\mathfrak{l}^s[u,u]=\langle\cL^su,u\rangle_{L^2(\Omega)}+\langle \gaN^{\cL^s,1}u,\gaD u\rangle_{H^{-1/2}(\partial\Omega)},
\end{equation}
where
\begin{align}
\begin{split}
\mathfrak{l}^s[u,v]&=\sum_{j,k=1}^{n}\langle a^s_{jk}\partial_ku,\partial_j v\rangle_{L^2(\Omega)}+\sum_{j=1}^{n}\langle a^s_j\partial_ju,v\rangle_{L^2(\Omega)}\\
&+\sum_{j=1}^{n}\langle u,a^s_j\partial_jv\rangle_{L^2(\Omega)}+\langle a^su,v\rangle_{L^2(\Omega)},\ u,v\in H^1(\Omega), s\in\cI.\lb{q112}
\end{split}
\end{align}
Our immediate objective is to show that the inequality,
\begin{align}
\begin{split}
\| u\|_{H^1(\Omega)}^2\leq C\big(\| u\|_{L^2(\Omega)}^2+\|\cL^su\|_{L^2(\Omega)}^2+\|\tr_{\cL^s}u\|_{\bndro}^2\big)\lb{qq13},
\end{split}
\end{align}
holds for some $C>0$ independent of $s$, and all $s\in\cI$. To this end, we first notice that by the elliptic property \eqref{q11}, one has
\begin{align}
&\sum_{j,k=1}^{n}\langle a^s_{jk}\partial_ku,\partial_j u\rangle_{L^2(\Omega)}\geq c\|\nabla u\|^2_{L^2(\Omega)}.\lb{q111}
\end{align}
Second, using \eqref{q6} and \eqref{q112} we obtain
\begin{align}
&\left|\sum_{j,k=1}^{n}\langle a^s_{jk}\partial_ku,\partial_j u\rangle_{L^2(\Omega)}\right|\leq \left|\sum_{j=1}^{n}\langle a^s_j\partial_ju,u\rangle_{L^2(\Omega)}
\right|+\left|\sum_{j=1}^{n}\langle u,a_j^s\partial_ju\rangle_{L^2(\Omega)}\right|\lb{q7}\\
&\quad +\left|\langle a^su,u\rangle_{L^2(\Omega)}\right|+\left|\langle\cL^su,u\rangle_{L^2(\Omega)}\right|+\left|\langle \gaN^{\cL^s,1}u,\gaD u\rangle_{H^{-1/2}(\partial\Omega)}\right|\lb{q8}.
\end{align}
Next, the Cauchy--Schwarz inequality together with \eqref{q7}, \eqref{q8} yield
\begin{align}
&\left|\sum_{j,k=1}^{n}\langle a^s_{jk}\partial_ku,\partial_j u\rangle_{L^2(\Omega)}\right|\leq 2\sup\limits_{\substack{s\in[\alpha, \beta], x\in\overline{\Omega},\\ 1\leq j\leq n}}\|a_j^s(x)\|_{\C^m}\|\nabla u\|_{L^2(\Omega)}\| u\|_{L^2(\Omega)}\lb{qq77}\\
&\quad +\sup\limits_{s\in[\alpha, \beta], x\in\overline{\Omega}}\|a^s(x)\|_{\C^{m\times m}}\|u\|^2_{L^2(\Omega)}+ \| \cL^s u\|_{L^2(\Omega)}\| u\|_{L^2(\Omega)}\lb{qq88}\\
&\quad + \|\gaN^{\cL^s,1}u\|_{H^{-1/2}(\partial\Omega)}\|\gaD u\|_{H^{1/2}(\partial\Omega)}.\lb{qq99}
\end{align}
Finally, the inequalities 
\begin{align}
&\|\nabla u\|_{L^2(\Omega)}\| u\|_{L^2(\Omega)}\leq  \frac{\|u\|_{L^2(\Omega)}^2}{2\varepsilon^2}+\frac{\varepsilon^2\|\nabla u\|_{L^2(\Omega)}^2}{2},\ \text{with\  }\varepsilon>0 \text{\ small enough},\no\\
& \| \cL^s u\|_{L^2(\Omega)}\| u\|_{L^2(\Omega)}\leq \frac1 2\left(\|u\|_{L^2(\Omega)}^2+\| \cL^s u\|_{L^2(\Omega)}^2\right),\no\\
&\|\gaN^{\cL^s,1}u\|_{H^{-1/2}(\partial\Omega)}\|\gaD u\|_{H^{1/2}(\partial\Omega)}\leq \frac1 2\left(\|\gaN^{\cL^s,1}u\|_{H^{-1/2}(\partial\Omega)}^2+\|\gaD u\|_{H^{1/2}(\partial\Omega)}^2\right),\no
\end{align} 
together with \eqref{qq77}-\eqref{qq99} imply
\begin{align}
\begin{split}
&\left|\sum_{j,k=1}^{n}\langle a^s_{jk}\partial_ku,\partial_j u\rangle_{L^2(\Omega)}\right|\leq C_1\big(\| u\|_{L^2(\Omega)}^2+\|\cL^su\|_{L^2(\Omega)}^2\lb{q9}\\
& \quad +\|\gaN^{\cL^s}u\|_{H^{-1/2}(\partial\Omega)}^2+\|\gaD u\|_{H^{1/2}(\partial\Omega)}^2\big)+\frac{c}{2}\|\nabla u\|_{L^2(\Omega)}^2, s\in\cI,
\end{split}
\end{align}
where $0<C_1=C_1(a, a_j, n,\Omega)$, and $c>0$ is from \eqref{q11}. Combining \eqref{q111} and \eqref{q9}, one infers \eqref{qq13}.

We intend to derive from \eqref{qq13} yet a stronger inequality, 
\begin{align}
\begin{split}
\|u\|_{H^1(\Omega)}^2\leq C\left(\|\cL^su\|_{L^2(\Omega)}^2+\|\tr_{\cL^s}u\|_{\bndro}^2\right)\lb{q14},\ s\in\cI,
\end{split}
\end{align}
which trivially implies \eqref{q5}.
We prove \eqref{q14} by contradiction: Assume that there exist
\begin{equation*}
s_n\in \Sigma, u_n\in \cD^1_{\cL^{s_n}}(\Omega), n\geq 1,
\end{equation*}
such that 
\begin{align}
\begin{split}
\|u_n\|_{H^1(\Omega)}^2>n\left(\|\cL^{s_n}u_n\|_{L^2(\Omega)}^2+\|\tr_{\cL^{s_n}}u_n\|_{\bndro}^2\right)\lb{q15},\ n\geq 1.
\end{split}
\end{align}
Without loss of generality we may assume that 
\begin{equation*}
s_n\rightarrow s_0, n\rightarrow \infty,\text{\ and that\ }\|u_n\|_{L^2(\Omega)}=1, n\geq1.
\end{equation*}
It follows from \eqref{qq13} and \eqref{q15} that the sequence $\{u_n:n\geq 1\}$ is bounded in $H^1(\Omega)$, and therefore that 
\begin{equation}\lb{q16}
\|\cL^{s_n}u_n\|_{L^2(\Omega)}\rightarrow 0\text{\ and\ }\|\tr_{\cL^{s_n}}u_n\|_{\bndro}\rightarrow 0,\text{\ as\ }n\rightarrow\infty.
\end{equation}
Passing to a subsequence if necessary, we have the weak convergence,
\begin{equation}\lb{qw1}
u_n\rightharpoonup u_0, \text{\ as\ } n\rightarrow\infty\text{\ in\ }H^1(\Omega).
\end{equation}
Since $H^1(\Omega)$ is compactly embedded into $L^2(\Omega)$, we conclude that $u_n\rightarrow u_0$, $n\rightarrow\infty$ in $L^2(\Omega)$. 
We claim that
\begin{equation}\lb{q17}
u_0\in\cD^1_{\cL^{s_0}}(\Omega),\ \cL^{s_0}u_0=0,
\end{equation} 
\begin{equation}\lb{q18}
\tr_{\cL^{s_0}}u_0=0.
\end{equation}
Granted \eqref{q17},\eqref{q18}, we notice that the unique continuation principle yields $u_0=0$, which in turn, contradicts the fact that $\|u_0\|_{L^2(\Omega)}=1$, and finishes the proof of \eqref{q14}. 

 It remains to prove the claim. First, we prove \eqref{q17}. For arbitrary $\varphi\in C_0^{\infty}(\Omega)$ the second Green identity yields \eqref{e13}
\begin{align}\lb{q19}
\langle u_n, \cL^{s_n}\varphi\rangle_{L^2(\Omega)}=\langle \cL^{s_n} u_n, \varphi\rangle_{L^2(\Omega)},\ n\geq 1.
\end{align} 
On the other hand, since $\gaD\varphi=0, \gaN^{\cL^{s_n}}\varphi=0$, the first Green identity \eqref{e14} yields
\begin{align}\lb{q20}
\langle \cL^{s_n} u_n, \varphi\rangle_{L^2(\Omega)}=\mathfrak{l}^{s_n}[u_n,\varphi],\ n\geq 1.
\end{align} 
Furthermore, using the first limit in \eqref{q16} and \eqref{qw1}
we obtain
\begin{equation}\lb{q21}
\mathfrak{l}^{s_n}[u_n,\varphi]\rightarrow\mathfrak{l}^{s_0}[u_0,\varphi]\text{\and\ }\langle \cL^{s_n} u_n,\varphi\rangle_{L^2(\Omega)}\rightarrow 0,\text{\ as\ } n\rightarrow\infty. 
\end{equation}
Combining \eqref{q20} and \eqref{q21} we obtain 
\begin{equation}\lb{q22}
0=\mathfrak{l}^{s_0}[u_0,\varphi],\ \text{for arbitrary\ }  \varphi\in C^{\infty}_0(\Omega),
\end{equation}
moreover, by the first Green identity
\begin{equation*}
\mathfrak{l}^{s_0}[u_0,\varphi]=\langle u_0, \cL^{s_0}\varphi\rangle_{L^2(\Omega)}\text{\ for arbitrary\ }  \varphi\in C^{\infty}_0(\Omega).
\end{equation*}
Hence, \eqref{q17} holds.

It remains to check \eqref{q18}. First, the equality $\gaD u_0=0$ holds since, by using the second limit in \eqref{q16},
\begin{equation*}
\gaD u_n\rightharpoonup \gaD u_0\text{\ and\ }\gaD u_n\rightarrow 0,\text{as } n\rightarrow\infty\text{\ in\ }H^{1/2}(\partial\Omega).
\end{equation*}
Next, by the first Green identity
\begin{equation}\lb{q23}
\mathfrak{l}^{s_n}[u_n,f]=\langle \cL^{s_n} u_n, f\rangle_{L^2(\Omega)}+ \langle \gaN^{\cL^{s_n}}u_n,\gaD f\rangle_{H^{-1/2}(\partial\Omega)},\ n\geq 1,
\end{equation}
for arbitrary $f\in H^1(\Omega)$. The left hand-side of \eqref{q23} tends to $\mathfrak{l}^{s_0}[u_0,f]$ (due to the weak convergence of $u_n$), whereas by \eqref{q16}, the right hand-side converges to $0$, as $n\rightarrow 0$, implying  
\begin{equation}
\mathfrak{l}^{s_0}[u_0,f]=0, \text{\ for all\ }f\in H^1(\Omega).
\end{equation}
The first Green identity and \eqref{q17} yield
\begin{equation}\lb{q24}
0=\mathfrak{l}^{s_0}[u_0,f]= \langle \gaN^{\cL^{s_0}}u_0,\gaD f\rangle_{H^{-1/2}(\partial\Omega)}, \text{\ for all\ }f\in H^1(\Omega).
\end{equation}
Finally, since $\gaD:H^1(\Omega)\rightarrow H^{1/2}(\partial\Omega)$ is onto, \eqref{q24} implies $\gaN^{\cL^{s_0}}u_0=0$ as required.
\end{proof}
In the remaining part of this section we illustrate several applications of the general formula \eqref{w32}, that is, how several known and some unknown results can be derived from this formula.

\subsection{The spectral flow and the Maslov index}
Assume hypotheses of Theorem \ref{w33}. Then the spectral flow of the one-parameter operator family $\left\{\cL^t_{\mathscr D_t}\right\}_{t=\alpha}^{\beta}$ is defined as follows: There exists a partition  $\alpha=t_0<t_1<\cdots<t_N=\beta,$ and $N$ intervals $[a_{\ell},b_{\ell}],\ a_{\ell}<0<b_{\ell},\ 1\leq \ell\leq N$ such that
\begin{equation}\lb{r1}
a_{\ell}, b_{\ell}\not\in\spec\left(\cL^t_{\mathscr D_t}\right),\text{\ for all\ }t\in [t_{\ell-1},t_{\ell}], \ 1\leq \ell\leq N.
\end{equation}
The spectral flow through $\lambda=0$ is defined by the formula
\begin{equation}\lb{r2}
	\spflow\left(\{\cL^t_{\mathscr D_t}\}_{t=\alpha}^{\beta}\right):=\sum\limits_{\ell=1}^N\sum\limits_{a_{\ell}\leq\lambda<0}\left(\dim\ker\left(\cL^{t_{\ell-1}}_{\mathscr D_{t_{\ell-1}}}-\lambda\right)-\dim\ker\left(\cL^{t_{\ell}}_{\mathscr D_{t_{\ell}}}-\lambda\right)\right).
\end{equation}
It can be shown that $\spflow\left(\left\{\cL^t_{\mathscr D_t}\right\}_{t=a}^b\right)$ does not depend on the choice of partition of the interval $[\alpha,\beta]$ (cf., \cite[Appendix]{BZ1}). In fact, since $\{\cL^t_{\mathscr D_t}\}_{t=a}^b$ is uniformly bounded from below (with lower bound $\lambda_{\infty}$), we can assume that $[\lambda_{\infty},0]\subset[a_{\ell},b_{\ell}],\ 1\leq \ell\leq N$.  In this case \eqref{r2} reads
\begin{equation}\lb{r3}
\spflow\left(\left\{\cL^t_{\mathscr D_t}\right\}_{t=\alpha}^{\beta}\right)=\sum\limits_{\ell=1}^N\left(\mo(\cL^{t_{\ell-1}}_{\mathscr D_{t_{\ell-1}}})-\mo(\cL^{t_{\ell}}_{\mathscr D_{t_{\ell}}})\right).
\end{equation}
Combining \eqref{w32} and \eqref{r3}, one obtains
\begin{equation}\lb{367new}
\spflow\left(\left\{\cL^t_{\mathscr D_t}\right\}_{t=\alpha}^{\beta}\right)=\mi\left((\cK_{0,t},\cG_t)|_{t\in\cI}\right).
\end{equation} 
By rescaling, a similar formula holds for the spectral flow through any point $\lambda_0\in\bbR$ with $\cK_{0,t}$ replaced by  $\cK_{\lambda_0,t}$. Of course, relations between the spectral flow and the Maslov index of this type have been obtained in many important papers, cf., e.g., \cite{BZ3}, \cite{BZ1}, \cite{BZ2}, \cite{BZ4} \cite{CLM}, \cite{F04}, \cite{KL}, \cite{Nic}, \cite{rs93}, \cite{RoSa95}, \cite{SW}. We stress, however, that in our case $\mathscr{D}_t\subset H^1(\Omega)$, $t\in[\alpha,\beta]$, and that we use the ``usual" PDE trace operators as oppose to the abstract traces acting into the quotient spaces.
\subsection{Spectra of elliptic operators on deformed domains and the Maslov index}
In this section we revise a main result, Theorem 1, from \cite{CJM1}, and place it in a general framework of Theorem \ref{w33}. Given a second order elliptic operator $\cL$ on $\Omega$ and a one-parameter family of diffeomorphisms, \cite[Theorem 1]{CJM1} expresses the difference of  Morse indices of $\cL$ and its pullback in terms of the Maslov index. 
 
Let $\Omega_0\subset \bbR^n$ be a bounded open set with smooth boundary, let $\varphi_t:\bbR^n\rightarrow\bbR^n,\ t\in[0,1],$ be one-parameter family of diffeomorphisms, such that the mapping $t\mapsto\varphi_t$ is contained in $C^1([0,1], L^{\infty}(\Omega_0,\R^n))$, and $\varphi_0={\text{Id}}_{\Omega_0}$. Let us denote
\begin{equation*}
\Omega_t:=\{\varphi_t(x):x\in\Omega_0\},\ \Omega:=\cup_{0\leq t\leq 1}\Omega_t.
\end{equation*}
Suppose that  the coefficients of the second order differential operator satisfy
\begin{align}
& \cA:=\{a_{jk}\}_{1\leq j,k \leq n}\in C^{\infty}(\Omega, \C^{n\times n}), \cA=\overline{\cA^{\top}},\ \lb{r22}\\
& a_{jk}(x)\xi_k\overline{\xi_j}\geq c \sum\limits_{j=1}^n|\xi_j|^2,\ \text{for all}\ \xi=(\xi_j)_{j=1}^n\in\C^n, x\in\overline{\Omega}, \text{and some}\ c>0,\lb{3.71} \\
& b_{j}\in C^{\infty}(\Omega),\ 1\leq j \leq n, B:=(b_1,\cdots, b_n)^{\top},\lb{r23}\\
& q \in C^{\infty}(\Omega), q(x)\in\bbR,\ x\in\Omega,\lb{r24}
\end{align}
and fix a subspace $\cX_0$ such that
\begin{align}
&H^1_0(\Omega_0)\subset{\cX_0}\subset H^1(\Omega_0),\ {\cX_0}\text{\ is closed subset of } H^1(\Omega_0).\lb{r6}
\end{align}
Using \eqref{r22}-\eqref{r24} we construct a family $\{\cL_{\cX_0}^t\}$ of operators in $L^2(\Omega_0)$ as follows. Let us define the one-parameter family of sesquilinear forms on $H^1(\Omega_t)$,
\begin{align}
\begin{split}
\mathfrak{l}^t[u,v]:&=\langle \cA\nabla u,\nabla v\rangle_{L^2(\Omega_t)}+\langle B\nabla u,v\rangle_{L^2(\Omega_t)}\lb{r7r}\\
&+\langle u,B\nabla v\rangle_{L^2(\Omega_t)}+\langle qu,v\rangle_{L^2(\Omega_t)},\ u,v\in H^1(\Omega_t), t\in[0,1].
\end{split}
\end{align}
Changing variables in the right hand-side of \eqref{r7r}, we arrive at
\begin{align}
\begin{split}
\wti{\mathfrak{l}}^t[\wti u,\wti v]:&=\langle \cA^t\nabla \wti u,\nabla \wti v\rangle_{L^2(\Omega_0)}+\langle B^t\nabla \wti u,\wti v\rangle_{L^2(\Omega_0)}\lb{r8}\\
&+\langle \wti u,B^t\nabla \wti v\rangle_{L^2(\Omega_0)}+\langle q^t\wti u,\wti v\rangle_{L^2(\Omega_0)},\ \wti u,\wti v\in H^1(\Omega_0), t\in[0,1],
\end{split}
\end{align}
where the functions on $\Omega_0$ satisfy
\begin{align}
&\cA^t:=\det(D\varphi_t)(D\varphi_t^{\top})^{-1}[\cA\circ\varphi_t](D\varphi_t)^{-1},\ B^t:= \det(D\varphi_t)(D\varphi_t^{\top})^{-1}[B\circ\varphi_t],\no\\
& q^t:=\det(D\varphi_t)q\circ\varphi_t,\ \wti u:= u\circ\varphi_t,\ \wti v:= v\circ\varphi_t,\ t\in[0,1].\no
\end{align}
If $t\in[0,1]$ then the form
\begin{equation}
\wti{\mathfrak{l^t}}:L^2(\Omega_0)\times L^2(\Omega_0)\rightarrow \C,\ \dom(\wti{\mathfrak{l}}):=\cX_0, \lb{yy17}
\end{equation}
is closed and bounded from below.
Hence, by \cite[Theorem 2.8]{EE89} there exists a unique self-adjoint operator $\cL^t_{\cX_0}$ acting in $L^2(\Omega_0)$, such that
\begin{equation}\lb{r10}
\wti{\mathfrak{l}}^t[u,v]=\langle\cL^t_{{\cX_0}}u,v\rangle_{L^2(\Omega_0)} \text{\ for all\ }u\in\dom(\cL^t_{{\cX_0}}),\ v\in{\cX_0}.
\end{equation}
Moreover, by \cite[Lemma 4.1 and Proposition C.1 ]{CJM1} there exist positive constants $C_1,C_2$ such that
\begin{equation}\lb{yy18}
\wti{\mathfrak{l^t}}[f,f]\geq C_1\|f\|_{H^1(\Omega_0)}^2-C_2\|f\|_{L^2(\Omega_0)}^2.
\end{equation}
Since the form domain of $\wti{\mathfrak{l^t}}$ is compactly embedded into $L^2(\Omega_0)$, the spectrum of $\cL^t_{\cX_0}$ is purely discrete. Since $\wti{\mathfrak{l^t}}, t\in[0,1]$, is uniformly bounded from below in $L^2(\Omega_0)$ there exists $\lambda_{\infty}$ such that 
\begin{equation}\lb{yy19}
\ker\left(\cL^t_{\cX_0}-\lambda\right)=\{0\}\text{\ for all\ }\lambda\leq\lambda_{\infty}, t\in[0,1].
\end{equation}
We notice that $\cL^t_{{\cX_0}}$ is a self-adjoint extension of $\cL^t_{min}$ given by the closure of 
\begin{align}
L^t u:=- \text{div}\cA^t\nabla u+ B^t\nabla u-\nabla\cdot\overline{B^t}u+q^tu, \dom(L^t):=C^{\infty}_0(\Omega_0). 
\end{align} 
\begin{proposition}\lb{r11} Let $\cL^t_{{\cX_0}},$ $t\in[0,1]$ be the one-parameter family of self-adjoint operators defined by \eqref{r10}. Then
\begin{align}
\begin{split}
\overline{\tr_{\cL^t}\left(\dom(\cL^t_{{\cX_0}})\right)}=\{(f,g)\in\bndro:\ \ & \lb{r12}\\
f\in\gaD({\cX_0}), \langle g, \gaD u \rangle_{-1/2}=0 \text{\ for all\ }u\in {\cX_0}\} &,\ t\in[0,1]
\end{split}
\end{align}
where the bar in the left-hand side denotes closure in $\bndro$. Hence, the right-hand side of \eqref{r12} is a Lagrangian plane.
\end{proposition}
\begin{proof}
 Let us fix $t\in[0,1]$. The right-hand side of \eqref{r12} is isotropic. By Theorem \ref{l4}, $\overline{\tr_{\cL^t}(\dom(\cL_{t,{\cX_0}}))}$ is Lagrangian, hence, it suffices to show that $\tr_{\cL^t}(\dom(\cL^t_{{\cX_0}}))$ is contained in the right-hand side of \eqref{r12}. The first Green identity yields
\begin{align}
&\wti{\mathfrak{l}}^t[u,v]=\langle\cL^t u, v\rangle_{L^2(\Omega)}+ \langle {\gaN^{\cL^t}}u,\gaD v\rangle_{-1/2},\ u\in\cD^1_{\cL^t}(\Omega), v\in H^1(\Omega).\lb{r13}
\end{align}
On the other hand,
\begin{equation}\lb{r14}
\wti{\mathfrak{l}}^t[u,v]=\langle\cL^tu,v\rangle_{L^2(\Omega_0)},  \text{\ for all\ }u\in\dom(\cL^t_{{\cX_0}}), v\in{\cX_0}.
\end{equation}
Since $\dom(\cL^t_{{\cX_0}})\subset \cD^1_{\cL^t}(\Omega)$ and ${\cX_0}\subset H^1(\Omega_0)$, one has
\begin{equation}\lb{r15}
\langle {\gaN^{\cL^t}}u,\gaD v\rangle_{-1/2}=0, \text{\ for all\ } u\in \dom(\cL^t_{{\cX_0}}), v\in {\cX_0},
\end{equation}
thus $(\gaD u,\gaN^{\cL^t}u)$ is contained in the right-hand side of \eqref{r12} whenever $u\in \dom(\cL^t_{{\cX_0}})$.
\end{proof}
The form $\wti{\mathfrak{l}}^1$ and the subspace ${\cX_0}$ can be pulled back to $\Omega_1$ (via $\varphi: \Omega_0\rightarrow\Omega_1$), giving rise to a self-adjoint operator $\cL^1_{\cX_1}$ acting in $L^2(\Omega_1)$ and defined by 
\begin{equation}\lb{r16}
\mathfrak{l}^1[u,v]=\langle\cL^1_{\cX_1}u,v\rangle_{L^2(\Omega_1)},\ u\in\dom(\cL^1_{\cX_1}), v\in\cX_1,
\end{equation}
where $\cX_1:=\{u\circ \varphi^{-1}_1: u\in \cX_0\}$. Employing min-max type argument one can show that
\begin{equation}\lb{r17}
\mo(\cL^1_{\cX_1})=\mo(\cL^1_{\cX_0}).
\end{equation}
Finally, let us introduce the path of Lagrangian planes in $H^1/2(\partial \Omega_0)\times H^{-1/2}(\partial\Omega_0)$, corresponding to the weak solutions by setting
\begin{equation}\lb{r18}
\cK_{0,t}:=\tr_{\cL^t}\{u\in H^1(\Omega_0): \wti{\mathfrak{l}}^t[u,\psi]=0,\ \text{for all\ } \psi\in H^1_0(\Omega_0)\},\ t\in[0,1],
\end{equation}
and the {\it constant} (cf., Proposition \ref{r11}) path of Lagrangian planes corresponding to the boundary conditions
\begin{equation}\lb{r19}
\cG_{t}:=\overline{\tr_{\cL^t}(\dom(\cL^t_{{\cX_0}}))},\ t\in[0,1].
\end{equation}
Then employing Theorem \ref{w33}, we arrive at the formulas originally derived in \cite[Theorem 1]{CJM1},
\begin{equation}\lb{r20}
\mo(\cL^{0}_{\cX_0})-\mo(\cL^1_{\cX_0})=\mi\big((\cK_{0,t},\cG_t)|_{t\in[0,1]}\big),
\end{equation}
and, using \eqref{r17}, at the formula
\begin{equation}\lb{r21}
\mo(\cL^{0}_{\cX_0})-\mo(\cL^1_{\cX_1})=\mi\big((\cK_{0,t},\cG_t)|_{t\in[0,1]}\big).
\end{equation}

\subsection{Spectra of elliptic operators with Robin boundary conditions and the Maslov index}
We will now derive the Smale-type formula for second order differential operators subject to Robin boundary conditions, cf. \cite{S65, U73}, and also \cite{CJLS, CJM1, PW}. Assume that $\Omega\subset\bbR^n$, $n\geq 2$ is bounded open set with smooth boundary. Let us fix coefficients $\cA, B, q$ as in  \eqref{r22}, \eqref{3.71}, \eqref{r23},\ \eqref{r24}, and define the differential expression
\begin{equation}\lb{r31}
\cL:=- \text{div}\cA\nabla + B\nabla -\nabla\cdot\overline{B} +q.
\end{equation}
If $\theta\in \bbR$ then the linear operator $\cL_{\theta}$ acting in $L^2(\Omega)$ and defined by
\begin{align}
&\cL_{\theta}u:=- \text{div}(\cA\nabla u)+ B\nabla u-\nabla\cdot(\overline{B} u)+qu,\ u\in\dom(\cL_{\theta}),\lb{r30}\\
&\dom(\cL_{\theta}):=\{u\in H^1(\Omega): \cL u\in L^2(\Omega),\ \gaN^{\cL}u+\theta\gaD u=0\}\lb{r32},
\end{align}   
is self-adjoint, moreover, its essential spectrum is empty, cf. \cite[Proposition 2.3]{R14}.
\begin{proposition}\lb{r33} Assume that $\theta_1<\theta_2$, then
\begin{equation}\lb{r34}
\mo(\cL_{\theta_1})-\mo(\cL_{\theta_2})=\sum\limits_{\theta_1\leq \theta\leq \theta_2}\dim\ker(\cL_{\theta}).
\end{equation}
\begin{proof}
We will use  \eqref{w32} and show that all crossing corresponding to the variation of parameter $\theta\in[\theta_1,\theta_2]$ are sign-definite.  Theorem \ref{w33} yields
\begin{equation}\lb{r35}
\mo(\cL_{\theta_1})-\mo(\cL_{\theta_2})=\mi\big((\cK,\cG_{\theta})|_{\theta\in[\theta_1,\theta_2]}\big),
\end{equation}
where 
\begin{align}
\cK=\tr_{\cL}&\{u\in H^1(\Omega): \langle \cA\nabla u,\nabla \psi\rangle_{L^2(\Omega)}+\langle B\nabla u,\psi\rangle_{L^2(\Omega)}\ \\
&\hspace{1cm}+\langle u,B\nabla \psi\rangle_{L^2(\Omega)}+\langle qu,\psi\rangle_{L^2(\Omega)}=0,\ \text{for all}\ \psi\in H^1_0(\Omega)\},\\
&\hspace{-1.5cm}\cG_{\theta}:=\{(f,-\theta f): f\in H^{1/2}(\partial\Omega)\}\subset\bndro,\ \theta\in[\theta_1,\theta_2]
\end{align}
(we notice that $\cK$ does not depend on parameter $\theta$). Clearly $\cG_{\theta}$ is Lagrangian for each $\theta\in \bbR$, moreover,
the path 
\begin{equation}\lb{r36}
[\theta_1,\theta_2]\ni\theta\mapsto\cG_{\theta}\in\bndro,
\end{equation}
is continuously differentiable. Let $\theta_*\in[\theta_1,\theta_2]$ be a crossing point, then, by \cite[Lemma 3.8]{CJLS}, there exists a neighbourhood $\Sigma_*\subset[\theta_1,\theta_2]$ containing $\theta_*$ and a mapping
\begin{equation}
s\mapsto R_s\ \text{in}\ C^1\left(\Sigma_*, \cB(\cG_{\theta_*},\cG_{\theta_*}^{\perp})\right),
\end{equation}
such that
\begin{equation}
\cG_{s}=\left\{(f,g)+R_s(f,g): (f,g)\in\cG_{\theta_{*}}\right\}, s\in\Sigma_*.
\end{equation}
Next, pick any $(f_{\theta_*},g_{\theta_*})\in\cG_{\theta_*}\cap\cK$, then $g_{\theta_*}=-\theta_* f_{\theta_*},\ f_{\theta_*}\in H^{1/2}(\partial\Omega)$, and there exists $u_*\in\ker(\cL_{\theta_*})$ such that
\begin{equation}\lb{r37}
	\tr_{\cL}u_*=(f_{\theta_*},g_{\theta_*}),
\end{equation}
moreover,
\begin{equation}\lb{r38}
	(f_{\theta_*},g_{\theta_*})+R_s(f_{\theta_*},g_{\theta_*})=(f_s,-sf_s), s\in\Sigma_*,
\end{equation}
where the mapping $s\mapsto f_s$ is contained in $C^1(\Sigma_*,H^{1/2}(\partial\Omega))$. The derivative of $f_s$ with respect to $s$ evaluated at $s_*$ is denoted by $f'_{s_*}$ . We proceed by evaluating the Maslov crossing form at $\tr_{\cL}u_*=(f_{\theta_*},g_{\theta_*})$
\begin{align}
&\cQ_{\theta_*, \cK}\left(\tr_{\cL}u_*,\tr_{\cL}u_*\right)=\omega\left((f_{\theta_*},g_{\theta_*}),\frac{d}{ds} R_s(f_{\theta_*},g_{\theta_*})\right)|_{s=\theta_*}\no\\
&=\omega\big((f_{\theta_*},-\theta_*f_{\theta_*}), (f'_{\theta_*}, -(f_{\theta_*}+\theta_*f'_{\theta_*}))\big)\no\\
&=-\overline{\langle f_{\theta_*}+\theta_*f'_{\theta_*}, f_{\theta_*}\rangle}_{L^2(\partial\Omega)}-\langle -\theta_*f_{\theta_*}, f'_{\theta_*}\rangle_{L^2(\partial\Omega)}=-\|f_{\theta_*}\|_{L^2(\Omega)}^2\no,
\end{align}
finally we arrive at
\begin{equation}\lb{r39}
\cQ_{\theta_*, \cK}(\tr_{\cL}u_*,\tr_{\cL}u_*)=-\|\gaD u_*\|_{L^2(\Omega)}^2<0.	
\end{equation}
Therefore, a calculation similar to \eqref{w38} shows that
\begin{align}\lb{r40}
\mi\left((\cK,\cG_{\theta})|_{\theta\in[\theta_1,\theta_2]}\right)=-\mi\big((\cG_{\theta},\cK)|_{\theta\in[\theta_1,\theta_2]}\big)&\no\\
\quad =\sum\limits_{\theta_1\leq \theta\leq \theta_2}\dim\left(\cK\cap\cG_{\theta}\right)=\sum\limits_{\theta_1\leq \theta\leq \theta_2}\dim\ker\left(\cL_{\theta}\right),&\no
\end{align}
as asserted.
\end{proof}	
\end{proposition}

\section{The Maslov index for the Schr\"odinger operators on Lipschitz domains}\lb{section4}
In this section we establish relations between the Maslov and Morse indices, and, consequently, relations between the Maslov index and the spectral flow for Schr\"odinger operators with matrix valued potentials on Lipschitz domains.  The general result will be applied to two specific types of boundary conditions: First, $\vec{\theta}-$periodic on a cell $\Omega\subset\R^n$, and second to the Robin-type boundary conditions on star-shaped domains. Hypothesis \ref{d2} is imposed throughout this section. 
\subsection{A general result for the Schr\"odinger operators}
First, we verify Hypothesis \ref{bb1} in the present settings, that is, for the Schr\"odinger operator $\cL=-\Delta+V$ with bounded matrix valued potential. 
Assuming Hypothesis \ref{d2} and denoting the outward pointing normal unit vector to $\partial \Omega$ by $\vec{\nu}=(\nu_1,\, \cdots,\ \nu_n)$, we recall from \cite{GM10} two boundary spaces:
\begin{equation}\label{N12}
N^{1/2}(\dOm,\C^m):=\{g\in L^{2}(\dOm,\C^m)\,|\,g\nu_j\in H^{1/2}(\dOm,\C^m),\,1\leq j\leq n\},
\end{equation}
equipped with the natural norm
\begin{equation}\lb{N32}
\|g\|_{N^{1/2}(\dOm,\C^m)}:=\sum_{j=1}^n\|g\nu_j\|_{H^{1/2}(\dOm,\C^m)},
\end{equation}
and
\begin{equation}
N^{3/2}(\dOm):=\{g\in H^{1}(\dOm)\,|\,\nabla_{\tan}g\in (H^{1/2}(\dOm))^n\},
\end{equation}
equipped with the natural norm
\begin{equation}
\|g\|_{N^{3/2}(\dOm)}:=\|g\|_{L^{2}(\dOm)}+\|\nabla_{\tan}g\|_{H^{1/2}(\dOm)^n}.
\end{equation}
Here, the tangential gradient operator $\nabla_{\text{tan}}: H^1(\partial \Omega)\mapsto L^2(\partial \Omega)^n$ is defined as $$f\mapsto \left(\sum\limits_{k=1}^{n}\nu_k\frac{\partial f}{\partial \tau_{k,l}}\right)_{l=1}^n,$$
and $\frac{\partial}{\partial \tau_{k,l}}$  is the tangential derivative, which is a bounded operator between $H^s(\partial \Omega)$ and $H^{s-1}(\partial \Omega),\ 0\leq s\leq 1$, that extends the operator 
$$\frac{\partial}{\partial \tau_{k,l}}:\psi \mapsto \nu_k(\partial_{l}\psi)\big|_{\partial \Omega}-\nu_l(\partial_{k}\psi)\big|_{\partial \Omega},$$
originally defined for $C^1$ function $\psi$ in a neighbourhood of $\partial \Omega$.

\begin{lemma}[\cite{GM10}, Lemma 6.3]\lb{ff16}
Assume Hypothesis \ref{d2}. Then the Neumann trace operator $\gaN u=\nu\cdot\nabla u|_{\partial\Omega}$, $u\in H^2(\Omega)$ considered in the context
\begin{equation}\lb{ff12}
\gaN:H^2(\Om)\cap H^1_0(\Om)\to N^{1/2}(\dOm),
\end{equation}
is well-defined, bounded, onto, and with a bounded right-inverse. 
In addition, the null space of $\gaN$ in \eqref{ff12}
is precisely $H^2_0(\Om)$, the closure of $C^{\infty}_0(\Om)$ in $H^2(\Om)$.
\end{lemma}  
We will now show that both density assumptions in Hypothesis \ref{bb1} are satisfied for the Schr\"odinger operators on Lipschitz domains. Since $V$ is bounded it suffices to verify the assumptions for the Laplace operator. Let the function space \begin{equation}\lb{ff10}
\cD^s_{\Delta}(\Omega):=\{u\in H^s(\Omega,\mathbb C^m): \Delta u\in L^2(\Omega,\mathbb C^m)\},\ s\geq 0,
\end{equation} 
be equipped with the natural norm
\begin{equation}\lb{ff11}
\|u\|_{\Delta,s}:=\left(\|u\|_{H^s({\Omega},\C^m)}^2+\|\Delta u\|_{L^2({\Omega,\C^m})}^2\right)^{1/2},\ s\geq 0.
\end{equation}
Let us denote 
\begin{equation}
\tr_{\Delta}:=\left(\gaD u, \gaN u\right), \tr_{\Delta}\in \cB\left(\cD_{\Delta}^1(\Omega),\bndr\right).
\end{equation}
\begin{proposition}\lb{ff8}
Assume Hypothesis \ref{d2}. Then
\begin{align}
&i)\ \ran(\tr_{\Delta,1})\text{\ is dense in\ }\bndr, \lb{ff9}\\
&ii)\ \cD^1_{\Delta}(\Omega)\text{\ is dense in\ }\cD^0_{\Delta}(\Omega).\lb{ff9f}
\end{align}
\end{proposition}
\begin{proof}
First we prove part {\it i}). It suffices to show that 
\begin{equation}\lb{gg5}
\left(\left\{\left(\gaD u, \gaN u\right): u\in \cD^{1}_{\Delta}(\Omega)\right\}\right)^{\circ}=\{(0,0)\},
\end{equation} 
where the left-hand side denotes the annihilator with respect to the sympletic form \eqref{dd9}.
Pick an arbitrary
\begin{equation}\lb{gg7}
	(\varphi,\psi)\in\left(\left\{\left(\gaD u, \gaN u\right): u\in \cD^{1}_{\Delta}(\Omega)\right\}\right)^{\circ},
\end{equation}
then
\begin{align}\lb{gg8}
\overline{\langle \psi  ,\gaD f \rangle}_{{-1/2}}-\langle \gaN f ,\varphi \rangle_{{-1/2}}=0,\ \text{for all}\ f\in\cD_{\Delta}^1(\Omega).
\end{align}
By Lemma \ref{ff16}, for arbitrary $g\in N^{1/2}(\partial\Omega,\C^m)$ there exists $F_g\in H^2(\Omega,\C^m)$ such that
\begin{align}
&\gaD F_g= 0,\ \gaN F_g=g.
\end{align} 
Using equation \eqref{gg8} with $f=F_g$, one obtains
\begin{equation}\lb{gg9}
\langle g,\varphi \rangle_{{-1/2}}=0,\ \text{for all}\ g\in N^{1/2}(\partial\Omega,\C^m).
\end{equation}  
 In addition, by \cite[Corollary 6.12]{GM10}, we have
\begin{equation}\lb{gg3}
 N^{1/2}(\partial\Omega,\C^m)\hookrightarrow L^{2}(\partial\Omega,\C^m)\hookrightarrow H^{-1/2}(\partial\Omega,\C^m),
\end{equation}
where both inclusions are dense and continuous. Therefore, \eqref{gg9} can be extended by continuity to $H^{-1/2}(\partial\Omega,\C^m)$, and one has 
\begin{equation}\lb{gg10}
\langle g,\varphi \rangle_{{-1/2}}=0,\ \text{for all}\ g\in H^{-1/2}(\partial\Omega,\C^m),
\end{equation}
hence, $\varphi=0$. Combining \eqref{gg8} and \eqref{gg10}, one obtains
\begin{equation}\lb{gg11}
\langle \psi  ,\gaD f \rangle_{{-1/2}}=0,\ \text{for all}\ f\in\cD_{\Delta}^1(\Omega).
\end{equation}
Recall from \cite[Lemma 2.3]{GM10} that $\gaD$ considered in the context 
\begin{equation}\lb{gg12}
\gaD:\cD_{\Delta}^{3/2}(\Omega)\to H^1(\partial\Omega,\C^m),
\end{equation}
is compatible with \eqref{e9}, bounded, has bounded right-right inverse (hence, onto). Then, for arbitrary $h\in H^1(\partial\Omega,\C^m)$ there exits $G_h\in\cD_{\Delta}^{3/2}(\Omega)\subset\cD_{\Delta}^{1}(\Omega)$, such that $\gaD G_h=h$. Let us set $f=G_h$ in \eqref{gg11} and obtain
\begin{equation}\lb{gg13}
\langle \psi  ,h \rangle_{{-1/2}}=0,\ \text{for all}\ h\in H^1(\partial\Omega,\C^m).
\end{equation}
Since the inclusion 
\begin{equation}\lb{gg14}
H^{1}(\partial\Omega,\C^m)\hookrightarrow H^{1/2}(\partial\Omega,\C^m)
\end{equation}
is dense, \eqref{gg13} yields $\psi=0$. Thus, $(\varphi,\psi)=(0,0)$ and consequently part {\it i}) holds.

The second assertion follows from the fact that
\begin{equation}\lb{gg15}
C^{\infty}(\overline{\Omega})\hookrightarrow\cD^0_{\Delta}(\Omega),
\end{equation}
densely, cf. \cite{BM}. 
\end{proof}

Next, we turn to a Lagrangian formulation of eigenvalue problems for self-adjoint extensions of $-\Delta_{min}$, 
\begin{equation}\lb{ab20}
	-\Delta_{min} u:=-\Delta u, u\in\dom(-\Delta_{min}):= H^2_0(\Omega).
\end{equation}
Recall, that $(-\Delta_{min})^{*}=-\Delta_{max}$, where 
\begin{equation}\lb{ac10}
	-\Delta_{max}u:=-\Delta u,\ u\in \dom(-\Delta_{max}):=\cD_{\Delta}^1(\Omega).
\end{equation}
The self-adjoint extension of $-\Delta_{min}$ with domain $\mathscr D\subset \cD_{\Delta}^1(\Omega)$ is denoted by $-\Delta_{\mathscr D}$.
\begin{hypothesis}\lb{r51} Let $\Omega\subset\bbR^n, n\geq 2$ be open, bounded, Lipschitz domain and assume that the mapping
\begin{equation}\no
\cI\ni t\mapsto V_t\in L^{\infty}(\Omega,\C^{m\times m}),\ 	V_t=\overline{V_t}^{\top}, t\in\cI,
\end{equation}
is contained in $ C^1(\cI, L^{\infty}(\Omega, \C^{m\times m})),\ \cI:=[\alpha,\beta]$. 

Moreover, let us assume that $f$ is a given function such that
\begin{equation}\lb{y24y}
f:\cI\rightarrow\bbR,\ f\in C^1(\cI),\ f(t)>0,\ \partial_tf(t)\not=0,\ t\in \cI.
\end{equation}
\end{hypothesis}
Let us denote by $K_{\lambda, t,f}$ the trace of the set of weak solutions to the eigenvalue problem $-\Delta u+Vu=\lambda u$, that is,
\begin{align}
\begin{split}
&{K}_{\lambda,t,f}:= \tr_{\Delta}\big\{u\in \cD_{\Delta}^1(\Omega): f(t)\langle \nabla u,\nabla \varphi\rangle_{L^2(\Omega,\C^m)}+\langle V_tu,\varphi\rangle_{L^2(\Omega,\C^m)}\lb{r53}\\
&=\lambda\langle u,\varphi\rangle_{L^2(\Omega,\C^m)},\text{\ for all}\ \varphi\in H^1_0(\Omega,\C^m)\big\}, \lambda\in \bbR,t\in\cI,
\end{split}
\end{align} 
where $\nabla u:=[\nabla u_1,\cdots,\nabla u_{m}]^{\top}\in \C^{m\times n},$
\begin{align}
&\langle \nabla u,\nabla v\rangle_{L^2(\Omega,\C^{m})}:=\sum\limits_{i=1}^{m}\langle\nabla u_i,\nabla v_i\rangle_{{[L^2(\Omega,\C)]^n}},\no
\end{align}
for given $u=(u_i)_{i=1}^{m},v=(v_i)_{i=1}^{m}\in H^1(\Omega, \C^m)$.
\begin{theorem}\lb{r57}
Assume Hypotheses \ref{d2} and \ref{r51}.  Let $\mathscr D_t\subset \cD_{\Delta}^1(\Omega)$, $t\in \cI$, and assume that the linear operator $\cL^t_{\mathscr D_t}=-f(t)\Delta_{\mathscr D_t}+V_t$ acting in $L^2(\Omega,\C^m)$ and given by
\begin{equation}\lb{r54}
 \cL^t_{\mathscr D_t} u:=-f(t)\Delta u +V_tu,\ u\in \dom(\cL^t_{\mathscr D_t}):={\mathscr D_t},
\end{equation}
is self-adjoint with $\spec_{ess}\left(\cL^t_{\mathscr D_t}\right)=\emptyset,\ t\in\cI$.  Assume that there exists $\lambda_{\infty}<0$, such that 
\begin{equation*}
\ker(\cL^t_{\mathscr D_t}-\lambda)=\{0\}\text{\ for all\ }\lambda\leq \lambda_{\infty}, t\in\cI.
\end{equation*}
Suppose that the path
\begin{equation}\lb{r55}
t\mapsto\cG_t:=\overline{\tr_{\Delta,1}(\mathscr D_t)}\in\bndr,
\end{equation}
is contained in $ C^1\left(\cI,\Lambda\left(\bndr\right)\right)$.	

Then, one has
\begin{equation}\lb{r56}
\mo\left(\cL^{\alpha}_{\mathscr D_{\alpha}}\right)-\mo\big(\cL^{\beta}_{\mathscr D_{\alpha}}\big)=\mi\left((K_{0,t, f},\cG_t)|_{t\in\cI}\right).
\end{equation}
\end{theorem}
The proof of Theorem \ref{r57} is similar to that of Theorem \ref{w33}, and is omitted.
We complete this section by illustrating applications of \eqref{r56}.
We note that the Maslov index of the path $\big((K_{0,t,f},\cG_t)|_{t\in\cI}\big)$ is equal to the spectral flow of $\{\cL^t_{\mathscr D_t}\}_{t=\alpha}^{\beta}$, that is,
\begin{equation}\lb{y1}
\spflow\left(\{\cL^t_{\mathscr D_t}\}_{t=\alpha}^{\beta}\right)=\mi\big((K_{0,t, f},\cG_t)|_{t\in\cI}\big).
\end{equation}

\subsection{Spectra of $\vec{\theta}-$periodic Schr\"odinger operators and the Maslov index}
In this section we derive a relation between the Maslov and Morse indices for multidimensional $\vec{\theta}-$periodic Schr\"odinger operators  as an application of \eqref{r56}. 

Firstly, we define the self-adjoint extension of $-\Delta_{min}$ corresponding to $\vec{\theta}-$periodic boundary conditions 
\begin{equation}\no
u(x+a_j)=\e^{2\pi\bfi\theta_j}u(x),\  \frac{\partial u}{\partial \vec{\nu}}(x+a_j)=\e^{2\pi\bfi\theta_j}\frac{\partial u}{\partial \vec{\nu}}(x),
\end{equation} 
where $\{a_1, \dots a_n\}\subset\R^n$ are linearly independent vectors, $\vec{\theta}:=(\theta_1,\dots, \theta_n)\in[0,1)^n$.  Let 
$Q$ denote the unit cell
\begin{align*}
Q&:=\{t_1{a}_1+\cdots+t_n{a}_n|\ 0\leq t_j\leq 1 , j\in\{1,\dots, n\}\}.
\end{align*}
The faces $\partial Q^s_j$ of the unit cell $Q$
(so that $\partial Q=\cup_{s=0}^1\cup_{j=1}^n\partial Q^s_j$) are denoted by \[{\partial Q}^{s}_{j}: =\{t_1{a}_1+\cdots+t_n{a}_n\in Q\big|\, t_{j}=s\}, \  j\in\{1,\dots, n\},\ s\in\{0,1\}.\] The $n$-tuple $\{{a}_1, \dots {a}_n\}\subset\R^n$ is uniquely associated with an $n\times n$ matrix $A$ by the condition $A{a}_j=2\pi {e}_j$, where $\{{e}_j\}_{1\leq j\leq n}$ is the standard basis in $\C^n$. For the matrix $A$ just defined, and $k\in\bbZ^n$ we denote 
\begin{align}
\zeta_k(x)&:=|Q|^{-1}\e^{{\bfi A^{\top}(\vec{\theta}-k)\cdot x}},\,   x\in Q.\lb{y17}
\end{align} 
Recalling that $\partial Q=\cup_{s=0}^{1}\cup_{j=1}^{n}{\partial Q}_j^s$, we define the Dirichlet trace operators corresponding to each face of $Q$ as follows,
\begin{align}
&\gamma_{D,{\partial Q}_{j}^{s}}:H^2(Q, \C^m)\rightarrow L^{2}({\partial Q}_{j}^{s},\C^{m}),\no\\
&\gamma_{D,{\partial Q}_{j}^{s}}(u):=(\gamma_{D}u)|_{{\partial Q}_{j}^{s}},\ \ 1\leq j\leq n,\ s\in\{0,1\}.\no
\end{align}
It follows that $\gamma_{D,{\partial Q}_j^{s}}\in\cB\big(H^2(Q, \C^m),L^2({\partial Q}_j^{s};{\C}^{m})\big)$ for $1\leq j\leq n$
and $s\in\{0,1\}$. The Neumann trace is given by
\begin{align}
&\gamma_{N,{\partial Q}_{j}^{s}}:H^2(Q, \C^m)\rightarrow L^2({\partial Q}_{j}^{s};{\C}^{m}),\no\\
&\gamma_{N,{\partial Q}_{j}^{s}}(u):=\big({{\gamma}}_{D}(\nabla u) \overrightarrow{\nu}\big)\big|_{{\partial Q}_{j}^{s}},\ \ 1\leq j\leq n,\ s\in\{0,1\},\no
\end{align}
where $\vec{\nu} \text{\ is the outward pointing normal unit vector to\ }\partial Q$. 
The inclusion   
\begin{equation*}
\gamma_{N,{\partial Q}_j^{s}}\in \cB\left(H^2(Q, \C^m),L^2\left({\partial Q}_j^{s};\C{^{m\times n}}\right)\right),
\end{equation*}
holds for all $1\leq j\leq n,\ s\in\{0,1\}$. For each $u\in \Htworm$ we denote
\begin{equation}\label{y4}
u_{j}^{s}:=\gamma_{D,{\partial Q}_{j}^{s}}(u),\quad \partial_{\nu} u^{s}_{j}:=\gamma_{N,{\partial Q}^{s}_{j}}(u),  \ 1\leq j\leq n,\ s\in\{0,1\}.
\end{equation} 
Let us also introduce the weighted translation operators
\begin{align}
\begin{split}
&\bbM_j\in\cB \left(L^2({\partial Q}_j^0;\C^m), L^2({\partial Q}_j^1;\C^m)\right),\label{y6}\\
&(\bbM_j u)(x)=\e^{2\pi\bfi\theta_j}u(x-a_j) \text{\ for a.a.\ }x\in{\partial Q}_j^1,\ 1\leq j\leq n.
\end{split}
\end{align}

\begin{proposition}\lb{y5} Recall notation \eqref{y4}, \eqref{y6}. Then the linear operator
\begin{align}
&\hspace{-0.25cm}-\Delta_{\vec{\theta}}:\dom (-\Delta_{\vec{\theta}})\subset L^2(Q,\C^m)\rightarrow L^2(Q,\C^m),\lb{y15}\\
&\hspace{-0.25cm}\dom (-\Delta_{\vec{\theta}}):=\big\{u \in H^2(Q,\C^m):\,u^1_j=\bbM_ju^0_{j},\partial_{\nu} u^1_j=-\bbM_j\partial_{\nu} u^0_j, 1\leq j\leq n\big\},\lb{y14}\\
&\hspace{-0.25cm}-\Delta_{\vec{\theta}} u:=-\Delta u, \ u\in  \dom (-\Delta_{\vec{\theta}})\lb{y16}
\end{align}
is self-adjoint, moreover
\begin{equation}
-\Delta_{\min}\subset-\Delta_{\vec{\theta}}\subset-\Delta_{\max}.\no
\end{equation}
In addition, $-\Delta_{\vec{\theta}}$ has compact resolvent, in particular, it has purely discrete spectrum. Finally, $\spec(-\Delta_{\vec{\theta}})=\big\{\|A^{\top}(\vec{\theta}-k)\|^2_{\R^n}\big\}_{k\in\bbZ^n}$.
\end{proposition}
\begin{proof}
Recall \eqref{y17}. Then the sequence of functions
\begin{align}
&\phi_{k,l}(x):=(0,\cdots, \underbrace{\zeta_k(x)}_{l-th \text{\ position}},\cdots, 0)^{\top}, k\in\bbZ^n, 1\leq l\leq m,\lb{y7}
\end{align}
form an orthonormal basis in $L^2(Q,\C^m)$. In addition, $\phi_{k,l}\in\dom(-\Delta_{\vec{\theta}})$, since by
\begin{equation}\lb{y10}
A^{\top}(\vec{\theta}-k)\cdot {a}_j=(\vec{\theta}-k)\cdot A {a}_j=2\pi(\vec{\theta}-k)\cdot {e}_j=2\pi(\theta_j-k_j),\ k\in \bbZ^n,\ 1\leq j\leq n,
\end{equation}
one has
\begin{align}
\begin{split}
{|Q|}^{-1}{\e^{\bfi  A^{\top}(\vec{\theta}-k)\cdot (x+{a}_j)}}&=\e^{2\pi \bfi \theta_j}{|Q|}^{-1}{\e^{\bfi  A^{\top}(\vec{\theta}-k)\cdot x}},\lb{y11}\\
\nu\cdot \nabla\Big({|Q|}^{-1}{\e^{\bfi  A^{\top}(\vec{\theta}-k)\cdot (x+{a}_j)}}\Big)&=\e^{2\pi \bfi \theta_j}\nu\cdot\nabla\Big({|Q|}^{-1}{\e^{\bfi  A^{\top}(\vec{\theta}-k)\cdot x}}\Big),
\end{split}
\end{align}
that is 
\begin{equation*}
(\phi_{k,l})_j^1=\bbM_j(\phi_{k,l})_j^0\text{\ and\ }\partial_{\nu}(\phi_{k,l})_j^1=\bbM_j\partial_{\nu}(\phi_{k,l})_j^0, 1\leq j\leq n.
\end{equation*}
Furthermore,
\begin{equation}\lb{y12}
-\Delta\phi_{k,l}=\|A^{\top}(\vec{\theta}-k)\|^2_{\R^n}\phi_{k,l}, k\in\bbZ^n, 1\leq l\leq m.
\end{equation}
From these facts we infer (cf., \cite{LSS} for details) that
\begin{equation}\lb{y13}
\text{span}\{\phi_{k,l}: k\in\bbZ^n,\ 1\leq l\leq m\},
\end{equation}
is a core of operator $-\Delta_{\vec{\theta}}$. Hence, $-\Delta_{\vec{\theta}}$ is self-adjoint with domain \eqref{y14}, it has compact resolvent due to the fact that 
\begin{equation}
	\|A^{\top}(\vec{\theta}-k)\|^2_{\R^n}\rightarrow \infty,\text{\ as\ } \|k\|_{\C^n}\rightarrow\infty,
\end{equation}
cf. \cite[Lemma 3.2]{LSS}.
\end{proof}
Let $tQ:=\{tx, x\in Q\}, t\in (0,1],$ and define
\begin{align}
-\Delta^t_{\vec{\theta}}:&\dom (-\Delta_{\vec{\theta}})\subset L^2(tQ,\C^m)\rightarrow L^2(tQ,\C^m),\no\\
\dom (-\Delta_{\vec{\theta}}):&=\big\{u \in H^2(tQ,\C^m):\,u^1_j=\bbM^t_ju^0_{j},\partial_{\nu} u^1_j=-\bbM^t_j\partial_{\nu} u^0_j, 1\leq j\leq n\big\},\no\\
-\Delta^t_{\vec{\theta}} u:&=-\Delta u, \ u\in  \dom (-\Delta^t_{\vec{\theta}})\no,
\end{align}
where $\bbM^t_j$ is the weighted translation operator acting from $L^2({\partial (tQ)}_j^0;\C^m)$ to $ L^2({\partial (tQ)}_j^1;\C^m)$, cf. \eqref{y6}.
Assume that $V\in L^{\infty}(Q,\C^{m\times m})$, and denote 
\begin{align}
\begin{split}
&{K}_{\lambda,t}:= \tr_{\Delta}\Big\{u\in \cD_{\Delta}^1(Q): \int_{Q} t^{-2} \langle \nabla u(x),\nabla \varphi(x)\rangle_{\C^{m\times n}}\\
&\hspace{2cm}+\langle V(tx)u(x),\varphi(x)\rangle_{\C^{m}} -\lambda\langle u,\varphi\rangle_{\C^m}d^nx=0, \lb{y20}\\
&\hspace{5cm}\text{\ for all}\ \varphi\in H^1_0(Q,\C^m)\Big\},\ \lambda\in \bbR,\ t\in\bbR,\\
&\cG_{\vec{\theta}}:=\overline{\tr_{\Delta}\big\{\dom(-\Delta_{\vec{\theta}})\big\}}^{\bndr}.
\end{split}
\end{align} 
\begin{theorem}\lb{y19} 
If $V\in L^{\infty}(Q,\C^{m\times m})$ then for any $\tau\in(0,1],$ $\vec{\theta}\in[0,1)^n,$ and $K_{\lambda,t}$ from \eqref{y20}, one has
\begin{equation}\lb{y21}
\mo\left(-\Delta^{\tau}_{\vec{\theta}}+V|_{\tau Q}\right)-\mo\left(-\Delta_{\vec{\theta}}+V\right)=\mi\left((K_{0,t},\cG_{\vec{\theta}})|_{t\in[\tau,1]}\right).
\end{equation}

If $\vec{\theta}\not=0$, then
\begin{equation}\lb{y22}
\mo\left(-\Delta_{\vec{\theta}}+V\right)=-\mi\big((K_{0,t},\cG_{\vec{\theta}})|_{t\in[\tau_0,1]}\big),
\end{equation}
for small enough $\tau_0>0$. 

If $V$ is continuous at $0$ and $V(0)$ is invertible, then
\begin{equation}\lb{y23}
\mo(V(0))-\mo\left(-\Delta_{\vec{0}}+V\right)=\mi\big((K_{0,t},\cG_{\vec{0}})|_{t\in[\tau_0,1]}\big).
\end{equation}
for small enough $\tau_0>0$. 
\end{theorem}
\begin{proof}
Introducing one-parameter family of self-adjoint operators acting in $L^2(Q,\C^m)$ by the formula
\begin{equation}\lb{y24}
	\cL^t:=-t^{-2}\Delta_{\vec{\theta}}+V(t\cdot),\ \dom(\cL^t):=\dom(-\Delta_{\vec{\theta}}),\ t\in(0,1],
\end{equation}
and using Theorem \ref{r57}, we arrive at the relation
\begin{equation}\lb{y25}
\mo(\cL^\tau)-\mo(\cL^1)=\mi\big((K_{0,t},\cG_{\vec{\theta}})|_{t\in[\tau,1]}\big).
\end{equation}
Notice that $\cL^1=-\Delta_{\vec{\theta}}+V$, and that 
\begin{equation}
u\in\ker(\cL^{\tau})\text{\ if and only if\ }u({\cdot}/{\tau})\in\ker(-\Delta^{\tau}_{\vec{\theta}}+V|_{\tau Q}), 
\end{equation}
then 
\begin{equation}\lb{y26}
\mo(\cL^\tau)=\mo(-\Delta^{\tau}_{\vec{\theta}}+V|_{\tau Q}).
\end{equation}
Combining \eqref{y25} and \eqref{y26}, we infer \eqref{y21}. By \cite[Lemma 3.10, Proposition 3.13]{LSS}, we infer
\begin{align}
&\mo(-\Delta^{\tau}_{\vec{\theta}}+V|_{\tau Q})=0, \text{whenever $\tau$ is small enough},\lb{y30}\\
&\mo(-\Delta^{\tau}_{\vec{0}}+V|_{\tau Q})=\mo(V(0)),\ \text{whenever $\tau$ is small enough}.\lb{y31}
\end{align}
Equations \eqref{y21}, \eqref{y30}, \eqref{y31} imply \eqref{y22}, \eqref{y23}.
\end{proof}

\subsection{Spectra of Schr\"odinger operators on star-shaped domains} 
To set the stage we impose the following hypothesis.
\begin{hypothesis}\lb{rr1}
Let $\Omega\subset\bbR^n, n\geq 2,$ be non-empty, open, bounded, star-shaped, Lipschitz domain. Let $\cG\subset \bndr$ be a Lagrangian plane with respect to symplectic form \eqref{dd9}. Assume that $V\in L^{\infty}(\Omega,\C^m),\ m\in\bbN$.
\end{hypothesis}
Without loss of generality we assume that $\Omega$ is centered at the origin. Let $\tau>0$, $t\in[\tau,1)$ and denote
\begin{equation}\lb{rr2}
\Omega_t:=\{x\in\Omega: x=t'y, \text{\ for\ } t'\in[0,t),\ y\in\partial\Omega  \}.
\end{equation}
The Dirichlet and Neumann trace operators considered in $\Omega_t$ are denoted by
\begin{align}
&\gamma_{D,t}\in\cB(H^1(\Omega_t), H^{1/2}(\partial\Omega_t)),\ \gamma_{N,t}\in\cB(\cD^1_{\Delta}(\Omega_t),  H^{-1/2}(\partial\Omega_t)),\no\\
& \tr_{\Delta,t}:=(\gamma_{D,t}, \gamma_{N,t}):\cD^1_{\Delta}(\Omega_t)\rightarrow H^{1/2}(\partial\Omega_t)\times H^{-1/2}(\partial\Omega_t),\ t\in[\tau,1).\no
\end{align}
The minimal and maximal Laplacians on $\Omega_t$ are denoted by $\Delta_{min,t}$ and $\Delta_{max,t}$.
Following \cite[Section 4.1]{CJLS} we introduce the scaling operators,
\begin{align}
&U_t: L^2(\Omega_t)\rightarrow L^2(\Omega),\ \ (U_tw)(x):=t^{n/2}w(tx), x\in\Omega, \no\\
&U^{\partial}_t: L^2(\partial\Omega_t)\rightarrow L^2(\partial\Omega),\ \ (U^{\partial}_th)(y):=t^{(n-1)/2}h(ty), y\in\partial\Omega, \no\\
&U^{\partial}_{1/t}: L^2(\partial\Omega)\rightarrow L^2(\partial\Omega_t),\ \ (U^{\partial}_{1/t}f)(z):=t^{-(n-1)/2}h(t^{-1}z), z\in\partial\Omega_t.
\end{align}
Finally, we notice that $U_t\in\cB(H^1(\Omega_t),H^1(\Omega)),$ $U_t^{\partial}\in\cB(H^{1/2}(\partial\Omega_t),H^{1/2}(\partial\Omega))$, and define
$U_t^{\partial}: H^{-1/2}(\partial\Omega_t)\rightarrow H^{-1/2}(\partial\Omega)$ by
\begin{equation}\lb{rr4}
\langle U^{\partial}_{t} g, \phi\rangle_{-1/2}:=_{H^{-1/2}(\partial\Omega_t)}\langle g, U^{\partial}_{1/t}\phi\rangle_{H^{1/2}(\partial\Omega_t)},\  \phi\in H^{1/2}(\partial\Omega).
\end{equation}
It follows that the subset 
\begin{equation}\lb{rr5}
	\cG_{\partial\Omega_t}:=\left\{\left(U_{1/t}^{\partial}f,U_{1/t}^{\partial}g\right): (f,g)\in\cG\right\}\subset H^{1/2}(\partial \Omega_t)\times H^{-1/2}(\partial \Omega_t),
\end{equation}
is Lagrangian with respect to the natural symplectic form $\omega_t$ defined on $H^{1/2}(\partial\Omega_t)\times H^{-1/2}(\partial\Omega_t)$. Let $\cS_{\Omega_t}$ denote the self-adjoint extension of $-\Delta_{min,t}+V|_{\Omega_t}$ associated with $\cG_{\partial\Omega_t}$ via Theorem \ref{l4}.
\begin{hypothesis}\lb{rr6}
Assume that $\spec_{ess}\left(\cS_{\Omega_t}\right)\cap(-\infty, 0]=\emptyset$, $t\in[\tau,1)$, and that there exists $\lambda_{\infty}<0$ such that 
\begin{equation}\lb{rr7}
\spec \left(\cS_{\Omega_t}\right)\subset[\lambda_{\infty},+\infty)\ \text{for all\ } t\in [\tau,1).
\end{equation}
\end{hypothesis}
\begin{proposition}\lb{rr8}
Assume Hypotheses \ref{rr1} and \ref{rr6}. Then, for arbitrary $\tau>0$, one has
\begin{equation}\lb{rr9}
\mo(\cS_{\Omega_{\tau}})-\mo(\cS_{\Omega_{1}})=\mi\left( (\cK_{0,t}, \cG_t)|_{t\in[\tau,1]} \right), 
\end{equation}
where $\cK_{0,t}$ is defined by  \eqref{y20} with $\lambda=0$ and $Q$ replaced by $\Omega$, and
\begin{equation*}
\cG_t:=\left\{ (f,g)\in\bndr: (f,t^{-1}g)\in\cG\right\},\ t\in[\tau,1].
\end{equation*}
\end{proposition}
\begin{proof}
Clearly $\cG_t$, $t\in[\tau,1]$ is contained in 
\begin{equation*}
C^1\left([\tau,1],\Lambda\left(\bndr\right)\right).
\end{equation*}
Let $\cL_t$ be the self-adjoint operator associated (via Theorem \ref{l4}) with the differential expression
\begin{equation}
L_t=-t^{-2}\Delta +V(tx), x\in\Omega,
\end{equation}
and  the Lagrangian plane $\cG_t,\ t\in[\tau,1].$ By \cite[Lemma 4.1]{CJLS}, 
\begin{equation}
	 w\in\ker\big(\cS_{\Omega_t}-\lambda\big) \text{\ if and only if\ }(U_tw)\in\ker\big(\cL_t-\lambda \big), t\in[\tau,1], \lambda\in\bbR.
\end{equation}
Hence, $\mo(\cS_{\Omega_t})=\mo(\cL_t), t\in[\tau,1].$ The one-parameter family of self-adjoint operators $\cL_t$ acting in $L^2(\Omega)$ together with one-parameter family of Lagrangian planes $\cG_t, t\in[\tau,1]$ satisfy hypotheses of Theorem \ref{r57}, therefore
\begin{equation}\lb{rr10}
\mo(\cL_{\tau})-\mo(\cL_1)=\mi\left( (\cK_{0,t}, \cG_t)|_{t\in[\tau,1]} \right).
\end{equation} 
Combining \eqref{rr10}, $\cL_1=\cS_{\Omega_1}$ and $\mo(\cS_{\Omega_{\tau}})=\mo(\cL_{\tau})$, we arrive at \eqref{rr9}.
\end{proof}
\begin{example}
Assume Hypothesis \ref{rr1}. Let
\begin{equation*}
0\leq \theta\in L^{\infty}(\partial\Omega, \C^{m\times m}),\ \theta(x)=\overline{\theta(x)}^{\top}, x\in\Omega.
\end{equation*}
The Lagrangian plane
\begin{equation}\lb{rr11}
\cG:=\left\{(f,g)\in \bndr: \theta f+g=0\right\},
\end{equation}
gives rise to a one-parameter family of self-adjoint Schr\"odinger operators $\cS_{\Omega_t}, t\in[\tau,1]$ acting in $L^2(\Omega_t),t\in[\tau,1], 0<\tau<1$ and given by
\begin{align}
&\cS_{\Omega_t}u=-\Delta u+V|_{\Omega_t}u, u\in \dom(\cS_{\Omega_t}),\no\\
\dom(\cS_{\Omega_t})&=\{u\in \cD^1_{\Delta}(\Omega_t): \theta\left(x/t\right)\gamma_{D,t}u(x)+\gamma_{N,t}u(x)=0, x\in\partial\Omega_t\}.\no
\end{align} 
By \cite[Theorem 2.6]{GM08}, the operator $\cS_{\Omega_t}$ is bounded from below and has compact resolvent. Hypothesis \ref{rr6} is satisfied since $\theta$ is bounded and nonnegative. Therefore, \eqref{rr9} holds in case of Schr\"odinger operators with Robin boundary conditions on star-shaped domains.
\end{example}
\section{The abstract boundary value problems}\lb{section5}
In this section we elaborate on a natural relation between theory of ordinary boundary triples originated in \cite{Br}, \cite{GG}, \cite{Ko} and the theory of abstract boundary value spaces exploited in \cite{BbF95}.
\subsection{Lagrangian planes and self-adjoint extensions via the abstract boundary triples} We begin with several abstract results concerning the relations between the Morse and Maslov indices in the context of boundary triples. The following hypothesis is imposed throughout this section.
\begin{hypothesis}\lb{y35}
Let $\cH, \mathfrak{H}$ be complex, separable Hilbert spaces. Assume that $A$ is a densely defined, symmetric operator acting in $\cH$. Assume that $A$ has equal deficiency indices, that is,
\begin{equation}
\dim\ker(A^*-\bfi)=\dim\ker(A^*+\bfi).
\end{equation}
\end{hypothesis}

\begin{definition}[\cite{GG}]\lb{y33} Assume Hypothesis \ref{y35}. Let $\Gamma_1, \Gamma_2: \dom(A^*)\rightarrow \mathfrak{H}$ be linear maps. Then $(\mathfrak{H},\Gamma_1, \Gamma_2)$ is said to be a boundary triple if the following assumptions are satisfied

1)  the abstract second Green identity holds, that is, for all $f,g\in\dom(A^*)$
 \begin{equation}\lb{i7}
\langle A^*f,g\rangle_{\cH}-\langle f,A^*g\rangle_{\cH}=\langle\Gamma_1f, \Gamma_2g\rangle_{\mathfrak{H}}-\langle\Gamma_2f,\Gamma_1g\rangle_{\mathfrak{H}},
\end{equation}


2) the map $\tr_{\mathfrak H}:=(\Gamma_1,\Gamma_2):\dom(A^*)\rightarrow \mathfrak{H}\times\mathfrak{H}$ is onto, i.e., for arbitrary $(\varphi,\psi)\in\mathfrak{H}\times\mathfrak{H}$ there exists $u\in\dom(A^*)$, such that $\Gamma_1 u=\varphi,\ \Gamma_2 u=\psi$.
\end{definition} 
If Hypothesis \ref{y35} holds then there exists a boundary triple associated to $A$, cf., \cite{GG}. Moreover, 
\begin{align}
&\tr_{\mathfrak H}\in\cB(\dom({A^*}), \mathfrak H\times \mathfrak H)\text{\ and\ }\ker(\tr_{\mathfrak H})=\dom(A)\lb{i8},
\end{align}
where $\dom(A^*)$ is viewed as a Hilbert space equipped with the graph norm of $A^*$
\begin{equation}\lb{y36}
\|x\|_{A^*}^2:=\|x\|_{\cH}^2+\|A^*x\|_{\cH}^2, \ x\in\dom(A^*).
\end{equation}
The quotient space $\dom(A^*)/\dom(A)$ equipped with the bounded, non-degenerate, skew-symmetric form
\begin{align}
&\omega_{\cH}([x],[y]):=\langle A^*x,y\rangle_{\cH}-\langle x,A^*y\rangle_{\cH}, [x],[y]\in\dom(A^*)/\dom(A),\lb{i4}\\
&\text{($[x]$ denotes the equivalence class of vector $x\in\dom(A^*)$)},\no
\end{align}
is a symplectic Hilbert space with respect to the standard quotient norm induced by $\|\cdot\|_{A^*}$. It was originally used in \cite{BbF95}. 
\begin{proposition}\lb{i15}
Let $(\mathfrak H, \Gamma_1, \Gamma_2)$ be a boundary triple. The map 
\begin{align}
&\wti \tr_{\mathfrak H}: \dom(A^*)/\dom(A)\rightarrow \mathfrak H\times \mathfrak H,\lb{i6},\\
&\dom(A^*)/\dom(A)\ni[x]\mapsto(\Gamma_1x,\Gamma_2x)\in\mathfrak H\times \mathfrak H,\lb{i5}	
\end{align}
is well defined, bounded, has bounded inverse, and 
\begin{equation}\lb{i10}
\omega_{\cH}\left([x],[y]\right)=\omega_{\mathfrak H}\left(\wti \tr_{\mathfrak H}[x], \wti \tr_{\mathfrak H}[y]\right),\ [x], [y]\in \dom(A^*)/\dom(A),
\end{equation}
where the symplectic form is defied by
\begin{equation}\lb{i11}
\omega_{\mathfrak H}((f_1,g_1), (f_2,g_2)):=\langle f_1, g_2\rangle_{\mathfrak{H}}-\langle g_1,f_2\rangle_{\mathfrak{H}}, (f_k,g_k)\in\mathfrak H\times \mathfrak H,\  k=1,2.
\end{equation}
That is, $\wti \tr_{\mathfrak H}$ is a symplectomorphism of $(\cH,\omega_{\cH})$ onto $(\mathfrak H \times\mathfrak H, \omega_{\mathfrak H})$.
\end{proposition}
\begin{proof}
Combining \eqref{i8} and the fact that $\tr_{\mathfrak H}$ is onto, we infer that $\wti \tr_{\mathfrak H}$ is well defined, one-to-one, onto, and bounded. By the Open Mapping Theorem, $(\wti \tr_{\mathfrak H})^{-1}\in \cB\left(\mathfrak H\times \mathfrak H, \dom(A^*)/\dom(A)\right)$. The abstract second Green identity \eqref{i7} yields \eqref{i10}.  
\end{proof}
We now provide a description of all self-adjoint extensions of $A$ in terms of Lagrangian subspaces of $(\mathfrak H \times\mathfrak H, \omega_{\mathfrak H})$ (which is a consequence of the Lagrangian description via abstract traces acting into the quotient space $\dom(A^*)/\dom(A)$ cf. \cite[Lemma 3.3]{BbF95}).
\begin{corollary}\lb{i1}
Assume Hypothesis \ref{y35} and recall Definition \ref{y33}. Let $\mathscr D\subset \dom(A^*)$, and let $A_{\mathscr D}$ be an operator acting in $\cH$ and given by
\begin{equation}
A_{\mathscr D}u=A^*u,\ u\in\dom(A_{\mathscr D}):=\mathscr D.
\end{equation}
If $A_{\mathscr D}$ is self-adjoint, then 
\begin{equation}\lb{y39}
\tr_{\mathfrak H}(\mathscr D)=\{(\Gamma_1u, \Gamma_2u): u\in\mathscr D\}\subset\mathfrak H\times \mathfrak H,
\end{equation} 
is Lagrangian with respect to symplectic form \eqref{i11}.

Conversely, if $\cG\subset \mathfrak H\times \mathfrak H$ is Lagrangian, then the operator $A_{\tr_{\mathfrak H}^{-1}(\cG)}$ acting in $\cH$ and given by
\begin{equation}\lb{y38}
A_{\tr_{\mathfrak H}^{-1}(\cG)}u=A^*u,\ u\in\dom(A_{\tr_{\mathfrak H}^{-1}(\cG)}):=\tr_{\mathfrak H}^{-1}(\cG),
\end{equation}
is self-adjoint $($where ${\tr_{\mathfrak H}^{-1}(\cG)}$ denotes preimage of set $\cG$$)$.
\end{corollary}
\begin{proof} 
Assume that $A_{\mathscr D}$ is self-adjoint. Then by Lemma 3.3 in \cite{BbF95}, 
\begin{equation}\lb{i12}
[\mathscr D]:=\{[x]: x\in\mathscr D\},
\end{equation}
is Lagrangian in $\dom(A^*)/\dom(A)$ with respect to symplectic form  $\omega_{\cH}$, cf., \eqref{i4}. Since $\wti \tr_{\mathfrak H}$ is symplectomorphism, $\wti \tr_{\mathfrak H}([\mathscr D])$ is Lagrangian plane in $\mathfrak H\times\mathfrak H$ with respect to form $\omega_{\mathfrak H}$, cf., \eqref{i11}. Furthermore,
\begin{equation}\lb{i13}
\wti \tr_{\mathfrak H}([\mathscr D])=\tr_{\mathfrak H}(\mathscr D),
\end{equation}
hence, $\tr_{\mathfrak H}(\mathscr D)$ is also Lagrangian.

Conversely, assume that $\cG$ is Lagrangian in $\mathfrak H\times \mathfrak H$. Then, since $\ker(\tr_{\mathfrak H})=\dom(A)$, one has
\begin{equation}\lb{i14}
A\subset A_{\tr_{\mathfrak H}^{-1}(\cG)}.
\end{equation}
By Proposition \ref{i15}, $\wti \tr_{\mathfrak H}^{-1}(\cG)$ is Lagrangian in $\dom(A^*)/\dom(A)$. Since $\wti \tr_{\mathfrak H}^{-1}(\cG)=[\tr_{\mathfrak H}^{-1}(\cG)]$ (we denote $[\tr_{\mathfrak H}^{-1}(\cG)]=\{[x]: x\in\tr_{\mathfrak H}^{-1}(\cG)\}$), by Lemma 3.3 in \cite{BbF95} the operator $A_{\tr_{\mathfrak H}^{-1}(\cG)}$ is self-adjoint in $\cH$.
\end{proof}%
Next we turn to the Maslov index in the context of self-adjoint, Fredholm extensions of symmetric operators.
\begin{hypothesis}\lb{i17}
Assume that a one-parameter family  $t\mapsto V_t\in \cB(\cH)$ is contained in $C^1([\alpha,\beta], \cB(\cH)),\ \alpha<\beta$, and $V^*_t=V_t,\ t\in[\alpha,\beta]$.  

Assume Hypothesis \ref{y35} and that $\ker(A^*+V_t-\lambda)\cap\dom(A)=\{0\}$ for all $t\in[\alpha,\beta]$, and $\lambda\geq \kappa$ for some  $\kappa<0$.

Assume that $(\mathfrak H,\Gamma_{1,t},\Gamma_{2,t}),$ $t\in[\alpha,\beta],$ is a one-parameter family of boundary triples associated with $A$ such that the family $t\mapsto\tr_{\mathfrak H,t}:=(\Gamma_{1,t},\Gamma_{2,t})$ is contained in $C^1\big([\alpha,\beta],\cB(\dom(A^*),\mathfrak H\times \mathfrak H)\big)$.
\end{hypothesis}
We remark that the second condition in Hypothesis \ref{i17} often holds in case of second order differential operators considered on bounded domain $\Omega\subset\bbR^n$ (and can be viewed as an abstract version of the unique continuation principle). The third condition is natural in the context of geometric deformations of domain $\Omega$ and the corresponding change of variables in conormal derivative.

The following theorem is a corollary of results from \cite{BbF95} and Proposition \ref{i15}, hence we will only sketch the proof.
\begin{theorem}\lb{i18}
Assume Hypotheses \ref{y35} and \ref{i17}. Let $t\mapsto \cG_t$ be one-parameter family containing in $C^1([\alpha,\beta],\Lambda(\mathfrak H\times \mathfrak H))$. Let $A_{\mathscr D_t}, t\in[\alpha,\beta]$ denote the self-adjoint extension of operator $A$ with domain $\tr_{\mathfrak H,t}^{-1}(\cG_t),t\in[\alpha,\beta]$. Assume that $A_{\mathscr D_t}, t\in[\alpha,\beta]$ has compact resolvent and that there exists $\lambda_{\infty}\in[\kappa, 0)$ $($recall $\kappa$ from Hypothesis \ref{i17}$)$ such that 
\begin{equation*}
\ker(A_{\mathscr D_t}+V_t-\lambda)=0\text{\ for all\ }t\in[\alpha,\beta], \lambda<\lambda_{\infty}.
\end{equation*}
Then
\begin{equation}\lb{i19}
\mo\left(A_{\mathscr D_{\alpha}}+V_{\alpha}\right)-\mo\left(A_{\mathscr D_{\beta}}+V_{\beta}\right)=\mi\big((\bbK_{0,t},\cG_t)|_{t\in[\alpha,\beta]}\big),
\end{equation}
where $\bbK_{\lambda,t}$ denotes the traces of the ``strong" solutions of the equation $A^*u+V_tu=\lambda u,\ u\in \dom(A^*)$, that is, 
\begin{equation}\lb{i20}
\bbK_{\lambda,t}:=\tr_{\mathfrak H,t}\big(\ker(A^*+V_t-\lambda)\big), t\in[\alpha,\beta], \lambda\in\bbR.
\end{equation}
\end{theorem}
\begin{proof}
First, using parametrization \eqref{w17w}-\eqref{21} we introduce two loops with values in $\Lambda(\mathfrak H\times \mathfrak H)$ by the formulas
\begin{align}
&\Sigma\ni s\mapsto \bbK_{\lambda(s), t(s)}\in\Lambda(\mathfrak H\times \mathfrak H),\lb{i21}\\
&\Sigma\ni s\mapsto \cG_{t(s)}\in\Lambda(\mathfrak H\times \mathfrak H).\lb{i22}
\end{align}
By \cite[Theorem 3.9]{BbF95}, the one-parameter family $\Sigma\ni s\mapsto\ker(A^*+V_{t(s)}-\lambda(s))/\dom(A)$ is  continuous and contained in  $C^1(\Sigma_k,\Lambda(\dom(A^*)/\dom(A)),\ 1\leq k\leq 4$. That is, there exists a family of orthogonal projections $\Sigma\ni s\mapsto P_s\in\cB\left(\dom(A^*)/\dom(A)\right)$ such that
\begin{align}
&\ran({P_s})=\ker(A^*+V_{t(s)}-\lambda(s))/\dom(A),\\
&P_s\in C^1(\Sigma_k,\cB(\dom(A^*)/\dom(A))),\ 1\leq k\leq 4,\\
&P_s\in C(\Sigma,\cB(\dom(A^*)/\dom(A))).
\end{align}
Then $\Sigma\ni s\mapsto Q_s:=\wti\tr_{\mathfrak H, t(s)}P_s(\wti\tr_{\mathfrak H, t(s)})^{-1}\in\cB(\mathfrak H\times \mathfrak H)$ is a family of bounded projections such that
\begin{align}
&\ran({Q_s})=\tr_{\mathfrak H,t(s)}\big(\ker(A^*+V_{t(s)}-\lambda(s)\big), s\in\Sigma,\\
&Q_s\in C^1(\Sigma_k,\cB(\mathfrak H\times \mathfrak H)),\ 1\leq k\leq 4,\ Q_s\in C(\Sigma,\cB(\mathfrak H\times \mathfrak H)). \lb{i23}
\end{align}
The projection $Q_s$ may not be orthogonal; however it can be ``straighten" while preserving regularity as in \eqref{i23}.  

Second, we observe that $\mi\big((\bbK_{\lambda(s),t(s)},\cG_{t(s)})|_{s\in\Sigma}\big)=0$ by the homotopy invariance of the Maslov index. On the other hand, 
\begin{align}
\begin{split}
&\mi\big((\bbK_{\lambda(s),t(s)},\cG_{t(s)})|_{s\in\Sigma}\big)=\lb{i24}\\
&\quad +\mi\big((\bbK_{\lambda(s),t(s)},\cG_{t(s)})|_{s\in\Sigma_1}\big)+\mi\big((\bbK_{\lambda(s),t(s)},\cG_{t(s)})|_{s\in\Sigma_2}\big)\\
&\quad +\mi\big((\bbK_{\lambda(s),t(s)},\cG_{t(s)})|_{s\in\Sigma_3}\big)+\mi\big((\bbK_{\lambda(s),t(s)},\cG_{t(s)})_{s\in\Sigma_4}\big).
\end{split}
\end{align}
Finally, proceeding as in the proof of Theorem \ref{w33} one can show that the crossings on $\Sigma_1$ are negative definite, the crossings on $\Sigma_3$ are positive definite, and that there are no crossings on $\Sigma_4$. Thus,

\begin{align}
\begin{split}
0&=-\sum\limits_{\lambda_{\infty}<\lambda<0}\dim\left(\ker(A_{\mathscr D_{\alpha}}+V_{\alpha}-\lambda)\right)\lb{i25}\\
&+\mi\big((\bbK_{\lambda(s),t(s)},\cG_{t(s)})|_{s\in\Sigma_2}\big)+\sum\limits_{\lambda_{\infty}<\lambda<0}\dim(\ker(A_{\mathscr D_{\beta}}+V_{\beta}-\lambda)),
\end{split}
\end{align}
as asserted in \eqref{i20}.
\end{proof}

We will now discuss several particular applications of the results of this subsection.
\subsection{Spectra of $\theta-$periodic Schr\"odinger operators on [0,1] and the Maslov index}
The boundary triple technique employed in Theorem \ref{i18} is well suited for ordinary differential operators. Indeed, let
\begin{equation*}
\cT_{min}:=-\partial_x^2,\ \dom(\cT_{min}):=H^2_0[0,1], \cT_{max}:=(\cT_{min})^*,
\end{equation*}
and recall from \cite[Chapter 3]{GG} that 
\begin{equation*}
\cT_{max}:=-\partial_x^2,\  \dom(\cT_{max}):=H^2[0,1].
\end{equation*}
The operator $\cT_{min}$ admits a boundary triple
\begin{align}
\begin{split}
&\mathfrak H=\C^{2m},\ \Gamma_1: H^2[0,1]\rightarrow \C^{2m},  \Gamma_1 u:= (u(1),u(0))^{\top},\lb{rr12}\\
& \Gamma_2: H^2[0,1]\rightarrow \C^{2m},  \Gamma_2 u:= (u'(1),-u'(0))^{\top}.
\end{split}
\end{align}
Next we turn to a self-adjoint extension of $\cT_{\min}$. For each fixed $\theta\in[0,2\pi)$ the operator 
\begin{align}
&(-\partial_x^2)_{\theta}:L^2([0,1], \C^m)\rightarrow L^2([0,1], \C^m),\ (-\partial_x^2)_{\theta} u:=-u'',\ u\in\dom((-\partial_x^2)_{\theta}),\no \\
&\dom((-\partial_x^2)_{\theta}):=\{u\in AC([0,1],\C^m): u'\in AC[0,1], u''\in L^2([0,1],\C^m), \no\\
&\hspace{6.5cm}  u(1)=e^{\bfi \theta}u(0),\ u'(1)=e^{\bfi \theta}u'(0)\},\no
\end{align}
is self-adjoint with compact resolvent. Let $V\in L^{\infty}([0,1], \C^{m\times m}), V=\overline{V}^{\top}$, and denote $\cL_{\theta}:=(-\partial_x^2)_{\theta}+V$. Then $\cL_{\theta}$ is also self-adjoint, has compact resolvent, and 
\begin{equation}
\inf_{\theta\in [0,2\pi) }\min \{\lambda: \lambda\in\spec(H_{\theta})\}>-\infty.
\end{equation}
Let us denote  $\cG_{\theta}:=(\Gamma_1,\Gamma_2)(\dom(\cL_{\theta}))$. Clearly,
\begin{equation}\lb{rr14}
\cG_{\theta}=\{(e^{\bfi \theta}a,a,-e^{\bfi \theta}b,b): a,b\in\C^m\},
\end{equation}
is contained in $C^1([0,\pi], \Lambda(\C^{2m}\times \C^{2m}))$. Hence, the one-parameter family $\cL_{\theta}, \theta\in[0,\pi],$ together with boundary triple  $(\C^{2m}, \Gamma_1, \Gamma_2)$ satisfy hypotheses of Theorem \ref{i18}. Then  
\begin{equation}\lb{rr15}
\mo(\cL_{\theta_1})-\mo(\cL_{\theta_2})=\mi\big((\bbK, \cG_{\theta})|_{\theta\in[\theta_1,\theta_2]}\big), 0\leq \theta_1<\theta_2\leq \pi,
\end{equation}
where 
\begin{equation}\no
\bbK:=\{(u(1),u(0),u'(1),-u'(0))^{\top}: -u''+Vu=0\}\subset\C^{4m}.
\end{equation}
\begin{remark}
We stress that the result concerning equality of the spectral flow and the Maslov index for Sturm-Liouville operators on $[0,1]$ is obtained in full generality in \cite[Theorem 0.4]{BZ4} . In particular, \eqref{rr15} can be alternatively derived using \cite[Theorem 0.4]{BZ4}. The symplectic structure used in \cite[Theorem 0.4]{BZ4} is determined by the first order system of ODE's equivalent to the eigenvalue problem for original Sturm-Liouville operator. In contrast, our symplectic structure is induced by the right-hand side of the Green's formula \eqref{i7} and we do not need to rewrite the eigenvalue problem as the first order ODE. As a result we deal with Lagrangian planes that are symplectomorphic to their counterparts from \cite{BZ4}.

\end{remark}
\subsection{Spectra of self-adjoint Schr\"odinger operators and the Maslov index} In this section we illustrate \eqref{i19} in the context of self-adjoint extensions of $-\Delta_{min}$, cf. \eqref{ab20}, \eqref{ac10}. 
Hypothesis \ref{d2} is assumed throughout this subsection. Let as recall the following two facts from \cite{GM10}: 

1) there exists a unique linear, bounded operator 
\begin{equation}
\gd:\cD_{\Delta}^1(\Omega)\to (N^{1/2}(\dOm))^*,
\end{equation}
which is compatable with the Dirichlet trace $\gamma_D$, cf. \eqref{N12},

2) there exists a unique linear, bounded operator 
\begin{equation}
\gn:\cD_{\Delta}^1(\Omega)\to (N^{3/2}(\dOm))^*,
\end{equation}
which is compatable with the Neumann trace $\gaN$, cf. \eqref{N32}.

The Dirichlet-to-Neumann map $M_{D,N}$ is defined by
\begin{equation}
M_{D,N}: \left\{
\begin{array}{rl}
(N^{1/2}(\dOm))^*&\to(N^{3/2}(\dOm))^* \\
f&\to-\gn(u_D),  
\end{array}
\right.
\end{equation}
where $u_D$ is the unique solution of the boundary value problem 
\begin{equation}
-\Delta u=0\,\,\hbox{in}\,\,\Om,\,\,\,\,u\in L^{2}(\Om),\,\,\gd u=f\,\,\hbox{in}\,\,\dOm.
\end{equation}
Denoting $\tN u:=\gn u+ M_{D,N}(\gd u)$, one has 
\begin{align}\lb{Green}
&(-\Delta u,v)_{L^{2}(\Om)}-(u,-\Delta v)_{L^{2}(\Om)}\no\\
&=\overline{{\lnoh \tN v, \gd u \rnohs}}-\lnoh \tN u, \gd v \rnohs,
\end{align}
for every $u,v\in\dom(-\Delta_{\max})$, cf. \cite{GM10}. In other words, $-\Delta_{min}$ admits the following boundary triple
\begin{equation}
\mathfrak H:= \noh, \Gamma_1:=R^{-1}\hatt{\gamma}_D,\ \Gamma_2:=\tau_N,
\end{equation}
where $R:N^{1/2}(\partial\Omega)\rightarrow(N^{1/2}(\partial\Omega))^*$ is the Riesz duality isomorphism.
With this at hand, the following proposition is a corollary of Theorem \ref{i18}.
\begin{corollary}
Assume Hypothesis \ref{d2}. Let $\mathscr D\subset \cD_{\Delta}^1(\Omega)$ and assume that the operator $-\Delta_{\mathscr D}$ acting in $L^2(\Omega, \C^m), m\in\bbN$ and given by
\begin{equation}\lb{rr16}
-\Delta_{\mathscr D}u=-\Delta u,\ u\in\dom(-\Delta_{\mathscr D}):=\mathscr D.
\end{equation}
is self-adjoint, bounded from below and has compact resolvent. Let 
\begin{equation*}
t\mapsto V_t\in L^{\infty}(\Omega, \C^{m\times m}), V_t(x)=\overline{V_t(x)^{\top}}, x\in\Omega,
\end{equation*}
be a one-parameter family containing in $C^1([\alpha,\beta], L^{\infty}(\Omega, \C^m))$. Then
\begin{equation}
\mo(-\Delta_{\mathscr D}+V_{\alpha})-\mo(-\Delta_{\mathscr D}+V_{\beta})=\mi\big( (K_{0,t}, \cG)|_{t\in[\alpha,\beta]}\big),
\end{equation}
where 
\begin{align}
&K_{0,t}=(R^{-1}\hatt{\gamma}_D, \tN)\big(\{u\in\cD_{\Delta}^1(\Omega): -\Delta_{max}u+V_tu=0\}\big),\ \cG:=(R^{-1}\hatt{\gamma}_D, \tN)(\mathscr D)\no.
\end{align}
\end{corollary}


\end{document}